\documentclass[10pt]{amsart}

\calclayout

\usepackage{amsmath, amssymb}
\usepackage{amsfonts}
\usepackage{mathrsfs}
\usepackage[color, arrow,matrix,curve,cmtip,ps]{xy}
\usepackage{paralist}
\usepackage{quiver}
\usepackage{amsthm}
\usepackage{tikz-cd}
\usetikzlibrary{arrows,calc,matrix}
\tikzset{
curvarr/.style={
  to path={ -- ([xshift=2ex]\tikztostart.east)
    |- (#1) [near end]\tikztonodes
    -| ([xshift=-2ex]\tikztotarget.west)
    -- (\tikztotarget)}
  }
}
\tikzset{%
    symbol/.style={%
        draw=none,
        every to/.append style={%
            edge node={node [sloped, allow upside down, auto=false]{$#1$}}}
    }
}

\usepackage{caption}
\usepackage[final]{pdfpages}
\usepackage{dsfont,txfonts}
\usepackage{tikz}
\usepackage{rotating}
\usepackage{mathrsfs}
\PassOptionsToPackage{hyphens}{url}\usepackage{hyperref}
\usepackage{enumitem}
\usepackage{graphicx}
\usepackage{mathtools}
\usetikzlibrary{arrows,chains,matrix,positioning,scopes}

\usepackage[T1]{fontenc}
\usepackage[variablett]{lmodern}
\usepackage{xcolor}
\usepackage[citestyle=alphabetic,bibstyle=alphabetic]{biblatex}
\usepackage{lipsum}
\usepackage[most]{tcolorbox}

\newtheorem{theorem}{Theorem}[section]
\theoremstyle{definition}
\newtheorem{lemma}[theorem]{Lemma}

\newtheorem{proposition}[theorem]{Proposition}
\newtheorem{corollary}[theorem]{Corollary}
\newtheorem{definition}[theorem]{Definition}

\newtheorem{remark}[theorem]{Remark}

\newtheorem{example}[theorem]{Example}

\everymath{\displaystyle}
\allowdisplaybreaks

\begin{document}
\title{Measure theory via Locales}
\author{Georg Lehner}
\begin{abstract} We present an approach to measure theory using the theory of locales. This includes concrete constructions of measure algebras associated to Radon measures, such as the Lebesgue measure on $\mathbb{R}^n$, via Grothendieck topologies constructed from valuations, that circumvent the classical approach via $\sigma$-algebras. As an application we obtain a functorial construction of the induced measure $\mu_*$ on the locale of sublocales $\mathfrak{Sl}(X)$ of a Hausdorff space $X$ equipped with a Radon measure $\mu$, which in particular shows that $\mu_*$ is invariant under measure-preserving homeomorphisms. We furthermore give a construction of the measurable locale associated to a smooth manifold, functorial in submersions, as well as comparison results to classical measure theory.
\end{abstract}

\maketitle

\begin{figure}[h!]
    \centering
    \includegraphics[width=0.45\textwidth]{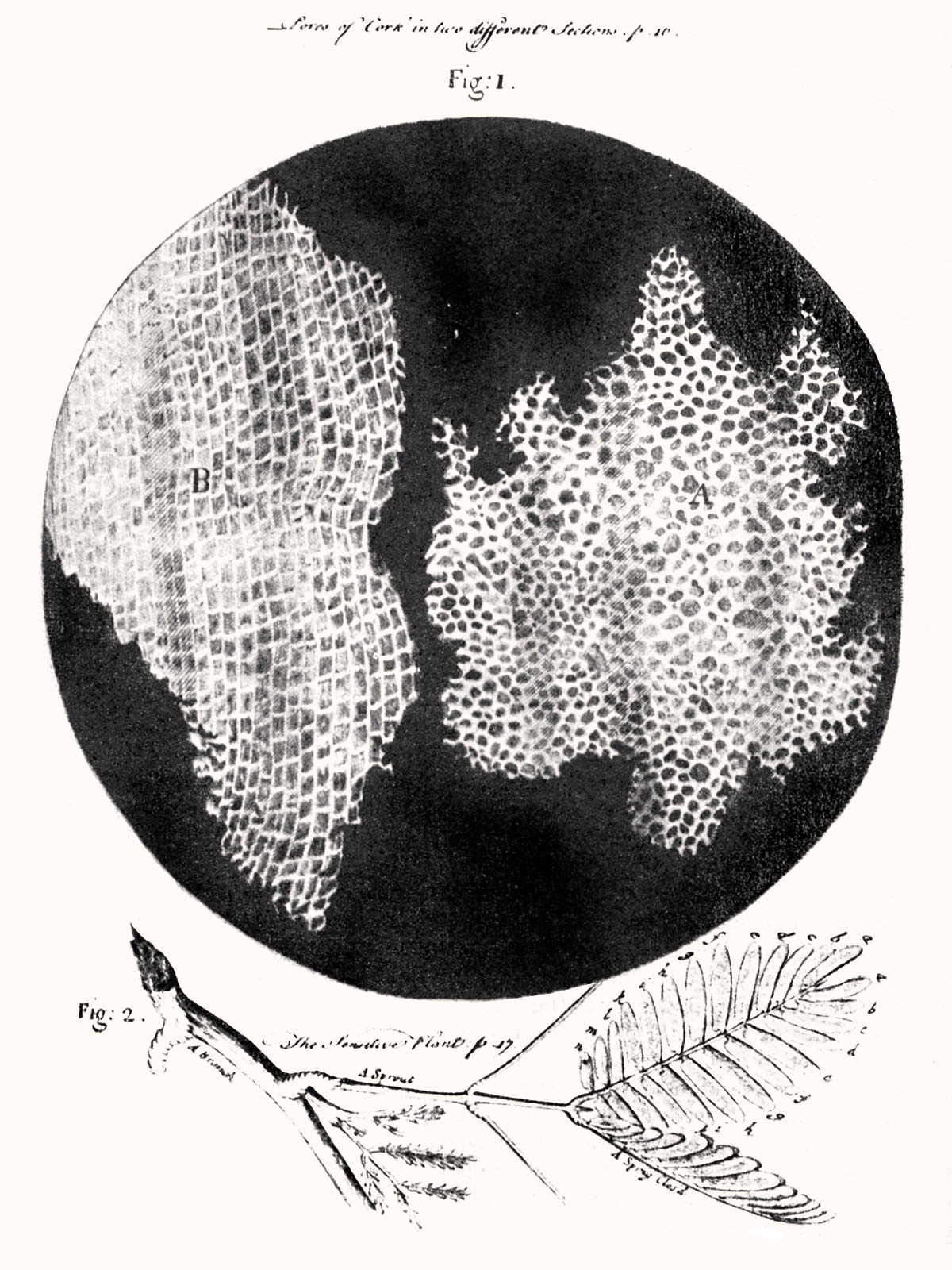}
    \caption*{\tiny Robert Hooke, \textit{Micrographia}, 1665, Public domain, via Wikimedia Commons}
\end{figure}

\tableofcontents

\newpage

\section{Introduction}

A common mode of thought in mathematical circles is that measure and probability theory belong to the domain of analysis, and that categorical or structuralist approaches to the subject miss the point. Consider, for example, the MathOverflow question titled “Is there an introduction to probability theory from a structuralist/categorical perspective?” \cite{20740}, to which the second most upvoted answer, with 133 votes at the time of this writing, is:

\begin{tcolorbox}[colback=gray!10, colframe=gray!30, boxrule=0pt, sharp corners, enhanced]
In the spirit of this answer to a different question, I'll offer a contrarian answer. 
How to understand probability theory from a structuralist perspective:

\medskip
\colorbox{gray!20}{\strut \textbf{Don't}}

\medskip
\noindent To put it less provocatively, what I really mean is that probabilists don't think about 
probability theory that way, which is why they don't write their introductory books 
that way. The reason probabilists don't think that way is that probability theory is 
not about probability spaces. Probability theory is about families of random variables. (Mark Meckes \cite{20828}) 
\end{tcolorbox}

While we think that the sentiment expressed in this answer is well-intentioned - separating formalism from what truly matters is indeed crucial -  we also think that the general willingness to accept the traditional $\sigma$-algebra-based approach, with all its paradoxes and exceptions, as a necessary evil has hindered the development of potential connections between measure theory and areas not closely related to analysis. It has also led to the quiet acceptance of paradoxes such as those due to Vitali and Banach–Tarski. Let us be more specific about what a ``good'' theory of measure should be capable of:
\begin{enumerate}
\item The central formal objects representing the many examples used in mathematical practice should be directly constructible from geometric or combinatorial data, with constructions explicit enough to permit direct arguments.
\item The notion of almost-everywhere equivalence should be built into the formalism so that handling it becomes automatic. Objects that agree classically up to almost-everywhere equivalence should be intrinsically identical.
\item The central formal objects should organise into a category with good formal properties.
\item Results in the classical literature on measure theory should admit a mostly direct translation into this new framework.
\end{enumerate}

Let us remark why we would require points (2) and (3). These are an attempt to address the sentence ``Probability theory is about families of random variables.'' given in the above quote by Meckes. First of all, we want to regard random variables as the same if they are the same up to almost everywhere equivalence. Second, the choice of domain of a random variable should not matter in practice: any precomposition by an isomorphism should be regarded as equivalent. A categorical approach can handle this issue, since it is an intrinsic feature of category theory that using categorical concepts allows one to treat ``isomorphic situations as equivalent''.

It should be said that the attitude towards a structural approach to measure and probability theory has shifted in recent years. We would like to highlight work of authors such as Simpson \cite{SIMPSON20121642}, Vickers \cite{DBLP:journals/mlq/Vickers08}, \cite{vickers_monad}, Henry \cite{HENRY_2017}, Jackson \cite{jackson_sheaves}, Pavlov \cite{PAVLOV2022106884}, Tao \& Jamneshan \cite{jamneshan2022foundationalaspectsuncountablemeasure} and Jamneshan \cite{jamneshan2014conditional}, among others, which for the most part emphasize sheaf-theoretic and point-free perspectives to measure theory, by using locale and topos-theory as a foundation. 

In this paper we will address points (1) to (4) by adopting a locale theoretic approach to measure theory, building on the techniques presented by the aforementioned authors. Locales serve as a replacement for topological spaces. The standard references for locale theory are \cite{johnstone1982stone}, \cite{picado_pultr} and  \cite{vickers_topology}. A \emph{locale} $L$ is specified by its \emph{frame} of abstract open sets $ \mathcal{O}(L) = F$, a poset $(F, \leq)$ that generalizes the lattice of open subsets of a topological space. This poset is required to satisfy:
\begin{itemize}
\item Closure under arbitrary suprema, written as $\bigvee_{i \in i} U_i$ for a collection $U_i \in F, i \in I.$ This includes the existence of a minimal element $0$ (the empty supremum).
\item Closure under finite infima, written as $U_1 \wedge \hdots \wedge U_n$. This includes the existence of a maximal element $1$ (the empty infimum).
\item The distributivity relation
$$V \wedge \bigvee_{i \in I} U_i = \bigvee_{i \in I} V \wedge U_i $$
needs to hold.
\end{itemize}
A frame homomorphism $f^* : F \rightarrow F'$ is a function preserving the order, arbitrary suprema and finite infima. A \emph{sublocale} $S \hookrightarrow L$ of a frame is given by the image of a surjective frame homomorphism $\mathcal{O}(L) \rightarrow \mathcal{O}(S)$.

For comparison, if $X$ is a topological space, then $\mathcal{O}(X)$ forms a frame, and every continuous map $f : X \rightarrow Y$ induces a frame homomorphism $f^{-1} : \mathcal{O}(Y) \rightarrow \mathcal{O}(X)$. However not every frame arises this way. Similarly, every subspace $S \subset X$ of a topological space determines a sublocale of $\mathcal{O}(X)$, but not every sublocale of $\mathcal{O}(X)$ corresponds to an actual subspace.

A \emph{valuation} is a function $\mu : D \rightarrow [0,\infty]$ on a lower bounded distributive lattive $D$, such that
\begin{itemize}
\item $\mu(U) \leq \mu(V)$ whenever $U \leq V \in D$.
\item $\mu(0) = 0$.
\item $\mu(U) + \mu(V) = \mu(U \vee V) + \mu( U \wedge V)$ for all $U, V \in D$.
\end{itemize}
If $D = \mathcal{O}(L)$ for a locale $L$, then we call a valuation $\mu : \mathcal{O}(L) \rightarrow [0,\infty]$ a \emph{measure} on $L$ if it furthermore satisfies
$$ \mu( \bigvee_{i \in I} U_i ) = \sup_{i \in I} \mu(U_i). $$
for every directed system of opens $U_i, i \in I$. Let us give two examples of measure in this locale theoretic sense:
\begin{itemize}
\item  In the case of a Radon measure $\mu$ on a Hausdorff space $X$, such as for example the Lebesgue measure on $\mathbb{R}^d$, the restriction of $\mu$ to the set of opens of $X$ will give a measure in the locale-theoretic sense. 
\item  If $(X, \mathcal{L}, \mu)$ is a \emph{localizable} measure space (See Section \ref{localizablemeasurespace}), then the quotient $\mathcal{L}/\mathcal{N}$, where $\mathcal{N}$ is the ideal of null sets, is an example of a frame. Then $\mu$ descends to a (locale-theoretic) measure on this frame. The corresponding locale $L(\mathcal{L}/\mathcal{N})$ is practically never obtained via an actual topological space.
\end{itemize}

The most striking result in the setting of localic measure theory may be the construction of a translation and rotation invariant measure $\lambda_d$ on the coframe of all sublocales of $\mathbb{R}^d$ proved independently by Leroy \cite{leroy2013theorielamesuredans} and Simpson \cite{SIMPSON20121642}, which ``resolves'' the well-known Banach-Tarski paradox. Any subset $A \subset\mathbb{R}^d$ induces a corresponding sublocale and therefore has an associated value $\lambda_d(A)$ which can rightfully be called the \emph{Lebesgue measure} of $A$. The reason there is no contradiction to the Vitali or Banach-Tarski paradox is that the intersections of the individual sets appearing in a paradoxical decomposition may have no points, but are nonetheless non-trivial sublocales. Among other things, we will generalize this result to the context of arbitrary Radon measures over Hausdorff spaces $X$ in this article.

The salient feature that separates the theory of locales from that of topological spaces is that one is able to use Grothendieck topologies for the construction and description of locales. This works much like the use of presentations of groups via generators and relations. The theory of Grothendieck topologies is well-established in the broader context of sheaf theory. For the sake of having a coherent account in the setting of locales, we collect the main techniques in Section \ref{sitesandgrothendiecktopologies}. A Grothendieck topology consists of the datum of a poset $(P, \leq)$, to be thought of as ``local pieces'', which is analogous to the set of generators of a group, together with a collection $\tau$ of coverings $\{p_i \leq p ~|~ i \in I \}$, which contain the information of when one local piece $p$ is covered by smaller local pieces $p_i$, which are analogous to the set of relations for a group presentation. There is a locale $L = L(P,\tau)$ generated from this datum, for which $P$ functions much like a basis in traditional point-set topology. Abstract open sets of $L$, continuous functions into $L$, as well as measures on $L$ are all determined from constructions purely on $P$, much like how a presentation of a group $G$ means that arbitrary group elements can be representated by words, and group homomorphisms out of $G$ are determined by how they act on generators.\footnote{See Definition \ref{definitionpropositionalsheaf}, Theorem \ref{flatfunctor} and Theorem \ref{valuationbasis}.} There is a recognition principle that allows one to determine when a given locale $L$ arises from a particular Grothendieck topology, called the Basis Theorem (Theorem \ref{basistheorem}). All this means that even if one only cares about an ordinary topological space $X$, it can still be very useful to find a suitable presentation via a Grothendieck topology.

The central example of a Grothendieck topology will be the $\mu$-\emph{inner topology} for the case of a finite valuation $\mu$ on a lower bounded distributive lattice $(D, \leq)$. We call this datum a \emph{valuation site}. The coverings in this case consist of two different cases. (See Definition \ref{muinnertopology}.)
\begin{itemize}
\item Finite unions cover: $\{p_i \leq p ~|~ i \in I\}$ is a cover whenever $$\bigvee_{i \in I} p_i = p$$ for $I$ finite.
\item $\mu$-approximations cover: $\{p_i \leq p ~|~ i \in I\}$  is a cover whenever $p_i, i \in I,$ is directed and $$ \mu(p) = \sup_{i \in I} \mu(p_i).$$
\end{itemize}
The corresponding locale $L(D,\mu) = (D,\mu)^{inn}$ comes equipped with a \emph{locally finite} and \emph{faithful} measure $\mu_*$ (See Definition \ref{definitionmeasure}), which extends $\mu$, and can be thought of as a $\mu$-generic point. The main technical result about this construction is summarized in the following theorem.

\begin{theorem}[See Theorem \ref{innervaluationadjunction}]
There exists an adjunction
\[\begin{tikzcd}
	{\mathrm{ValSite}} & {\mathrm{MeasFrm}}.
	\arrow[""{name=0, anchor=center, inner sep=0}, "{(-)^{inn}}", curve={height=-12pt}, from=1-1, to=1-2]
	\arrow[""{name=1, anchor=center, inner sep=0}, "{(-)^{fin}}", curve={height=-12pt}, from=1-2, to=1-1]
	\arrow["\dashv"{anchor=center, rotate=-90}, draw=none, from=0, to=1]
\end{tikzcd}\]
This adjunction is idempotent, inducing an equivalence of categories
\[\begin{tikzcd}
	{\mathrm{ValSite}_{\sigma , \mathrm{faithf}}} & {\mathrm{MeasFrm}_{\mathrm{l.f.~faithf}}}
	\arrow[""{name=0, anchor=center, inner sep=0}, "{(-)^{inn}}", curve={height=-12pt}, from=1-1, to=1-2]
	\arrow[""{name=1, anchor=center, inner sep=0}, "{(-)^{fin}}", curve={height=-12pt}, from=1-2, to=1-1]
	\arrow["\dashv"{anchor=center, rotate=-90}, draw=none, from=0, to=1]
\end{tikzcd}\]
where
\begin{itemize}
\item ${\mathrm{ValSite}_{\sigma , \mathrm{faithf}}}$ is the full subcategory of the category $\mathrm{ValSite}$ of valuation sites, given by bounded $\sigma$-complete and faithful valuation sites, and
\item ${\mathrm{MeasFrm}_{\mathrm{l.f.~faithf}}}$ is the full subcategory of the category ${\mathrm{MeasFrm}}$ of frames equipped with measures, given by faithful and locally finite measures.
\end{itemize}
\end{theorem}

One of the uses of the $\mu$-inner topology will be a concrete description of measure algebras. (Adressing point (1).) Suppose $(X, \mathcal{B}, \mu)$ is the structure of a Radon measure $\mu$ on a Hausdorff space $X$. The restriction of $\mu$ to the set $\mathcal{K}(X)$ of compact subsets is a finite valuation. 

\begin{theorem}[See Theorem \ref{radonmeasures}]
Let $(X, \mathcal{B}, \mu)$ be a Radon measure on a Hausdorff space $X$. Then we have a measure-preserving isomorphism
$$L(\mathcal{K}(X), \mu) \cong L(\mathcal{B}/\mathcal{N}).$$
where $L(\mathcal{B}/\mathcal{N})$ is the locale associated to the \emph{measure algebra} $\mathcal{B}/\mathcal{N}$ of $\mu$, obtained as the complete Boolean algebra of measurable sets modulo null sets.
\end{theorem}

The notion of Radon measure includes practically all relevant examples of measure spaces for applications in analysis, geometry and probability theory. For these applications, the measure algebra together with its associated measure contains arguably all of the important information needed for measure theory. (Most notably $L^1$ and $L^\infty$, and more generally $L^p$, can be constructed directly from it.) The associated sheaf topos has also been suggested as a natural place for a type of conditional set theory, within which all constructions and arguments are automatically measurable and invariant under almost everywhere equivalence \cite{jamneshan2014conditional}. The relevance of Theorem \ref{radonmeasures} is that for these applications, the entire information of the Boolean algebra $\mathcal{B}$ is superfluous. What is relevant is the datum of $\mu$ defined on the set of compact sets $\mathcal{K}(X)$. Constructions, such as sheaves, geometric morphisms or continuous functions, are entirely determined from geometric constructions on $\mathcal{K}(X)$ that agree with the $\mu$-inner topology (in other words, are insensitive to approximating from below via compacts). While the entire $\sigma$-algebra $\mathcal{B}$ is often a fairly wild and inexplicit mathematical object, when dealing with practical applications this wildness is irrelevant. Philosophically speaking, this should not be too surprising, as for the construction of a Radon measure one typically starts with giving a \emph{content} $\lambda : \mathcal{K}(X) \rightarrow [0,\infty)$ and then uses an extension principle, such as variants of Carathéodory's extension theorem.

The locale $L(\mathcal{K}(X), \mu) \cong L(\mathcal{B}/\mathcal{N})$ has a Boolean algebra of open sets and comes equipped with a locally finite and faithful measure. We call locales with these two properties  \emph{measurable locales}. Given a measurable locale $L$, define $L^\infty(L)$ to be the set of bounded, continuous, complex-valued functions on $L$. There exists a Gelfand-type duality due to Pavlov.

\begin{theorem}[\cite{PAVLOV2022106884}]
The functor $L^\infty$ induces an equivalence
$$L^\infty : \mathrm{MblLoc}^{op} \simeq \mathrm{CVNA},$$
where $\mathrm{MblLoc}$ is the category of measurable locales, and $\mathrm{CVNA}$ is the category of commutative von Neumann algebras and normal $*$-morphisms. The inverse functor sends a commutative von Neumann algebra to its locale of projections.
\end{theorem}

This theorem means that all statements about measurable locales have corresponding functional-analytic counterparts, and vice versa. In particular, a theory of integration for measurable locales is already worked out. Therefore it is worthwhile to give a concrete condition for when a valuation site $(D,\mu)$ produces a measurable locale $L(D,\mu)$. We call this condition \emph{almost Boolean}, see Definition \ref{almostboolean} for details.

\begin{theorem}[See Theorem \ref{almostbooleangivesboolean}]
If $(D,\mu)$ is almost Boolean, then $L(D,\mu)$ is a measurable locale.
\end{theorem}

\noindent The main example is given by a \emph{regular content} $\lambda : \mathcal{K}(X) \rightarrow [0,\infty)$ on a Hausdorff space $X$ (See Section \ref{regularcontents}).

\begin{theorem}[See Theorem \ref{regularcontentalmostboolean}]
Let $X$ be a Hausdorff topological space and $\lambda : \mathcal{K}(X) \rightarrow [0,+\infty)$ a regular content. Then $\lambda$ is an almost Boolean valuation on $\mathcal{K}(X)$.
\end{theorem}

This motivates the development of a theory of measure based purely on the notion of an almost Boolean valuation on the set $\mathcal{K}(X)$ of compact subsets of a Hausdorff space.

\begin{definition}[See \ref{definitionradonvaluation}]
Let $X$ be a Hausdorff space. A \emph{Radon valuation} on $X$ is a finite and almost Boolean valuation $\mu : \mathcal{K}(X) \rightarrow [0, \infty)$ on the set of compact subsets of $X$. Denote by $X^\mu$ the associated inner locale together with its locally finite and faithful measure $\mu_* $. It comes with a natural continuous map $X^\mu \rightarrow X$.
\end{definition}

As discussed before, classical Radon measures give examples of Radon valuations.  The existence of a Radon valuation has many pleasant consequences, which are discussed in Section \ref{radonvaluations}. First all, it equips $X$ with a measure $\mu_*$ via pushforward along $X^\mu \rightarrow X$. Furthermore, recall the notion of a sublocale of a locale $L$. The set of sublocales $\mathrm{Sl}(L)$ naturally forms a co-frame, or equivalently $\mathrm{Sl}(L)^{op}$ is a frame, whose corresponding locale we denote by $\mathfrak{Sl}(L)$, called the \emph{locale of sublocales} of $L$, sometimes also referred to as the \emph{dissolution locale} of $L$. This locale has a natural continuous map $\mathfrak{Sl}(L) \rightarrow L$, whose inverse image part sends an open $U$ to the closed complement of $U$ viewed as a sublocale, and which comes with a universal property regarding a specific class of maps $M \rightarrow L$. (See Theorem \ref{lifttosublocales2}.) This universal property holds in particular for the map $X^\mu \rightarrow X$ and produces a unique lift,
\[\begin{tikzcd}
	& {\mathfrak{Sl}(X)} \\
	{X^\mu}  & X.
	\arrow["{\mathrm{can}}", from=1-2, to=2-2]
	\arrow["{\exists !}", dashed, from=2-1, to=1-2]
	\arrow[from=2-1, to=2-2]
\end{tikzcd}\]
The map $X^\mu \rightarrow \mathfrak{Sl}(X)$ equips the locale of sublocales with a measure, also denoted as $\mu_*$, via pushforward, or equivalently the co-frame $(\mathrm{Sl}(X),\leq)$ with a \emph{co-measure}. Due to the universality of this construction it is straightforward to verify functoriality of this assignment.

\begin{theorem}[See Theorem \ref{functorialitymeasuresublocales}]
There exists a lift of the functor $\mathfrak{Sl} : \mathrm{RadHausSpc}_{glob} \rightarrow \mathrm{Loc}$ to a functor
$$\begin{array}{rcl}
\mathfrak{Sl} : \mathrm{RadHausSpc}_{glob} & \rightarrow & \mathrm{MeasLoc}_{glob} \\
 (X, \mu) & \mapsto & (\mathfrak{Sl}(X), \mu_*)
\end{array}$$
where $\mathrm{MeasLoc}_{glob}$ is the category of locales equipped with measures and globally defined measure-preserving maps between them.
\end{theorem}

\begin{corollary}[See Corollary \ref{invariancemeasuresublocale}]
Let $X$ be a Hausdorff space equipped with a Radon valuation $\mu$. Then there exists a measure $\mu_*$ on the locale of sublocales $\mathfrak{Sl}(X)$ of $X$, which is invariant under measure-preserving homeomorphisms of $X$, and the natural continuous map $\mathfrak{Sl}(X) \rightarrow X$ is measure-preserving.
\end{corollary}

What is fascinating is that this in particular assigns to each subset of $X$ a consistent measure, which is invariant under measure-reserving homeomorphisms, via pushforward along the natural map $X^{disc} \rightarrow  \mathfrak{Sl}(X)$, which in case the Radon valuation $\mu$ is a regular content agrees with the classical induced Radon measure on measurable sets, see Corollary \ref{measureofsubset}. The formula of the measure one obtains for an arbitrary subset $S \subset X$ is given by
$$\mu_*(S) = \sup \{ \mu(K) ~|~ K \subset S \text{ compact} \}.$$ The pushforward however does not need to send a partition of $X$ into disjoint subsets to a partition of $\mathfrak{Sl}(X)$, therefore no contradiction to Vitali's Paradox or the Banach-Tarski Paradox is obtained. Rather, the locale of sublocales ``knows'' when one has chosen a bad decomposition, since there are \emph{more} sublocales than subsets to detect these failures. This result is similar to the results obtained by Leroy \cite{leroy2013theorielamesuredans} and Simpson \cite{SIMPSON20121642}, although the main conceptual difference is that we construct a measure on $\mathrm{Sl}(X)^{op}$ rather than $\mathrm{Sl}(X)$, which results in much better formal properties.

\begin{remark}
This result is maybe the strongest indicator that from the point-of-view of topology the notion of a subspace of a topological space $X$ described by a \emph{set} of points is a concept that should be treated with caution and in general replaced by the notion of a sublocale. While it is fine to define a single particular subspace of a topological space via its subset of points, it is not a good idea to consider the totality of all such subspaces in a geometric or topological context. The notion of subspace really only works well when one restricts to finite unions and intersections of closed and open sets (and occasionally slightly more if one adds additional assumptions on the space $X$). The collection of all subspaces of a space should of course also be viewed \emph{topologically} - In this sense the locale $\mathfrak{Sl}(X)$ is the natural answer to how one should organize all ``topological sub-objects of $X$'' into one space. However, $\mathfrak{Sl}(X)$ is often not spatial, therefore one is forced to accept locale theory as the ambient setting in which to do topology. (Even if one is interested only in concrete topological spaces.)
\end{remark}

\begin{remark}
Compared with classical measure theory, the approach may seem backwards. Classically, for a measure space $(X,\mathcal{L},\mu)$ one first produces the $\sigma$-algebra $\mathcal{L} \subset \mathcal{P}(X)$, and then obtains the measure algebra $\mathcal{L}/\mathcal{N}$ after the fact. In our approach, we construct the measure algebra \emph{first}, and obtain the fact that we can measure all sublocales as a consequence.
\end{remark}

\begin{remark} The locale of sublocales $\mathfrak{Sl}(L)$ has a basis given by locally closed subsets of $X$, i.e.\ intersections of open and closed subsets. It therefore plays a role analogous to the constructible topology on a scheme. Continuous functions on the constructible topology have been used as a replacement for the notion of measurable function in the context of \emph{motivic integration theory} as developed originally by Kontsevich, see e.g.\ \cite{blickle2005shortcoursegeometricmotivic}. Similarly, in the context of classical measure theory, we can think of continuous functions defined on $\mathfrak{Sl}(L)$ as a useful and better behaved replacement of the notion of measurable function on $X$. A working theory of Lebesgue integration of functions defined on the locale of sublocales $\mathfrak{Sl}(L)$ is for example developed by Bernardes in \cite{bernardes_measure_and_integration}.
\end{remark}

Let us present some more results under the assumption that we have a regular content $\lambda$ on a Hausdorff space $X$.

\begin{theorem}[See Theorem \ref{measurablelocaleembedded}]
Let $\lambda$ be a regular content on a Hausdorff space $X$. Then the induced map $\tilde{p} : X^\lambda \rightarrow \mathfrak{Sl}(X)$ identifies $X^\lambda$ with a sublocale of $\mathfrak{Sl}(X)$.
\end{theorem}

A variant of this theorem has already been proven by Leroy \cite[Théoremè 2]{leroy2013theorielamesuredans}. The results are not entirely equivalent, but agree when $X$ is also assumed locally compact.

We give a universal property for the map $X^\lambda \rightarrow X$ as the terminal measurable locale equipped with a locally finite and faithful measure and a measure-preserving map to $X$. This characterization was originally suggested by Simon Henry in the mathoverflow post \cite{486494}.

\begin{theorem}[See Theorem \ref{universalpropertyregularcontent}]
Let $\lambda$ be a regular content on a Hausdorff space $X$ and assume that the induced measure $\lambda_*$ on $X$ is locally finite.\footnote{This condition is automatically satisfied if $X$ is locally compact, see Lemma \ref{localfinitenesslocallycompact}.} Suppose $ f : (Y,\mu) \rightarrow (X, \lambda_*)$ is a measure-preserving continuous map, with $Y$ being Boolean, and $\mu$ a locally finite and faithful measure on $Y$. Then there exists a unique measure-preserving map $ \tilde{f} : (Y,\mu) \rightarrow (X^\lambda, \lambda_*)$ making the triangle
\[\begin{tikzcd}
	Y \\
	{X^\lambda} & X
	\arrow["{\tilde{f}}"', dashed, from=1-1, to=2-1]
	\arrow["f", from=1-1, to=2-2]
	\arrow["p"', from=2-1, to=2-2]
\end{tikzcd}\]
commute.
\end{theorem}

The following is a characterization of the set of points of $X^\lambda$. As a consequence one sees that $X^\lambda$ is rarely ever spatial.

\begin{proposition}[See Propositon \ref{pointsregularcontent}]
Let $X$ be a Hausdorff space, $\lambda$ a regular content and $X^\lambda$ the associated Boolean inner measure locale. Then the natural map $p_\lambda : X^\lambda \rightarrow X$ induces an injection
$$ \mathrm{pts}( X^\lambda )  \hookrightarrow \mathrm{pts}( X )$$
which identifies the left-hand side with the set of point $x$ of $X$ such that $\lambda(\{x\}) > 0$.
\end{proposition}

We call a continuous map $f : X^\mu \rightarrow Y$ into an arbitrary locale $Y$ a $Y$-valued \emph{Random variable}. This terminology fits with the classical definition of a Random variable up to almost everywhere equivalence.

\begin{proposition}[See Proposition \ref{injectivity} and Corollary \ref{densityinjectivity}]
Let $X$ be a Hausdorff space and $\mu$ a Radon valuation. Assume that $\mu_*(U) > 0$ for all non-empty open sets $U \subset X$, or equivalently that for all non-empty open sets $U$ there exists $K \subset U$ compact with $\mu(K) > 0$. Then the map
$$p_\mu : X^\mu \rightarrow X$$
is dense. In particular, for any Hausdorff locale $Y$ it induces an injection
$$ \mathrm{Map}( X, Y ) \rightarrow \mathrm{Map}( X^\mu, Y )$$
of continuous maps into random variables.
\end{proposition}

The special case of $Y = \mathbb{C}$ produces an inclusion $C_b(X;\mathbb{C}) \hookrightarrow L^\infty(X,\mu)$ of bounded continuous complex-valued functions on $X$ into $L^\infty(X,\mu)$.

We also prove a representation theorem for measurable locales in Section \ref{sectionrepresentationtheorem}.

\begin{theorem}[See Theorem \ref{representationtheorem}]
Let $L$ be a measurable locale. Then there exists a locally compact Hausdorff space $X$ together with a regular content $\lambda$ such that $L \cong X^\lambda$.
\end{theorem}

Following this, we give a short discussion on how one would construct the classical Lebesgue measure on $\mathbb{R}^d$ as well as more generally the Haar measure on a locally compact Hausdorff group directly in Section \ref{examples}. This is meant to illustrate how results in the literature on classical measure theory can be imported into the localic setup in an efficient manner.

As a last application, we provide a construction of the \emph{Lebesgue locale} $M_{Leb}$ of a smooth manifold $M$, which is the localic counterpart to the $C^*$-algebra $L^{\infty}(M)$ (which can be constructed without reference to one particular measure). This construction has some pleasant functoriality and provides a basic object on which to set up integration theory on smooth manifolds.

\begin{theorem}[See Theorem \ref{lebesguelocale}]
There exists a functor
$$\begin{array}{rcl}
(-)_{Leb} : \mathrm{Man}_{\mathrm{subm}} & \rightarrow &  \mathrm{MblLoc} \\
 M & \mapsto & M_{Leb}
\end{array}$$
where $\mathrm{Man}_{\mathrm{subm}}$ is the category of large smooth manifolds and smooth submersions as maps between them.
\end{theorem}

\begin{remark}Given the framework provided in this article, one can give a simple logical definition of a probability space, which covers virtually all practical applications:
\begin{definition}
A \emph{probability system} consists of the datum of a complete Boolean algebra $\mathcal{E}$  of \emph{events} together with a \emph{faithful measure} $\mathbb{P}$ such that $\mathbb{P}(1) =1$, in other words an assignment $\mathbb{P} : \mathcal{E} \rightarrow [0,1]$, such that
$$\begin{array}{l}
\mathbb{P}(E) = 0 \text{ iff } E = 0 \\
\mathbb{P}(1) = 1 \\
\mathbb{P}(E_1) + \mathbb{P}(E_2) = \mathbb{P}(E_1 \vee E_2 ) + \mathbb{P}(E_1 \wedge E_2 ) \\
\text{Whenever } \{E_i \leq E ~|~ i \in I\} \text{ is a directed system: } \mathbb{P}( E ) = \sup_{i \in I} \mathbb{P}(E_i) \text{ iff } E = \bigvee_{i \in I} E_i \\ 
\end{array}$$
\end{definition}
The practical meaning of this definition is that once a class of ``basic events'' that generates $\mathcal{E}$ under suprema and whose probabilities are known has been identified, working within the algebra of events boils down entirely to the estimation of probabilities. This set of basic events can in praxis always be modelled via an almost Boolean valuation site. Viewing $\mathcal{E}$ as presenting a locale, a random variable with values in some locale $L$ is just a continuous function $\mathcal{E} \rightarrow L$. It equips $L$ with a probability measure via pushforward (In the case of $\mathbb{R}$-valued random variables: The \emph{distribution} of the random variable). The definition of an ergodic action of a group $G$ is also straightforward. An action of a group $G$ on $\mathcal{E}$ is \emph{ergodic}, if the quotient $\mathcal{E}_G$ (taken in the category of locales) is a point. The corresponding topos of sheaves on $\mathcal{E}$ is the topos of conditional set theory over $\mathcal{E}$.
\end{remark}

\begin{remark} There exist many close analogies between the study of algebraic $K$-theory and measure theory, as already elaborated upon to some extent in the article \cite{lehner2025algebraicktheorycoherentspaces}. One of the conceptual advantages of the localic definition of a measure is that it highlights these similarities much more closely than the classical definition of a measure on a $\sigma$-algebra does. To demonstrate, compare the following three definitions.
\begin{itemize}
\item A measure on a locale $L$ is a functor $\mu : \mathcal{O}(L) \rightarrow [0, \infty]$ such that
\begin{itemize}
\item $\mu(0) = 0$,
\item $\mu(U) + \mu(V) = \mu(U \vee V) + \mu( U \cap V)$, and
\item $\sup_{i \in I} \mu(U_i) = \mu( \bigvee_{i \in I} U_i )$ for every directed set $U_i, i \in I$.
\end{itemize}
\item Let $\mathcal{C}$ be a cocomplete 1- or $\infty$-category. A $\mathcal{C}$-valued \emph{cosheaf} on a locale $L$ is a functor $\mathcal{F} : \mathcal{O}(L) \rightarrow \mathcal{C}$ such that
\begin{itemize}
\item $\mathcal{F}(0) \cong 0$,
\item for every $U,V$ the square
\[\begin{tikzcd}
	{\mathcal{F}(U\wedge V)} & {\mathcal{F}(V)} \\
	{\mathcal{F}(U)} & {\mathcal{F}(U\vee V)}
	\arrow[from=1-1, to=1-2]
	\arrow[from=1-1, to=2-1]
	\arrow[from=1-2, to=2-2]
	\arrow[from=2-1, to=2-2]
\end{tikzcd}\]
is a pushout, and
\item $\mathrm{colim}_{i \in I} \mathcal{F}(U_i) =  \mathcal{F}( \bigvee_{i \in I} U_i )$ for every directed set $U_i, i \in I$.
\end{itemize}
\item A \emph{finitary localizing invariant} with values in a cocomplete stable $\infty$-category $\mathcal{E}$ is a functor $F : \mathrm{Cat}^{\mathrm{perf}} \rightarrow \mathcal{E}$ such that\footnote{See e.g.\ \cite{hebestreit2023localisationtheoremmathrmktheorystable}.}
\begin{itemize}
\item $F(0) \cong 0$,
\item for every Verdier square
\[\begin{tikzcd}
	{A} & {B} \\
	{C} & {D}
	\arrow[from=1-1, to=1-2]
	\arrow[from=1-1, to=2-1]
	\arrow[from=1-2, to=2-2]
	\arrow[from=2-1, to=2-2]
\end{tikzcd}\]
the square
\[\begin{tikzcd}
	{F(A)} & {F(B)} \\
	{F(C)} & {F(D)}
	\arrow[from=1-1, to=1-2]
	\arrow[from=1-1, to=2-1]
	\arrow[from=1-2, to=2-2]
	\arrow[from=2-1, to=2-2]
\end{tikzcd}\]
is a pushout, and
\item $\mathrm{colim}_{i \in I} F(A_i) \cong  F( \mathrm{colim}_{i \in I} A_i )$ for every filtered diagram $A_\bullet : I \rightarrow \mathrm{Cat}^{\mathrm{perf}}$. 
\end{itemize}
We remark that algebraic $K$-theory and topological Hochschild Homology are examples of finitary localizing invariants.
\end{itemize} 
The formal similarity should be clear: All three definitions are instances of combining the inclusion-exclusion principle with the principle of exhaustion from below. Other examples can be found as well throughout mathematics, e.g.\ the notion of a homology theory to name another one.
\end{remark}

%
%
%
%

\begin{remark}
The use of Hausdorff topological spaces may seem odd from the point of view of developing ``point-free'' versions of measure theory. Ultimately, this was a concession to keep in line with classical literature, as the theory of regular contents and Radon measures on Hausdorff spaces is already quite well developed, and regular contents are used to construct basically all relevant examples of measure spaces in mathematical practice. Using our present approach, all of these examples can be imported into the locale theoretic context in a fairly pain-free manner.

An approach that is less compromising with respect to classical literature, but more in line with the locale theoretic spirit is taken by Leroy \cite{leroy2013theorielamesuredans}, who instead studies locally finite measures on \emph{regular} locales. Regularity is a separation condition that is related to Hausdorffness, and works as a reasonable replacement. It seems likely to us that the theory presented in this article could be phrased in this language as well, by reinterpreting the poset of compact subsets $\mathcal{K}(X)$ of a Hausdorff space $X$ as (the opposite of) the poset of Scott open filters of a regular locale $L$. We have ommitted this approach, as it might alienate readers that are not as familiar with the workings of locale theory, and the conceptual shift present in this article is quite substantial as it is already.

In the case of locally compact Hausdorff spaces both approaches essentially become equivalent. The notion of a locally compact Hausdorff space can be developed entirely within locale theory using locally compact and completely regular locales. Any locally compact locale is automatically spatial (in the presence of the axiom of choice) and the two categories are equivalent, see e.g.\ \cite[Section 5.3 and 5.4]{picado_pultr}. Locally finite measures on a locally compact Hausdorff space $X$ are in bijection with regular contents on $X$ (See Theorem \ref{regularcontentslocallycompact}), therefore our theory simply adds a different angle to the theory presented in Leroy's work when working with locally compact Hausdorff spaces.
\end{remark}

\begin{remark} The present article is not complete. For example, the theory of integration, one of the main motivations of measure theory, is absent. Along with this there is no discussion of product measures, or of measures on function spaces. There is also no discussion of group actions on measure locales and ergodic theory. At least in principle, one should be able to readily import these topics from classical literature into the locale theoretic setup. This is something to be addressed in future work.
\end{remark}

\begin{remark}
Since there is a community of people interested in constructive approaches to measure theory (this interest is visible for example in the work of Vickers \cite{DBLP:journals/mlq/Vickers08}, Coquand-Spitters \cite{Coquand_Spitters_2009}, Spitters \cite{Spitters2005ConstructiveAI} among others) we should mention that this paper has been written assuming classical foundations. The use of choice as well as proof by contradiction is present, most notably in the comparison of the three different definitions of a measurable locale (Theorem \ref{allequivalent}), as well as (unsurprisingly) when comparing the classical measure theory via $\sigma$-algebras to the approach via the $\mu$-inner Grothendieck topology in Section \ref{classicalcomparison}. Nonetheless, we believe the usage of Grothendieck topologies to describe Radon measures can be a useful starting tool for constructive measure theory; something to be worked out by someone more experienced in the matter of constructive mathematics than the author of this article.
\end{remark}


\subsection{Acknowledgements}

The author thanks Dmitri Pavlov and Simon Henry for helpful feedback on mathoverflow to the question \cite{486494} that inspired this article, as well as David Kern, Christian Espíndola and Asgar Jamneshan for pointers and explanations. I also want to thank Holger Reich and Chris Huggle for enduring my many ramblings on the topic of this paper during our lunch breaks.

\section{Locale theory}

We will use the language of frames and locales throughout this paper. For general resources, see \cite{johnstone1982stone}, \cite{picado_pultr} and \cite{vickers_topology}. Suppose $(P, \leq)$ is a poset and $p, q \in P$. We call an element $p \wedge q$ a \emph{meet} if the condition
$$r \leq p \wedge q \text{ iff } r \leq p \text{ and } r \leq q$$
holds for all $r \in P$. Meets, if they exist, are unique. Of course, one can think of the poset $(P, \leq)$ as a special kind of category, a viewpoint that we will often abuse in this paper. Using this perspective, meets are just products.

The \emph{join} $p \vee q$ is defined dually as the meet of $p$ and $q$ in the opposite poset $P^{op}$. A \emph{distributive lattice} $(D,\leq)$ is a poset that admits joins, written as $U \vee V$ and meets, written $U \wedge V$, and the relation
$$U \wedge (V \vee W) = (U \wedge V) \vee (U \wedge W)$$
holds, for any $U,V,W \in D$. A distributive lattice is called \emph{lower bounded} if it has a bottom element $0$, and furthermore \emph{bounded} if it also has a top element $1$.

We refer to monotone maps $f : P \rightarrow P'$ as \emph{functors}. A functor $f : D \rightarrow D'$ between distributive lattices is called a \emph{lattice homomorphism} if $f$ preserves meets and joins. We denote the categories of lower bounded, respectively bounded, distributive lattices by
$$ \mathrm{DLatt}_{lb} ~\text{ and }~ \mathrm{DLatt}_{bd},$$
where morphisms are lattice homomorphisms that preserve the bottom element, respectively the bottom and top element.

More generally, for any poset $P$ one can define \emph{infima}, respectively \emph{suprema}, of families of elements $p_i \in P, i \in I$ analogously to how one defines meets and joins. We will denote these as $\bigwedge_{ i \in I} p_i$ for the infimum, and $\bigvee_{i \in I} p_i$ for the supremum. Note that top elements and meets are special cases of infima, whereas bottom elements and joins are special cases of suprema. Thinking in terms of categories, infima are just limits, and suprema are just colimits. We call a poset $(P,\leq)$ a \emph{complete lattice} if it is closed under arbitrary suprema, which is the case iff it is closed under arbitrary infima.

A \emph{frame} is a poset $(F, \leq )$ which admits arbitrary suprema and finite infima and satisfies the infinite distributivity law
$$ U \wedge \bigvee_{i \in I} V_i = \bigvee_{i \in I} U \wedge V_i$$
for all $U \in F$ and families $V_i \in F, i \in I$. We note that in particular any frame is a complete distributive lattice. An important example of a frame is given by the poset $( \mathcal{O}(X), \subset )$ of open sets of a topological space $X$. We highlight the following, slightly nonstandard, definition.

\begin{definition}
A functor $f^* : F \rightarrow F'$ between frames is called a \emph{partial frame homomorphism}, if $f^*$ preserves arbitrary suprema and (binary) meets. The element $f^*(1) \in F'$ is called the \emph{domain} of $f^*$. A partial frame homomorphism $f^*$ is called \emph{globally defined}, or simply a \emph{frame homomorphism}, if $f^*(1) = 1$.
\end{definition}

This defines the category $\mathrm{Frm}_{\mathrm{part}}$ of frames and partial frame homomorphisms, as well as the category $\mathrm{Frm}$ of frames and frame homomorphisms. As an example, a partially defined continuous map $f : X \rightarrow Y$ with open support $U$ gives a partial frame homomorphism $f^{-1} : \mathcal{O}(Y) \rightarrow \mathcal{O}(X)$, which is globally defined iff $U = X$. This reversal of direction motivates one to define the category of locales as the opposite category $\mathrm{Loc} = \mathrm{Frm}^{op}$. This way assigning to a topological space $X$ its frame $\mathcal{O}(X)$ produces a \emph{covariant} functor $\mathrm{Top} \rightarrow \mathrm{Loc}$, which is an equivalence when restricted to the full subcategories of sober spaces and spatial locales \cite[II.4.6]{picado_pultr}. By abuse of notation, for a general locale $L$ we use the notation $\mathcal{O}(L)$ for the corresponding frame of $L$ and refer to elements $U \in \mathcal{O}(L)$ as opens of $L$. Conversely, for a given frame $F$ we write $L(F)$ for the corresponding locale. We refer to arrows $f : L \rightarrow L'$ of locales as \emph{continuous maps} or also \emph{geometric morphisms}, which are of course just dually defined as frame homomorphisms $f^* : \mathcal{O}(L') \rightarrow \mathcal{O}(L)$. Similarly, we define the category $\mathrm{Loc}_{\mathrm{part}} = \mathrm{Frm}_{\mathrm{part}}^{op}$ of locales and partial continuous maps.

Soberness of a topological space is a rather mild separation condition. It is implied by Hausdorffness, and implies $T_0$, but is independent of $T_1$. We will not have any use for topological spaces which are not sober in this paper. On the other hand, the use of non-spatial locales is absolutely crucial for our purposes. 

We note that any partial frame homomorphism $f^* : F \rightarrow F'$ has an associated right adjoint $f_* : F' \rightarrow F$, defined via the characterization of adjoints, i.e.\ for all $U \in F, V \in F'$ we have
$$ f^*(U) \leq V \text{ iff } U \leq f_*(V). $$
The existence of $f_*$ follows either abstractly from the adjoint functor theorem, or can be seen concretely via the formula
$$f_*( V ) = \bigvee_{ U \text{ s.t. } f^*(U) \leq V } U.$$

\begin{remark}
A perhaps unusual feature to readers versed in classical point-set topology or also locale theory is that the category $\mathrm{Loc}_{\mathrm{part}}$ is a \emph{pointed category}, with the role of the zero object being taken by the locale of the empty set $\emptyset$. If $L, L'$ are locales, then the zero map $ 0 : L \rightarrow L'$ is given by the functor $0^* : \mathcal{O}(L') \rightarrow  \mathcal{O}(L), U \mapsto 0$.
\end{remark}

We call a globally defined map $f : L \rightarrow M$ a \emph{quotient map} if $f^*$ is injective. Quotient maps are epimorphisms in the category of locales \cite[\nopp IV.1]{picado_pultr}. Similarly, we call $f : L \rightarrow M$ an \emph{embedding} if $f_*$ is injective. Embeddings are equivalently \emph{regular} monomorphisms in the category of locales. Subsets $S \subset \mathcal{O}(M)$ that give rise to embeddings are called \emph{sublocales}. More on this in the upcoming Section \ref{sectionsublocales}.

Despite the fact that locale theory is often called point-free topology, every locale $L$ comes with a set of points, see e.g.\ \cite[\nopp II.3]{picado_pultr}. Let $\mathrm{pt}$ be the locale corresponding to the topological space given by a single point. We define $\mathrm{pts}(L) = \mathrm{Map}_{\mathrm{Loc}}(\mathrm{pt}, L)$. It can be checked that a point of a locale is equivalently described by a \emph{completely prime filter} $\mathcal{F} \subset \mathcal{O}(L)$, that is a subset $\mathcal{F} \subset \mathcal{O}(L)$ such that
\begin{itemize}
\item $\mathcal{F}$ is proper: $0 \not\in \mathcal{F}$.
\item $\mathcal{F}$ is upward closed: If $U \in \mathcal{F}, U \leq V$ then also $V \in \mathcal{F}$.
\item $\mathcal{F}$ is closed under binary meets: If $U, V \in \mathcal{F}$ then also $U \wedge V \in \mathcal{F}$.
\item $\mathcal{F}$ is completely prime: If $\bigvee_{i \in I} U_i \in \mathcal{F}$ then there exists $i \in I$ such that $U_i \in \mathcal{F}$.
\end{itemize}

\subsection{Heyting implication} \label{Heytingimplication}

Now let $L$ be a locale, with associated frame $F$. A somewhat under-appreciated feature in classical point-set topology is that in any frame, there is a notion of \emph{Heyting implication} $U \rightarrow V$ between opens $U,V$, defined by the adjunction\footnote{This is of course just a special case of the definition of the internal hom in a cartesian closed category.}
\[\begin{tikzcd}
	F & F
	\arrow[""{name=0, anchor=center, inner sep=0}, "{U \wedge -}", curve={height=-6pt}, from=1-1, to=1-2]
	\arrow[""{name=1, anchor=center, inner sep=0}, "{U \rightarrow - }", curve={height=-6pt}, from=1-2, to=1-1]
	\arrow["\dashv"{anchor=center, rotate=-90}, draw=none, from=0, to=1]
\end{tikzcd}\]
Concretely $U \rightarrow V$ can be computed as
	$$ U \rightarrow V = \bigvee_{ W \text{ s.t. } W \wedge U \leq V } W $$
The following are standard properties of Heyting implication.
\begin{proposition} \label{superheytingprop}
Let $F$ be a frame and $U, V, W \in F$. The following statements hold.
\begin{enumerate}
\item $(U \rightarrow V) \wedge U \leq V$
\item $1 \rightarrow U = U$ 
\item $U \rightarrow V = 1$ iff $U \leq V$.
\item $U \rightarrow ( V \rightarrow W ) = ( U \wedge V ) \rightarrow W$.
\item $V \leq ( U \rightarrow V )$.
\item $U \leq ( U \rightarrow V ) \rightarrow V$.
\item $(((U \rightarrow V) \rightarrow V ) \rightarrow V ) = U \rightarrow V$.
\item The assignment $- \rightarrow V : F^{op} \rightarrow F$ sends suprema to infima.
\end{enumerate}
\end{proposition}

\begin{proof}
\begin{enumerate}
\item We note that this is simply the counit of the adjunction, in this context also called evaluation. We compute it from the formula
$$(U \rightarrow V) \wedge U = \left( \bigvee_{ W \text{ s.t. } W \wedge U \leq V } W \right)\wedge U = \bigvee_{ W \text{ s.t. } W \wedge U \leq V } W  \wedge U  \leq V. $$
\item We have $U \leq (1 \rightarrow U)$ by using the adjunction on the equality $ 1 \wedge U = U$. The converse corresponds to the counit of the adjunction $(1 \rightarrow V ) \wedge 1 \leq V$.
\item Note that $(U \rightarrow V) \leq 1$ is always true. The converse is by the adjunction
$$1 \leq (U \rightarrow V) \text{ iff } U = 1 \wedge U \leq V. $$
\item Let $Z \in F$ be open. Then
$$ Z \leq U \rightarrow ( V \rightarrow W ) \text{ iff } Z \wedge U \leq (V \rightarrow W)  \text{ iff } Z \wedge U \wedge V \leq W \text{ iff } Z \leq ( U \wedge V ) \rightarrow W.$$
The statement follows by the Yoneda lemma.
\item This follows by applying the adjunction to $ U \wedge V \leq V$.
\item This follows by applying the adjunction to the counit
$$ U \wedge (U \rightarrow V) \leq V.$$
\item Applying (5) we immediately get 
$$U \rightarrow V \leq (((U \rightarrow V) \rightarrow V ) \rightarrow V )$$
as a special case. However, we can also apply the contravariant functor $(-) \rightarrow V$ to (5) to obtain  
$(((U \rightarrow V) \rightarrow V ) \rightarrow V ) \leq U \rightarrow V$.
\item One verifies that $(-) \rightarrow V$ is always right adjoint to itself (viewed as a functor $F \rightarrow F^{op}$).
\end{enumerate}
\end{proof}

\begin{proposition} \label{heytingpropmaps}
Let $f : L \rightarrow M$ be a partial continuous map. Let $U$ be an open of $L$ and $V, Z$ opens of $M$. Then:
\begin{enumerate}
\item $f_*( f^*(V)  \rightarrow U ) = V  \rightarrow  f_*(U)$, and
\item $f^*( V \rightarrow Z ) \leq f^*(V) \rightarrow f^*(Z)$.
\end{enumerate}
\end{proposition}

\begin{proof}
Consider the square of commuting left adjoints:
\[\begin{tikzcd}
	\mathcal{O}(L) & \mathcal{O}(L) \\
	\mathcal{O}(M) & \mathcal{O}(M)
	\arrow["{- \wedge f^*(V)}", from=1-1, to=1-2]
	\arrow["{f^*}", from=2-1, to=1-1]
	\arrow["{- \wedge V}"', from=2-1, to=2-2]
	\arrow["{f^*}"', from=2-2, to=1-2]
\end{tikzcd}\]
The first equation is obtained by passing to right adjoints, whereas the second is the mate of this square. (Alternatively verified directly by applying $f^*$ to (1) in Proposition \ref{superheytingprop} and adjoining over). 
\end{proof}

\subsection{Sublocales} \label{sectionsublocales}

Given a locale $L$ with corresponding frame $F$ and a subset $F' \subset F$, one may wonder when $F'$ itself defines a locale. If this happens, and the inclusion becomes the right adjoint of a frame homomorphism, we call the associated locale $L'$ a sublocale of $L$.

\begin{definition}
Let $i : L' \rightarrow L$ be a continuous map of locales. We call $i$ \emph{an embedding}, if $i_* : \mathcal{O}(L') \rightarrow \mathcal{O}(L)$ is injective.
\end{definition}

\begin{lemma}[\cite{picado_pultr}, III.1 \& 2, \cite{johnstone1982stone} II.2]
Let $i : L' \rightarrow L$ be a continuous map of locales. The following are equivalent:
\begin{itemize}
\item $i$ is an embedding.
\item The left adjoint $i^* : \mathcal{O}(L) \rightarrow \mathcal{O}(L')$ is surjective.
\item $i$ is an extremal monomorphism in the category of locales.
\item $i$ is a regular monomorphism in the category of locales.
\item $i^* i_* = \mathrm{id}_{\mathcal{O}(L')}$.
\end{itemize}
\end{lemma}

\begin{definition}
Let $L$ be a locale with frame $F$. A \emph{sublocale} of $L$ is defined as a subset $S \subset F$ such that $S$ is a frame in its induced order and the inclusion $i_* : S \hookrightarrow F$ is the right adjoint to a frame homomorphism $i^* : F \rightarrow S$.
\end{definition}

We note that any sublocale of $L$ gives an embedding $i : L(S) \rightarrow L$ and any embedding $f : L' \rightarrow L$ identifies $\mathcal{O}(L')$ with a sublocale of $L$.

\begin{example}[\cite{picado_pultr}, III.6]
Let $L$ be a locale and $U$ an open of $L$. We get the induced \emph{open sublocale} associated to $U$ by analyzing the adjunction
\[\begin{tikzcd}
	{\mathcal{O}(L)_{/U}} & \mathcal{O}(L)
	\arrow[""{name=0, anchor=center, inner sep=0}, "{i_!}", curve={height=-12pt}, from=1-1, to=1-2]
	\arrow[""{name=1, anchor=center, inner sep=0}, "{i_*}"', curve={height=12pt}, from=1-1, to=1-2]
	\arrow[""{name=2, anchor=center, inner sep=0}, "{i^*}"{description}, from=1-2, to=1-1]
	\arrow["\dashv"{anchor=center, rotate=-89}, draw=none, from=0, to=2]
	\arrow["\dashv"{anchor=center, rotate=-91}, draw=none, from=2, to=1]
\end{tikzcd}\]
The three functors involved are given as:
\begin{itemize}
\item $i_!(V) = V$ for any $V \leq U$. This functor is a partial frame homomorphism, with $\mathrm{dom}(i_!) = U$.
\item $i^*(W) = W \wedge U$. This functor is a frame homomorphism.
\item $i_*(V) = U \rightarrow V$.
\end{itemize}
It may be a bit unusual to a reader new to locale theory that the natural identification of the open sublocale associated to an open $U$ is via the functor $i_* = U \rightarrow -$, and not via the more easily understood forget functor $i_!$. We will often identify the open $U$ with the sublocale ${\mathcal{O}(L)_{/U}}$ in notation.
\end{example}

\begin{remark}
The above example justifies our naming of the notion of partial continuous maps. Let $f^* : F \rightarrow F'$ be a partial frame homomorphism. Then $f^*$ factors as
$$f^* : F \xrightarrow{f^\#} F'_{/f^*(1)} \subset F'.$$
Writing $\mathrm{dom}(f) = f^*(1)$, this corresponds to a geometric factorization
$$L(F') \supset \mathrm{dom}(f) \xrightarrow{f} L(F).$$
One could therefore identify the category $\mathrm{Loc}_{\mathrm{part}}$ with a certain subcategory of the category $\mathrm{Span}(\mathrm{Loc})$ of locales and morphisms given by spans, as is common in the setting of $6$-functor formalisms. We won't investigate this perspective any further in this paper.
\end{remark}

\begin{example} \label{closedsublocale}
Let $L$ be a locale and $U$ an open of $L$. We get the induced \emph{closed sublocale} associated to $U^c$ from the adjunction
\[\begin{tikzcd}
	{\mathcal{O}(L)_{U/}} & \mathcal{O}(L)
	\arrow[""{name=0, anchor=center, inner sep=0}, "{i_*}"', curve={height=6pt}, hook, from=1-1, to=1-2]
	\arrow[""{name=1, anchor=center, inner sep=0}, "{i^*}"', curve={height=6pt}, from=1-2, to=1-1]
	\arrow["\dashv"{anchor=center, rotate=-90}, draw=none, from=1, to=0]
\end{tikzcd}\]
The two functors involved are given as:
\begin{itemize}
\item $i^*(W) = W \vee U$. This functor is a frame homomorphism.
\item $i_*(V) = V$ for $V$ such that $U \leq V$.
\end{itemize}
\end{example}

\begin{remark}
Viewing $i_! : \mathcal{O}(U) \rightarrow \mathcal{O}(L)$ as a partial frame homomorphism, or dually as a partial map of locales $L \rightarrow U$ (the partial map that is the identity map on the open domain $U$), we see that the composite
$$ U^c \hookrightarrow L \rightarrow U$$
is the zero map. This sequence, which is in fact a fiber sequence, corresponds to the classical open-closed decomposition of $L$.
\end{remark}

\begin{example} \label{examplesubspace}
Let $X$ be a topological space, and $i : S \subset X$ a subset, equipped with the subspace topology. Then the frame homomorphism $ i^{-1} = S \cap - : \Omega(X) \rightarrow \Omega(S)$ gives an embedding between  associated locales. However, not every sublocale of the associated locale to $X$ arises this way, even if $X$ is sober. In other words, sublocales of spatial locales need not be spatial.
\end{example}

Sublocales can in fact also be characterized via the Heyting implication.

\begin{lemma}[\cite{picado_pultr}, III.2.2] \label{sublocales}
Let $S \subset F$ be a subset of a frame $F$. Then $S$ determines a sublocale iff
\begin{enumerate}
\item $S \subset F$ is closed under infima.
\item $S \subset F$ is closed under Heyting implication, i.e.\ if $U \in S$ and $V \in F$ then $V \rightarrow U \in S$.
\end{enumerate}
\end{lemma}

We finish this section by discussing the universal property of a sublocale, which is a very special case of the notion of a Bousfield localization.

\begin{proposition} \label{sublocaleuniversal}
Let $L$ be a locale with associated frame $F$, and $i_* : S \hookrightarrow F$ be a sublocale. Let $P$ be a poset and $f : F \rightarrow P$ be a functor. Then there exists a functor $g : S \rightarrow P$ such that
\[\begin{tikzcd}
	F & P \\
	S
	\arrow["f", from=1-1, to=1-2]
	\arrow["{i^*}"', from=1-1, to=2-1]
	\arrow["g"', from=2-1, to=1-2]
\end{tikzcd}\]
commutes iff $f = f i_* i^*$, in which case $g = f i_*$. Furthermore, given this situation:
\begin{itemize}
\item If $P$ has arbitrary suprema, and $f$ preserves them, then so does $g$.
\item If $P$ has a top element, and $f(1) = 1$, then so does $g$.
\item If $P$ has binary meets, and $f$ preserves them, then the same holds for $g$.
\end{itemize}
\end{proposition}

\begin{remark}
In particular, a (partial) continuous map $f : M \rightarrow L$ factors through an embedding $i : L' \rightarrow L$ iff $f = f i_* i^*$.
\end{remark}

\begin{proof}
Let $g: S \rightarrow P$ be such that $g i^* = f$. Then $g = g i^* i_* = f i_*$, and $f i_* i^* = g i^* = f$. Conversely, assume $f = f i_* i^*$. Setting $g = f i_*$ we simply verify that $g i^* = f i_* i^* = f$.

Now assume that $P$ has arbitrary suprema and $f$ preserves them. Then
$$g( \bigvee_{i \in I} U_i ) = g( i^*( \bigvee_{i \in I} i_* U_i )) = f( \bigvee_{i \in I} i_* U_i ) = \bigvee_{i \in I} f i_*( U_i ) = \bigvee_{i \in I} g( U_i ).$$

Lastly, the statements about binary meets and top elements follow immediately since $g = f i_*$ is then a composite of functors that preserves them.
\end{proof}

We obtain a simple universal property for closed sublocales.

\begin{corollary} \label{closedsublocaleuniversal}
Let $L, M$ be locales, $V$ an open of $M$, and $f : L \rightarrow M$ a partial continuous map. The following are equivalent:
\begin{enumerate}
\item $f$ factors through the closed embedding $V^c \hookrightarrow M$.
\item $f^*(V) = 0$.
\item $V \leq f_*(0)$.
\end{enumerate}
More generally, assume $U$ is an open of $L$. Then the following are equivalent:
\begin{enumerate}[label=(\arabic*\ensuremath{^{\prime}})]
\item $f$ induces a (unique) map $ f| : U^c \hookrightarrow V^c$.
\item $f^*(V) \leq  U$.
\item $V \leq f_*(U)$.
\end{enumerate}
If this is the case, then $f|^* = f^*(-) \vee U : \mathcal{O}(M)_{V / } \rightarrow \mathcal{O}(L)_{U / }$.
\end{corollary}

\begin{proof} Clearly, the first statement is a special case of the second, by setting $U=0$. (Note that $0 \leq f^*(V)$ is always satisfied.) The equivalence of (2) and (3) follows immediately from the adjunction:
$$V \leq f_*(U) \text{ iff } f^*(V) \leq U.$$
Now consider the composite $\mathcal{O}(M) \xrightarrow{f^*} \mathcal{O}(L) \xrightarrow{- \vee U} \mathcal{O}(L)_{U / }$. This functor descends along the functor $\mathcal{O}(M) \xrightarrow{ - \vee V} \mathcal{O}(M)_{V / }$ iff
$f^*( - \vee V ) \vee U = f^*(-) \vee f^*(V) \vee U = f^*( - ) \vee U.$
By evaluating at $0$ and using that $f^*(0)=0$, we see that this is the case iff $f^*(V) \vee U = U$, which is again equivalent to the statement $f^*(V) \leq  U$.
\end{proof}

Let $f : L \rightarrow M$ be a continuous map of locales. Since in any category, regular monomorphisms are closed under pullback along another map, we can define for any sublocale $S \hookrightarrow M$ a pre-image sublocale $f_{-1}[S] \hookrightarrow L$ via the pullback
\[\begin{tikzcd}
	{f_{-1}[S]} & L \\
	S & M.
	\arrow[hook, from=1-1, to=1-2]
	\arrow[from=1-1, to=2-1]
	\arrow["\lrcorner"{anchor=center, pos=0.125}, draw=none, from=1-1, to=2-2]
	\arrow["f", from=1-2, to=2-2]
	\arrow[hook, from=2-1, to=2-2]
\end{tikzcd}\]
in the category of locales. This pullback always exists, as the category of locales is closed under limits  \cite[Chapter IV.4]{picado_pultr}, and is given as the largest sublocale of $L$ contained in $f^{-1}(S) \subset \mathcal{O}(L)$ \cite[III 4.2]{picado_pultr}. Taking pre-images behaves well with respect to open and closed inclusions.
	
\begin{lemma}[\cite{johnstone1982stone}, II.2.8] \label{preimageofopen}
Let $f : L \rightarrow M$ be a continuous map of locales and $U$ an open of $M$. Then $f_{-1}[U]$ is the open sublocale given by $f^*(U)$, and $f_{-1}[U^c]$ is the closed sublocale given by $f^*(U)^c$.
\end{lemma}

\subsection{The locale of sublocales}

In classical point-set topology, for a topological space $X$, the natural continuous map $X^{disc} \rightarrow X$ corresponds to the inclusion $\mathcal{O}(X) \subset \mathcal{P}(X)$. This allows one to talk about subspaces of a space, determined by subsets. The locale-theoretic analogue of this is given by replacing the powerset $\mathcal{P}(X)$ for a space $X$ by the set $\mathrm{Sl}(L)$ of sublocales of a locale $L$.

\begin{theorem}[\cite{picado_pultr}, Chapter III, 3.2.1 and 6.4.1.] 
Let $L$ be a locale. The poset $\mathrm{Sl}(L)$ of sublocales, ordered by inclusion, is a co-frame, i.e.\ $\mathrm{Sl}(L)^{op}$ is a frame. Furthermore, there is a natural injective frame homomorphism
$$\begin{array}{rcl}
\nabla : \mathcal{O}(L) & \rightarrow & \mathrm{Sl}(L)^{op} \\
U &\mapsto & U^c
\end{array}$$
\end{theorem}

\begin{remark}
The reader may wonder why the natural morphism $\nabla$ sends an open $U$ to its closed complement, and not to $U$ itself viewed as a sublocale. In general for any frame $F$, the co-frame $F^{op}$ can be thought of as representing closed subsets. Since $\mathrm{Sl}(L)$ is a co-frame, we may think of it as representing \emph{closed} subsets of a (not-necessarily discrete!) locale analogous to $X^{disc}$. This reversal of direction is of course a formality, but can be confusing in practice.
\end{remark}

Infima and suprema in $\mathrm{Sl}(L)$ have concrete descriptions, see \cite[Ch. III 3.2]{picado_pultr}. Let $S_i \hookrightarrow L, i \in I,$ be a collection of sublocales of $L$.
\begin{itemize}
\item The infimum $\bigwedge_{i \in I} S_i$ is given by the intersection $\bigcap_{i \in I} S_i$.
\item For the supremum we have the description
$$ \bigvee_{ i \in I} S_i = \{ \bigwedge A | A \subset \bigcup_{ i \in I } S_i \}.$$
\end{itemize}
Moreover, the closed and open sublocales play a special role. Let $i : S \hookrightarrow L$ be a sublocale. Then 
$$S = \bigwedge \{ U^c \vee V ~|~ i_* i^*(U) = i_* i^*(V) \}.$$
In other words, joins of open and closed sublocales generate the entire co-frame of sublocales under infima.

\begin{definition}
Write $\mathfrak{Sl}(L)$ for the locale corresponding to the frame $\mathrm{Sl}(L)^{op}$ and $\mathrm{can} : \mathfrak{Sl}(L) \rightarrow L$ for the natural quotient map given by $\nabla$. We call $\mathfrak{Sl}(L)$ the \emph{locale of sublocales} of $L$, or sometimes also the \emph{dissolution locale}. (In reference to \cite{ISBELL199163}, where it was originally defined.)
\end{definition}

\begin{theorem}[\cite{picado_pultr} Ch. IV, Corollary 6.3.2.] \label{lifttosublocales1}
The assignment $L \mapsto \mathfrak{Sl}(L)$ is a functor
$$\mathfrak{Sl} : \mathrm{Loc} \rightarrow \mathrm{Loc}$$
and $\mathrm{can} : \mathfrak{Sl} \rightarrow \mathrm{id}$ is a natural transformation, that is valuewise a quotient map. A map of locales $f : L \rightarrow M$ is sent to the frame homomorphism $f_{-1}[-] :\mathrm{Sl}(M)^{op} \rightarrow \mathrm{Sl}(L)^{op}$.
\end{theorem}

\begin{remark} \label{remarksubspaces}
The map $\mathrm{can} : \mathfrak{Sl}(L) \rightarrow L$ comes with a universal property which will be expanded upon later in Theorem \ref{lifttosublocales2}. We just mention here that this universal property in particular implies that for $X$ a topological space, we get a unique lift against $\mathrm{can}$ as such
\[\begin{tikzcd}
	& {\mathfrak{Sl}(X)} \\
	{X^{disc}} & X.
	\arrow["{\mathrm{can}}", from=1-2, to=2-2]
	\arrow["{\exists! \varphi}", from=2-1, to=1-2]
	\arrow[from=2-1, to=2-2]
\end{tikzcd}\]
This can be interpreted either as a supremum-preserving functor $\varphi_\# : \mathcal{P}(X) \rightarrow \mathrm{Sl}(X)$, which associates to each subspace of $X$ (viewed as a closed subspace of $X^{disc}$) its corresponding sublocale, as defined in Example \ref{examplesubspace}, or dually as an infimum-preserving functor $\varphi_* : \mathcal{P}(X) \rightarrow \mathrm{Sl}(X)^{op}$, sending each subspace of $X$ (viewed as an open subspace of $X^{disc}$) to the open complement of its corresponding sublocale. We note that $\varphi_\#$ (respectively $\varphi_*$) will not preserve infima (respectively suprema) in general. This should be viewed as a feature and not a bug: As an example, the subspaces $\mathbb{Q}$ and $\mathbb{R} \setminus \mathbb{Q}$ of $\mathbb{R}$ have non-trivial (even dense) intersection when viewed as sublocales of $\mathbb{R}$. This is the reason why Banach-Tarski-like paradoxical decompositions fail to be decompositions in the locale theoretic setting.
\end{remark}

\subsection{Open maps}

The notion of an open map of locales generalizes that of an open map of topological spaces and will be useful later on.

\begin{definition}
A map of locales $f : L \rightarrow M$ is called \emph{open}, if $f^* : \mathcal{O}(M) \rightarrow \mathcal{O}(L)$ has a further left adjoint $f_!$ (called \emph{direct image} of $f$) and the adjunction $f_! \dashv f^*$ satisfies the \emph{Frobenius identity}
$$f_!( U \wedge f^*(V) ) = f_!(U) \wedge V $$
for all $U$ open in $L$ and $V$ open in $M$. 
\end{definition}

Open maps $f : L \rightarrow M$ can be equivalently characterized by the property that the direct image of an open sublocale of $L$ under $f$ is again an open sublocale, see \cite[\nopp III 7.2]{picado_pultr}. Important will be the following equivalent characterization.

\begin{proposition}
A map of locales $f : L \rightarrow M$ is open iff $f^*$ preserves arbitrary infima and implications.
\end{proposition}

\begin{proof}
By the adjoint functor theorem, $f_!$ exists iff $f^*$ preserves infima. Let us prove that $f^*$ preserving implications is equivalent to the Frobenius identity. Assume that the Frobenius identity holds. Let $U$ open in $L$, and $V,W$ open in $M$. We have the chain of equivalent statements:
$$\begin{array}{rl}
U \leq f^*( V \rightarrow W ) & \text{iff} \\
f_!(U) \leq V \rightarrow W  & \text{iff} \\
f_!(U) \wedge V \leq W  & \text{iff, by Frobenius} \\
f_!(U \wedge f^*(V) ) \leq W  & \text{iff} \\
U \wedge f^*(V)  \leq f^*(W)  & \text{iff} \\
U \leq f^*(V) \rightarrow f^*(W).
\end{array}$$
Applying the Yoneda Lemma to the variable $U$, we obtain that $f^*$ preserves implications. The reader is invited to verify that applying the Yoneda Lemma to $W$ instead gives the converse.
\end{proof}

\begin{lemma} \label{closedsublocaleuniversalopen}
Let $L, M$ be locales, $U$ an open of $L$, $V$ an open of $M$, and $f : L \rightarrow M$ an open map. Then $f$ induces an open map $f| : U^c \hookrightarrow V^c$ iff $U = f^*(V)$; or equivalently both $f_!(U) \leq V$ and $f^*(V) \leq U$ hold.
\end{lemma}

\begin{proof}
Assume $f$ is open, and $f$ induces a map $f| : U^c \hookrightarrow V^c$, which by Lemma \ref{closedsublocaleuniversal} happens iff $f^*(V) \leq U$. Since infima in over-posets are computed underlying, we see that
$$f^*(-) \vee U : \mathcal{O}(M)_{V / } \rightarrow \mathcal{O}(L)_{U / }$$
has a left adjoint. It can be identified directly with
$$f_!(-) \vee V : \mathcal{O}(L)_{U / } \rightarrow \mathcal{O}(M)_{V / }.$$
For the Frobenius identity, assume $U \leq W$, and $V \leq Z$. We need the two terms
$$\begin{array}{rcl}
f_!( W \wedge (f^*(Z) \vee U) ) \vee V &=&  f_!(( W \wedge f^*(Z)) \vee U ) \vee V \\
= f_!( W \wedge f^*(Z) ) \vee f_!(U) \vee V  &=& (f_!(W) \wedge Z) \vee f_!(U) \vee V 
\end{array}$$
and
$$\begin{array}{rcl}
(f_!( W ) \vee V ) \wedge Z &=&  (f_!( W ) \wedge Z) \vee V 
\end{array}$$
to agree. But, when evaluating at $W = 0$ we see that this is the case iff $f_!(U) \vee V = V$, i.e.\ $f_!(U) \leq V$.
Furthermore, using the adjunction this is equivalent to $U \leq f^*(V)$.
\end{proof}

Let us package the statement from Lemma \ref{closedsublocaleuniversalopen} slightly differently.

\begin{definition} \label{markedlocales}
We call a pair $(L,U)$, where $L$ is a locale and $U$ is an open of $L$, a \emph{marked} locale. If $(L, U), (M,V)$ are two marked locales, then a continuous map $f : L \rightarrow M$ will be called \emph{compatible}, if $f^*(V) = U$. We denote by $\mathrm{MarkLoc}_{\mathrm{open}}$ the category of marked locales and \emph{open} compatible maps.
\end{definition}

\begin{theorem} \label{markingadjunction}
There exists an adjunction
\[\begin{tikzcd}
	{\mathrm{MarkLoc}_{\mathrm{open}}} & {\mathrm{Loc}_{\mathrm{open}}}
	\arrow[""{name=0, anchor=center, inner sep=0}, "(-)^c"', curve={height=12pt}, from=1-1, to=1-2]
	\arrow[""{name=1, anchor=center, inner sep=0}, "{(-,0)}"', curve={height=12pt}, hook', from=1-2, to=1-1]
	\arrow["\dashv"{anchor=center, rotate=-90}, draw=none, from=1, to=0]
\end{tikzcd}\]
with the left adjoint $L \mapsto (L,0)$ obtained by equipping a locale with the marking given by the bottom element $0$, and the right adjoint being given by $(L,U) \mapsto U^c \hookrightarrow L$. Furthermore, the left adjoint is fully faithful.
\end{theorem}

\begin{proof}
Fully faithfulness of the left adjoint $(-,0)$ is automatic, since any frame homomorphism preserves $0$. The statement that the given functors are adjoint is just a reformulation of Lemma \ref{closedsublocaleuniversalopen}. 
\end{proof}

\section{Sites and Grothendieck topologies} \label{sitesandgrothendiecktopologies}

\subsection{Sites on posets}

Locales can be thought of as $0$-topoi, and just like topoi arise naturally from Grothendieck topologies. In this section we will discuss the notion of Grothendieck pretopologies and sites, and generalities around them, in the context of posets. This material is fairly standard and will be familiar to people coming from topos theory, see e.g.\ \cite[Chapter 2]{johnstone1982stone}; however, we have included many details for potential readers from other areas, who might be unfamiliar with it. This section will also set up consistent notation for what is to follow. We do not claim any originality in this section. The basic gist is that in order to define a locale, it often suffices to give a description of only some basic building blocks, maybe also thought of as ``local pieces'', together with a description of when one such piece is covered from smaller sub-pieces. This is similar to the notion of a basis in topology, but much more flexible since we do not require the existence of a set of points to start with.

\begin{definition}
Let $(P,\leq)$ be a poset with binary meets. A \emph{Grothendieck pretopology} $\tau$ on $P$ is for each $p \in P$ a collection of subsets of $P_{/p}$ called \emph{coverings}. We use the notation $\{p_i \leq p ~|~i \in I\}$ for such a subset of $P_{/p}$. Moreover, the collection of all coverings needs to satisfy the conditions:
\begin{itemize}
\item \emph{Identities:} For all $p \in P$, the set $\{ p \leq p \}$ is a covering.
\item \emph{Stability under base change:} If $\{p_i \leq p ~|~i \in I\}$ is a covering and $q \leq p$, then $\{p_i \wedge q \leq q ~|~i \in I\}$ is a covering.
\item \emph{Locality:} If $\{p_i \leq p ~|~i \in I\}$ is a covering and $\{p_{ij} \leq p_i ~|~j \in J_i\}$ is a covering for each $i \in I$, then $\{p_{ij} \leq p ~|~i \in I,j \in J_i\}$ is a covering.
\end{itemize}
We call $(P,\leq)$ together with a Grothendieck pretopology $\tau$ a \emph{locally cartesian 0-site.}
\end{definition}

From now on we will write $(P, \tau)$ for a poset $P$ together with its Grothendieck topology $\tau$. Let us understand what the information of $\tau$ actually is. First of all, we think of the elements $p \in P$ as \emph{elementary propositions}. We have a notion of implication given by the order $p \leq q$. The requirement that $P$ is closed under meets means that we are able to form logical ``and'' of elementary propositions. The Grothendieck topology controls what we mean by ``or''. To make sense of this, let us define what a general proposition is. Denote by $\mathbf{2} = (\{0,1\},\leq)$ the poset of Boolean truth values. We think of $0$ as false and $1$ as true.

\begin{definition} \label{definitionpropositionalsheaf} Let $(P, \tau)$ be a locally cartesian $0$-site. A \emph{proposition}, also called \emph{propositional sheaf}, is a functor $F : P^{op} \rightarrow \mathbf{2}$ such that for any cover $\{U_i\}$ of $U$ we have\footnote{Note that the ``only if'' direction in the definition of a propositional sheaf is automatic since $F$ is a functor, so $F(U) \leq F(U_i)$ for all $i \in I$.}
$$ F(U) = 1 \text{ iff } F(U_i) = 1 \text{ for all }i \in I.$$
The set of propositional sheaves forms a sub-poset of $\mathrm{Fun}(P^{op},\mathbf{2})$, which we denote by $\mathrm{Sh}(P, \tau; \mathbf{2})$.
\end{definition}

Any functor $F : P^{op} \rightarrow \mathbf{2}$ is determined by the set of objects that are mapped to $1$. In this sense we can give an equivalent definition.

\begin{definition} A $\tau$-ideal is a subset $U \subset P$ such that:
\begin{itemize}
\item $U$ is downward closed.
\item Whenever $\{p_i\}$ is a cover of $p$, and for all $i \in I$ we have $p_i \in U$, then also $p \in U$.
\end{itemize}
Write $\tau\text-\mathrm{Idl}(P)$ for the set of $\tau$-ideals.
\end{definition}

\begin{proposition}
Let $(P, \tau)$ be a locally cartesian $0$-site. There is a bijection
$$\begin{array}{rcl}
\mathrm{Sh}(P, \tau; \mathbf{2}) & \cong & \tau\text-\mathrm{Idl}(P) \\
F & \mapsto & F^{-1}(1).
\end{array}$$
The inverse sends a $\tau$-ideal $U$ to the functor $F_U : P^{op} \rightarrow \mathbf{2}$ uniquely determined by the condition $F_U(p) = 1$ iff $p \in U$.
\end{proposition}

The following is a standard fact in locale theory.

\begin{proposition} \label{generatedframe}
Let $(P, \tau)$ be a locally cartesian $0$-site. The poset of propositional sheaves $\mathrm{Sh}(P, \tau; \mathbf{2})$ is a frame, and the inclusion 
$$ \mathrm{Sh}(P, \tau; \mathbf{2}) \hookrightarrow \mathrm{Fun}(P^{op}, \mathbf{2})$$
has a left adjoint \emph{sheafification} functor, which is a frame homomorphism; in other words $\mathrm{Sh}(P, \tau; \mathbf{2})$ defines a sublocale of $L(\mathrm{Fun}(P^{op}, \mathbf{2}))$. Identifying the right-hand side with the set of downward closed sets of $P$, the sheafification of a downward closed set $V \subset P$ can be computed as the $\tau$-ideal
$$V^{sh} = \{p \in P ~|~ \exists \text{ covering } \{p_i \leq p ~|~ i \in I \} \text{ s.t. } p_i \in V \text{ for all } i \in I \}.$$
\end{proposition}

It is straightforward to verify that $V \rightarrow V^{sh}$ is left adjoint to the inclusion of propositional sheaves into presheaves. For the argument that it also preserves finite meets, see \cite[Lemma 2.4]{lehner2025algebraicktheorycoherentspaces}. We will denote the corresponding locale by $L(P,\tau)$.

\begin{example} \label{finitarytopology}
Let $D$ be a lower bounded distributive lattice. A central and canonical example is given by the \emph{finite join} pretopology $fin$ on $D$, which is defined by taking the coverings
$\{d_i \leq d ~|~i \in I\}$
whenever $I$ is finite and $\bigvee_{i \in I} d_i = d$. In this case, a $fin$-ideal is a set $U \subset D$ such that
\begin{itemize}
\item $0 \in U$,
\item $U$ is downward closed, and
\item whenever $d_1, d_2 \in U$, then also $d_1 \vee d_2 \in U$.
\end{itemize}
In other words, what is normally called an ideal of a lattice. We have canonical isomorphisms
$$ \mathcal{O}(L(D,fin)) = \mathrm{Idl}(D) \cong \mathrm{Fun}^{lex}( D^{op}, \mathbf{2} ) \cong  \mathrm{Ind}(D),$$
where $\mathrm{Ind}(D)$ is the filtered colimit completion of $D$, and $\mathrm{Fun}^{lex}( D^{op}, \mathbf{2} ) \subset \mathrm{Fun}( D^{op}, \mathbf{2} )$ refers to the set of finite limit (i.e.\ finite meet) preserving functors. The locale $L(D,fin)$ is a \emph{locally coherent} locale, automatically spatial and locally compact, with its set of compact open subsets given by $D$. For more information, see \cite[Section 3]{lehner2025algebraicktheorycoherentspaces}.
\end{example}

\begin{example} \label{canonicaltopology}
Let $F$ be a frame. We can equip it with the \emph{canonical pretopology} given by $\{U_i \leq U ~|~i \in I\}$ whenever $\bigvee_{i \in I} U_i = U$.
\end{example}

For any poset $P$, there is the Yoneda embedding
$$\begin{array}{rcl}
y : P &\rightarrow& \mathrm{Fun}(P^{op}, \mathbf{2}) \\
p &\mapsto& \mathrm{Hom}_P( - , p )
\end{array}$$
Interpreted as downward closed sets, this corresponds to the assignment $p \mapsto p \hspace{-0.5ex} \downarrow ~= \{ q \in P ~|~ q \leq p \}$. We can compose this with sheafification to get a functor
$$\begin{array}{rcl}
[-] : P & \rightarrow & \mathrm{Sh}(P, \tau; \mathbf{2}) \\
p & \mapsto & [p] = (y_p)^{sh}
\end{array}$$
This functor preserves any finite meets that exist in $P$. We call propositional sheaves of the form $[p]$ \emph{elementary propositions}. We remark that for a general Grothendieck pretopology, the representable functor $y_p$ need not be a sheaf, and sheafification is usually necessary. Moreover, whilst the Yoneda functor $y : P \rightarrow \mathrm{Fun}(P^{op}, \mathbf{2})$ is always \emph{fully faithful} (by which we mean injective), the same does not need to hold anymore for $[-]$.

\begin{definition}
Let $(P, \tau)$ be a locally cartesian $0$-site. The Grothendieck pretopology $\tau$ is called \emph{subcanonical} if the representable functors $y_p$ are propositional sheaves for all $p \in P$.
\end{definition}

In case of a subcanonical Grothendieck pretopology, the functor $[-] : P \rightarrow \mathrm{Sh}(P, \tau; \mathbf{2})$ is fully faithful. As an example, the Grothendieck topology $(D,fin)$ discussed in Example \ref{finitarytopology} is subcanonical. 

\begin{lemma} \label{subcanonicalfaithful}
Let $(P, \tau)$ be a locally cartesian $0$-site. Then $\tau$ is subcanonical iff $[-] : P \rightarrow  \mathrm{Sh}(P, \tau; \mathbf{2})$ is fully faithful.
\end{lemma}

\begin{proof}
The Yoneda embedding is fully faithful, so if $\tau$ is subcanonical, then $[-] = y$ is fully faithful. Conversely, assume $[-]$ is fully faithful. We then have for $p, q \in P$ a natural equality
$$ \mathrm{Hom}_P(p,q) = \mathrm{Hom}_\mathrm{Sh}([p],[q]) = \mathrm{Hom}_\mathrm{PSh}(y_p,[q]) = [q](p).$$
In other words, the functors $y_p$ and $[p]$ agree.
\end{proof}

It is tautological that for any downward closed set $V \subset P$ we have $V = \bigcup_{p \in V} p \hspace{-0.5ex} \downarrow$. Applying the colimit preserving sheafification functor, we obtain the following.

\begin{lemma}[Co-Yoneda Lemma for propositional sheaves] \label{coyonedalemma}
Let $(P, \tau)$ be a locally cartesian $0$-site and let $U$ be a propositional sheaf. Then
$$ U = \bigvee_{p \in U} [p] $$
in $\mathrm{Sh}(P, \tau; \mathbf{2})$.
\end{lemma}

This means that the frame $\mathrm{Sh}(P, \tau; \mathbf{2})$ is generated under suprema (= disjunctions) by elementary propositions. This lemma should not be surprising to readers familiar with topos theory, as it is the analog to the statement that any sheaf is given as a colimit of (sheafifications of) representables.

We have seen that (locally cartesian) $0$-sites present locales. We are now asking the converse question: If we're already given a locale $L$, how can we determine a (hopefully simpler) $0$-site which determines $L$?

\begin{definition}
Let $F$ be a frame. We call a subset $P \subset F$ a \emph{basis} for $F$ if $P$ is closed under binary meets and for any open $U \in F$ there exist $p_i, i \in I$ such that $U = \bigvee_{i \in I} p_i$.
\end{definition}

It is not hard to check that if $P$ is a basis, the canonical pretopology of $F$ restricts to a pretopology $\tau$ on $P$. In the special case when $F$ is given as $\mathcal{O}(X)$ for a topological space $X$, the definition of a basis reduces to the classical definition of a basis. Given a basis $P \subset F$ and $U \in F$ denote by $h_U : P^{op} \rightarrow \mathbf{2}$ the propositional sheaf corresponding to the $\tau$-ideal $\{p \in P ~|~ p \leq U \}$. The following theorem is saying that knowledge of the basis $P$ together with its topology determines the locale.

\begin{theorem}[Basis theorem]  \label{basistheorem}
Let $F$ be a frame and suppose $P \subset F$ is a basis. Then there is an equivalence
$$\begin{array}{rcl} F &\cong& \mathrm{Sh}(P, \tau; \mathbf{2}) \\ U &\mapsto& h_U \end{array}.$$
\end{theorem}

Variations of this theorem in the context of topoi are well known, where it is also referred to as the Covering Lemma or ``Lemme de comparaison'' \cite[III Théorème 4.1]{SGA4}. For a proof in the context of posets as is stated here, see \cite[Proposition 2.5.]{lehner2025algebraicktheorycoherentspaces}.

\begin{example} \label{lowerreals} The locale of lower reals $\overrightarrow{[0,+\infty)}$ is defined via the following site.
\begin{itemize}
\item Elementary propositions are given by the poset of positive rationals, $(\mathbb{Q}_{\geq 0}, \leq)$, to be thought of as representing the set $[0, a)$ for $a \in \mathbb{Q}_{\geq 0}$. The symbol $[0,0)$ is interpreted as the empty set, or equivalently, the initial object.
\item A rational number $q$ is covered by $q_i$ iff $\sup_{i \in I} q_i = q$.
\end{itemize}
The opens of the locale $\overrightarrow{[0,+\infty)}$ can be identified with the set $[0, +\infty]$, as for a real number $a$, we have the corresponding downward-closed set $[0, a] \subset \mathbb{Q}_{\geq 0}$ given as
$$[0,a] = \{ q \in \mathbb{Q}_{\geq 0} ~|~ q \leq a \}.$$

The locale $\overrightarrow{[0,+\infty)}$ plays a central role in analysis, as maps into it correspond to lower semicontinuous functions. It is locally compact, but not Hausdorff. Sheaves on this locale with value in an $\infty$-category appear naturally as well, e.g.\ in \cite{efimov2025ktheorylocalizinginvariantslarge} the $\infty$-category of such sheaves with values in an $\infty$-category $\mathcal{C}$ is referred to as $\mathrm{Sh}_{\geq 0}( \mathbb{R}; \mathcal{C})$, and it is shown that $\mathrm{Sh}_{\geq 0}( \mathbb{R}; \mathrm{Sp})$ is an $\omega_1$-compact generator of the $\infty$-category of dualizable stable $\infty$-categories. (See \cite[ Theorem D.1.]{efimov2025ktheorylocalizinginvariantslarge})
\end{example}

\begin{example}
Consider the topological space of real numbers $\mathbb{R}$. It is an easy fact to check that the set of rational, open intervals $(a,b)$ for $a < b \in \mathbb{Q}$ together with the empty set forms a basis for the topology of $\mathbb{R}$. Applying Theorem \ref{basistheorem} we thus get the presentation of the locale of reals as:
\begin{itemize}
\item The subposet $\mathrm{RatInt} \subset \mathbb{Q}^{op} \times \mathbb{Q} \cup \{\emptyset \}$, spanned by $(a,b)$ such that $a < b$ together with $\emptyset$, where $\emptyset$ is a formally added bottom element.
\item Coverings are given by three cases:
\begin{itemize}
\item $(a,b)$ is covered by $(a,d), (c,b)$ whenever $a < c < d < b$.
\item $(a,b)$ is covered by $\{(c,d) ~|~ a < c < d < b \}.$
\item The empty set $\emptyset$ is covered by the empty covering.
\end{itemize}
\end{itemize}
This is for example found in \cite{picado_pultr}, Chapter XIV, or Johnstone \cite{johnstone1982stone}. Restricting to dyadic rationals or other dense subsets of $\mathbb{R}$ also works. We leave it to the reader to give the analogous description of a site for $\mathbb{R}^n$.
\end{example}

\begin{example}
Let $\mathcal{F}$ be a $\sigma$-algebra on a set $X$, that is, a sub-Boolean algebra $\mathcal{F} \subset \mathcal{P}(X)$ closed under countable unions. Then the poset $(\mathcal{F},\subset)$ can be equipped with the \emph{countable cover} pretopology, by declaring
$$\{ U_i \subset U ~|~ i \in I \}$$
to be a covering if $I$ is a \emph{countable} set and $\bigcup_{ i\in I} U_i = U$. Denote the resulting locale by $\mathcal{L}$. The locale $\mathcal{L}$ plays a central role in the thesis by Jackson \cite{jackson_sheaves}, for example Jackson shows that measures (in the classical sense) on $\mathcal{F}$ correspond uniquely to (localic) measures on the locale $\mathcal{L}$ (see \cite[Section 2.1, Theorem 2]{jackson_sheaves}), a result that we obtain later as a special case of Theorem \ref{valuationbasis}.
\end{example}

\subsection{Flat functors and morphisms of sites}

Suppose we are given a locally cartesian $0$-site $(P,\tau)$ and a locale $M$. A continuous map $f : M \rightarrow L(P,\tau)$ 
is by definition defined via its inverse image functor
$$f^* : \mathrm{Sh}(P, \tau; \mathbf{2}) \rightarrow \Omega(M),$$
which is a frame homomorphism. Since $f^*$ is colimit preserving, it is in turn completely determined by its composition with the Yoneda functor $\mathbf{f} = f^* \circ [-] : P \rightarrow \mathcal{O}(M),$ since any propositional sheaf is a supremum of elementary propositions, as seen in Lemma \ref{coyonedalemma}. In the following we are concerned with the converse question: When does a functor $\mathbf{f} : P \rightarrow \mathcal{O}(M)$ define a continuous map $f : M \rightarrow L(P,\tau)$?

\begin{definition}
Let $(P,\tau)$ be a locally cartesian $0$-site and $M$ a locale. A functor $\mathbf{f} : P \rightarrow \mathcal{O}(M)$ is called (partial) $\tau$-flat, if it is obtained as $\mathbf{f} = f^* \circ [-]$ for a (partial) continuous map 
$f : M \rightarrow L(P,\tau).$
\end{definition}

We first analyse the situation of extending to the functor $\mathbf{f}$ to the category of propositional presheaves $\mathrm{Fun}(P^{op}, \mathbf{2})$. Suppose $\mathcal{E}$ is a poset with arbitrary suprema and $\mathbf{f} : P \rightarrow \mathcal{E}$ any functor. We can construct the left Kan extension $\mathbf{f}_! : \mathrm{Fun}(P^{op}, \mathbf{2}) \rightarrow \mathcal{E}$ as the left adjoint to the functor
$$\begin{array}{rcl}
\mathbf{f}^* : \mathcal{E} & \rightarrow & \mathrm{Fun}(P^{op}, \mathbf{2}) \\
e & \mapsto & \mathrm{Hom}_{\mathcal{E}}( \mathbf{f}(-), e)
\end{array}$$
or equivalently via the formula
$$\mathbf{f}_!(U) = \bigvee_{p \in U} \mathbf{f}(p).$$

\begin{theorem}[Universal property of pre-sheaves] \label{universalpropertypresheaves}
Let $P$ be a poset, and $\mathcal{E}$ a complete lattice. Then the Yoneda functor $y : P \rightarrow \mathrm{Fun}(P^{op}, \mathbf{2})$ induces a bijection
$$\mathrm{Fun}^{L}(\mathrm{Fun}(P^{op}, \mathbf{2}), \mathcal{E}) \cong \mathrm{Fun}(P, \mathcal{E}),$$
with inverse given by $\mathbf{f} \mapsto \mathbf{f}_!$. Here, $\mathrm{Fun}^{L}$ refers to functors preserving arbitrary suprema. Moreover, suppose $\mathbf{f} : P \rightarrow \mathcal{E}$ is given. Then:
\begin{itemize}
\item $\mathbf{f}_!(1) = 1$ iff $\bigvee_{p \in P} \mathbf{f}(p) = 1$ in $\mathcal{E}$.
\item Suppose $P$ has (binary) meets, $\mathcal{E}$ is a frame and  $\mathbf{f}$ preserves meets. Then $\mathbf{f}_!$ also preserves meets.
\end{itemize} 
\end{theorem}

\begin{proof}
It is clear that restriction along Yoneda and left Kan extension are inverse to each other, thus giving the claimed bijection.

The claim that $\mathbf{f}_!(1) = 1$ iff $\bigvee_{p \in P} f(p) = 1$ in $\mathcal{E}$ is also clear as in $\mathrm{Fun}(P^{op}, \mathbf{2})$ we have $1 = \bigvee_{p \in P} p$.

Finally, assume $P$ has meets, $\mathcal{E}$ is a frame and  $\mathbf{f}$ preserves meets. Let $U, V \in \mathrm{Fun}(P^{op}, \mathbf{2})$. Using that $\mathrm{Fun}(P^{op}, \mathbf{2})$ is a frame and that $y$ preserves any existing meets, we have
$$ U \wedge V = \bigvee_{p \in U, q \in V} y_{p \wedge q}.$$
Applying the supremum preserving functor $\mathbf{f}_!$ we obtain
$$ \mathbf{f}_!( U \wedge V ) = \bigvee_{p \in U, q \in V} \mathbf{f}(p \wedge q) = \bigvee_{p \in U, q \in V} \mathbf{f}(p) \wedge \mathbf{f}(q) = \bigvee_{p \in U} \mathbf{f}(p) \wedge \bigvee_{q \in V} \mathbf{f}(q) = \mathbf{f}_!(U) \wedge \mathbf{f}_!(V)$$
where we used that meets distribute over suprema in $\mathcal{E}$ and that $\mathbf{f}$ preserves finite meets.
\end{proof}

We now discuss how to incorporate Grothendieck pretopologies. Let us introduce with some terminology.

\begin{definition}
Let $(P,\tau)$ and $(P',\tau')$ be two locally cartesian $0$-sites. Let $\mathbf{f} : P \rightarrow P'$ be a functor. We say that:
\begin{itemize}
\item $\mathbf{f}$ \emph{preserves coverings} if whenever $\{p_i \leq p ~|~ i \in I\}$ is a covering in $P$, then $\{\mathbf{f}(p_i) \leq \mathbf{f}(p) ~|~ i \in I\}$ is a covering in $P'$.
\item $\mathbf{f}$ is a \emph{partial morphism of sites} if it preserves coverings and binary meets.
\item $\mathbf{f}$ has \emph{global support} if every element of $P'$ can be covered by elements of the form $\mathbf{f}(p)$ for $p \in P$. 
\item $\mathbf{f}$ is a \emph{morphism of sites} if it is a partial morphism of sites with global support.
\end{itemize}
\end{definition}

It is easy to verify that (partial) morphisms of sites compose, hence we obtain the categories $\mathrm{Site}_{\mathrm{part}}$ and $\mathrm{Site}$ of locally cartesian $0$-sites with (partial) morphisms of sites.

For the following, note that the inclusion $ i_* : \mathrm{Sh}(P, \tau; \mathbf{2}) \subset \mathrm{Fun}(P^{op}, \mathbf{2})$ determines a sublocale.

\begin{theorem}[Universal property of sheaves] \label{universalpropertysheaves}
Let $(P,\tau)$ be a locally cartesian $0$-site and $\mathcal{E}$ a complete lattice. Then the elementary proposition functor $[-] : P \rightarrow \mathrm{Sh}(P, \tau; \mathbf{2})$ induces a bijection
$$\mathrm{Fun}^{L}(\mathrm{Sh}(P, \tau; \mathbf{2}), \mathcal{E}) \cong \mathrm{Fun}^{cov}(P, \tau; \mathcal{E}),$$
with inverse $\mathbf{f} \mapsto \mathbf{f}_\# = \mathbf{f}_! i_*$, where 
$\mathrm{Fun}^{cov}(P, \tau; \mathcal{E})$ is the sub-poset of \emph{covering-preserving} functors, that is those $\mathbf{f} : P \rightarrow \mathrm{E}$ such that for any $\tau$-covering $\{ p_i \leq p ~|~ i \in I\}$ we have
$$\mathbf{f}(p) = \bigvee_{i \in I} \mathbf{f}(p_i) \text{ in } \mathcal{E}.$$
Moreover, suppose such a covering-preserving functor $\mathbf{f} : P \rightarrow \mathcal{E}$ is given. Then:
\begin{itemize}
\item The right adjoint to $\mathbf{f}_\#$ is given by
$$\begin{array}{rcl}
\mathbf{f}^\# : \mathcal{E} & \rightarrow & \mathrm{Sh}(P, \tau; \mathbf{2}) \\
e & \mapsto & \mathrm{Hom}_\mathcal{E}( \mathbf{f}(-) , e).
\end{array}$$
\item $\mathbf{f}_\#(1) = 1$ iff $\bigvee_{p \in P} \mathbf{f}(p) = 1$ in $\mathcal{E}$.
\item Suppose $\mathcal{E}$ is a frame and  $\mathbf{f}$ preserves meets. Then $\mathbf{f}_\#$ also preserves meets.
\end{itemize}
\end{theorem}

In particular, if $F$ is a frame, then $\mathbf{f} : P \rightarrow F$ is (partial) $\tau$-flat iff $\mathbf{f}$ is a (partial) morphism of sites, where we equip $F$ with the canonical topology.

\begin{proof} Let $\mathbf{f} : P \rightarrow \mathcal{E}$ be a functor. By Theorem \ref{universalpropertypresheaves} we have a unique supremum-preserving extension $\mathbf{f}_! : \mathrm{Fun}(P^{op}, \mathbf{2}) \rightarrow \mathcal{E}$. Let $\{p_i \leq p ~|~ i \in I\}$ be a $\tau$-covering. Then $ [p] = \bigvee_{i \in I} [p_i]$ in $\mathrm{Sh}(P, \tau; \mathbf{2})$, so if $\mathbf{f}_!$ descends to propositional sheaves, we necessarily have that $f$ is covering preserving. For the converse, since $ i_* : \mathrm{Sh}(P, \tau; \mathbf{2}) \subset \mathrm{Fun}(P^{op}, \mathbf{2})$ is a sublocale, by Proposition \ref{sublocaleuniversal} the functor $\mathbf{f}_!$ descends to $\mathrm{Sh}(P, \tau; \mathbf{2})$ iff for all downward closed sets $V$ we have $\mathbf{f}_!( V ) = \mathbf{f}_! (V^{sh})$. We have $V^{sh} = \{p \in P ~|~ \exists \text{ covering } \{p_i \leq p ~|~ i \in I \} \text{ s.t. } p_i \in V \text{ for all } i \in I \}$
hence $\mathbf{f}_!( V ) = \mathbf{f}_! (V^{sh})$ iff
$$ \bigvee_{ p \in V } \mathbf{f} (p) = \bigvee_{ q : ~ \exists \text{ covering with } p_i \in V } \mathbf{f}(q).$$
Note that we always have $\leq$ for the above line. Now, assume that $\mathbf{f}$ preserves coverings, then any $f(q)$ appearing in the right hand side we have $f(q) \leq \bigvee f(p_i)$ with $p_i$ in $V$. Therefore the reverse inequality also holds.

Again by Proposition \ref{sublocaleuniversal} in the case $\mathbf{f}_!$ descends we necessarily have $\mathbf{f}_\# = \mathbf{f}_! i_*$, so the claims about preservation of the top element or meets reduce to the analogous statements made in Theorem \ref{universalpropertypresheaves}.
\end{proof}

\begin{corollary} \label{flatfunctor}
Let $(P,\tau)$ be a locally cartesian $0$-site and $M$ a locale. There is a bijection
$$\begin{array}{rcl}
\mathrm{Map}_{\mathrm{part}}(M, L(P, \tau) ) &\cong& \mathrm{Mor}_{\mathrm{Site}}^{\mathrm{part}}((P,\tau), (\mathcal{O}(M),can) ) \\
f &\mapsto & f^* \circ [-]
\end{array}$$
where the left-hand side refers to partial continuous maps between locales. Under this bijection, the domain of a partial map $f : M \rightarrow L(P, \tau)$ is computed as $ \mathrm{dom}(f) = \bigvee_{p \in P} f^*([p])$. This bijection restricts furthermore to a bijection
$$\mathrm{Map}(M, L(P, \tau) ) \cong \mathrm{Mor}_{\mathrm{Site}}((P,\tau), (\mathcal{O}(M),can) ).$$
\end{corollary}

If we think of our site as given by elementary propositions, this just says that a continuous map is determined by an assignment $f^*$ which preserves (external) implication, ``and'' and ``or''.

\begin{example} \label{point}
In particular, if $L = \mathrm{pt}$, we obtain a characterization of a point $x$ of $L(P,\tau)$. It is a collection $x \subset P$, for which we use the formal symbol
``$x \in p$'' whenever $p \in x$, such that:
\begin{itemize}
\item There exists at least one $p \in P$ such that ``$x \in p$''.
\item If $p \leq q$, we have ``$x \in p$'' implies ``$x \in q$''.
\item If ``$x \in p$'' and ``$x \in q$'', then ``$x \in p \wedge q$''.
\item If $\{p_i\}_{i \in I}$ cover $p$, and ``$x \in p$'', then there exists an $i \in I$ such that ``$x \in p_i$''.
\end{itemize}
\end{example}

\begin{example}
It is clear that the Yoneda functor
$$[-] : P \rightarrow \mathrm{Sh}(P, \tau; \mathbf{2})$$
is $\tau$-flat, corresponding to the identity on $L(P,\tau)$.
\end{example}

We now come to the most flexible situation. Suppose $(P,\tau)$ and $(P',\tau')$ are two locally cartesian $0$-sites. When does a functor $\mathbf{f} : P \rightarrow P'$ induce a frame homomorphism on propositional sheaves?

\begin{theorem} \label{functorialitymorphismsites}
Let $(P,\tau)$ and $(P',\tau')$ be two locally cartesian $0$-sites. Then the assignments
$$\begin{array}{rcl}
\mathrm{Mor}_{\mathrm{Site}}^{\mathrm{part}}((P,\tau),(P',\tau')) & \rightarrow & \mathrm{Hom}^{\mathrm{part}}_{\mathrm{Frm}}( \mathrm{Sh}(P, \tau; \mathbf{2}), \mathrm{Sh}(P', \tau'; \mathbf{2})) \\
\mathbf{f} & \mapsto & ([\mathbf{f}])_\#
\end{array}$$
and
$$\begin{array}{rcl}
\mathrm{Mor}_{\mathrm{Site}}((P,\tau),(P',\tau')) & \rightarrow & \mathrm{Hom}_{\mathrm{Frm}}( \mathrm{Sh}(P, \tau; \mathbf{2}), \mathrm{Sh}(P', \tau'; \mathbf{2})) \\
\mathbf{f} & \mapsto & ([\mathbf{f}])_\#
\end{array}$$
are well-defined. Moreover, for a given (partial) morphism of sites $\mathbf{f} : P \rightarrow P'$, the right adjoint to
\mbox{$([\mathbf{f}])_\# : \mathrm{Sh}(P, \tau; \mathbf{2}) \rightarrow  \mathrm{Sh}(P', \tau'; \mathbf{2})$} is given by
$$\begin{array}{rcl}
\mathbf{f}^\# : \mathrm{Sh}(P', \tau'; \mathbf{2}) & \rightarrow & \mathrm{Sh}(P, \tau; \mathbf{2}) \\
F & \mapsto & F \circ \mathbf{f}.
\end{array}$$
\end{theorem}

In particular, the above proposition states that we have a functor
$$\mathrm{Sh}(-; \mathbf{2}) : \mathrm{Site}_{\mathrm{part}} \rightarrow \mathrm{Frm}_{\mathrm{part}},$$
and similarly
$$\mathrm{Sh}(-; \mathbf{2}) : \mathrm{Site} \rightarrow \mathrm{Frm}.$$
\begin{proof}
The claims about well-definedness follow directly from Corollary \ref{flatfunctor}, since for a (partial) morphism of sites $\mathbf{f}$, the composite $[\mathbf{f}] = [-] \circ \mathbf{f} : P \rightarrow \mathrm{Sh}(P', \tau'; \mathbf{2})$ is again a morphism of sites, where we equip $\mathrm{Sh}(P', \tau'; \mathbf{2})$ with the canonical topology. The statement about the right adjoint of $([\mathbf{f}])_\#$ follows directly from the description of the right adjoint provided by Theorem \ref{universalpropertysheaves} together with the Yoneda lemma.
\end{proof}

Consider the assignment $F \mapsto (F, can)$ where $F$ is a frame and $can$ is the canonical topology. Since a (partial) frame homomorphism $f^* : F \rightarrow F'$ is the same as a (partial) morphism of sites with respect to their canonical topologies we obtain a fully faithful functor
$$ can : \mathrm{Frm}_{\mathrm{part}} \rightarrow \mathrm{Site}_{\mathrm{part}}$$
which restricts to a fully faithful functor
$$can : \mathrm{Frm} \rightarrow \mathrm{Site}.$$
The following is an immediate consequence of Corollary \ref{flatfunctor}.

\begin{corollary}
There are adjunctions
\[\begin{tikzcd}
	{\mathrm{Site}_\mathrm{part}} & {\mathrm{Frm}_\mathrm{part}}
	\arrow[""{name=0, anchor=center, inner sep=0}, "{\mathrm{Sh}(-;\mathbf{2})}", curve={height=-6pt}, from=1-1, to=1-2]
	\arrow[""{name=1, anchor=center, inner sep=0}, "can", curve={height=-6pt}, hook', from=1-2, to=1-1]
	\arrow["\dashv"{anchor=center, rotate=-90}, draw=none, from=0, to=1]
\end{tikzcd}\]
and
\[\begin{tikzcd}
	{\mathrm{Site}} & {\mathrm{Frm}}
	\arrow[""{name=0, anchor=center, inner sep=0}, "{\mathrm{Sh}(-;\mathbf{2})}", curve={height=-6pt}, from=1-1, to=1-2]
	\arrow[""{name=1, anchor=center, inner sep=0}, "can", curve={height=-6pt}, hook', from=1-2, to=1-1]
	\arrow["\dashv"{anchor=center, rotate=-90}, draw=none, from=0, to=1]
\end{tikzcd}\]
with fully faithful right adjoints.
\end{corollary}

Philosophically speaking, this means that the notion of a (propositional) sheaf is in a suitable sense completely determined by the notions of frame and coverage.

\section{Valuations on lattices and frames}

Whereas measure theory and classical point-set topology may have many similarities, their theories are more or less parallel to each other - both $\sigma$-algebras as well as topologies are defined as different types of (infinitary) algebraic substructures of $\mathcal{P}(X)$ for a set $X$. When changing from point-set topology to locale theory however, measure theory integrates extremely well. We begin with the general definition of a valuation.

\subsection{Valuations}

\begin{definition}
Let $D$ be a lower bounded distributive lattice with bottom element $0$. We define the following terms:
\begin{itemize}
\item A \emph{valuation} on $D$ is a functor $\mu : D \rightarrow ([0,+\infty], \leq )$ such that $\mu(0) = 0$ and
$$ \mu(p) + \mu(q) = \mu(p \vee q) + \mu(p \wedge q) $$
for all $p,q$ in $D$.
\item A valuation is called \emph{finite} if $\mu(p) < +\infty$ for all $p$ in $D$.
\end{itemize}
Let $\mathbf{f} : D' \rightarrow D$ be a homomorphism of lower bounded distributive lattices, and suppose $\mu$ is a valuation on $D$. 
\begin{itemize}
\item The \emph{pullback} $\mathbf{f}^\# \mu$ of $\mu$ along $\mathbf{f}$ is defined as $\mathbf{f}^\# \mu(d) = \mu(\mathbf{f}(d))$, and gives a valuation on $D'$.
\item Suppose $\mu'$ is a valuation on $D'$. Then $\mathbf{f}$ is called \emph{valuation-preserving} if $\mathbf{f}^\# \mu = \mu'$.
\end{itemize}
Denote the set of finite valuations on $D$ by $\mathrm{Val}(D)$. A pair $(D,\mu)$ where $D$ is a lower bounded distributive lattice and $\mu$ a \emph{finite} valuation on $D$ is called a \emph{valuation site}. Valuation sites together with valuation-preserving homomorphisms form a category $\mathrm{ValSite}$.
\end{definition}

Now fix a lower bounded distributive lattice $D$ with a finite valuation $\mu$. We will use the following two simple lemmata.

\begin{lemma} \label{valuationlemma1}
Let $D$ be a lower bounded distributive lattice and $\mu$ a valuation on $D$. Suppose $p, q, r \in D$ such that $q \leq p$. Then
$$ \mu(p \wedge r) - \mu( q \wedge r ) \leq \mu(p) - \mu(q). $$
\end{lemma}

\begin{proof}
We compute
$$\begin{array}{rcl}
\mu(p \wedge r) - \mu( q \wedge r ) &=& \mu(p) + \mu(r) - \mu( p \vee r) - \mu(q) - \mu(r) + \mu(q \vee r) \\
 &=& \mu(p) - \mu(q) - ( \mu( p \vee r) - \mu(q \vee r) ) \\
 &\leq& \mu(p) - \mu(q)
\end{array}$$
since $p \geq q$ and therefore $p \vee r \geq q \vee r$.
\end{proof}

\begin{lemma} \label{valuationlemma2}
Let $p_1, p_2, q_1, q_2 \in D$ s.t. $q_1 \leq p_1$ and $q_2 \leq p_2$. Then
$$ \mu( p_1  \vee p_2 ) - \mu( q_1 \vee q_2 ) \leq \mu(p_1) - \mu(q_1) + \mu(p_2) - \mu(q_2). $$
\end{lemma}

\begin{proof}
We compute
$$\begin{array}{rcl}
\mu( p_1  \vee p_2 ) - \mu( q_1 \vee q_2 )  &=& \mu(p_1) + \mu(p_2) - \mu(q_1) - \mu( q_2 ) - ( \mu( p_1 \wedge p_2 ) - \mu( q_1 \wedge q_2 ) ) \\ &\leq & \mu(p_1) + \mu(p_2) - \mu(q_1) - \mu( q_2 ),
\end{array}$$
since $ p_1 \wedge p_2 \geq q_1 \wedge q_2$.
\end{proof}

\begin{lemma} \label{idealoffinitemeasure} 
Let $D$ be a lower bounded distributive lattice and $\mu$ a valuation. Then the subset $(D,\mu)^{fin}$ of $D$ consisting of those $p$ such that $\mu(p) < \infty$ is an ideal of $D$, and in particular itself a lower bounded distributive lattice. The valuation $\mu$ restricts to a finite valuation on $D$.
\end{lemma}

\begin{proof}We only need to check two things: 
\begin{itemize}
\item If $p \leq q$, and $\mu(q) < \infty$, then $\mu(p) \leq \mu(q) < \infty$.
\item If $p_1, p_2$ such that $\mu(p_i) < \infty$, then
$$\mu(p_1 \vee p_2) =  \mu(p_1) + \mu(p_2) - \mu( p_1 \wedge p_2) < \infty.$$
\end{itemize}
The statement that $\mu$ restricts to a finite valuation is obvious.
\end{proof}

\subsection{Measure locales}

\begin{definition} \label{definitionmeasure}
Let $F$ be a frame.
\begin{itemize}
\item A valuation $\mu$ on $F$ is called \emph{continuous} if
$$\mu( \bigvee_{i \in I} U_i ) = \sup_{i \in I} \mu( U_i )$$
for all directed sets $\{ U_i \}$.
\item A continuous valuation is called a \emph{measure}.
\item An open $K \in F$ is said to have \emph{finite measure} if $\mu(K) < \infty$.
\item A valuation $\mu$ on $F$ is called \emph{locally finite} if there exists a covering of $1$ with opens of finite measure.
\item A valuation is called \emph{faithful} if the functors $\mu(- \wedge K)$ are jointly conservative for $K$ of finite measure. This means concretely that whenever $U \leq V$ are opens, if $\mu( U \wedge K ) = \mu( V \wedge K)$ holds for all $K$ with finite measure, then $U = V$.
\end{itemize}
Denote the set of measures on $F$ by $\mathrm{Meas}(F)$. We call a pair $(F,\mu)$ a \emph{measure frame} and a valuation preserving partial frame homomorphism \emph{measure-preserving}. A measure frame $(F,\mu)$ will be called \emph{faithful} if $\mu$ is locally finite and faithful. Measure frames together with measure-preserving \emph{partial} frame homomorphisms form a category $\mathrm{MeasFrm}$. The opposite category will be referred to as $\mathrm{MeasLoc}$, with objects \emph{measure locales} and morphisms given by partial, measure-preserving, continuous maps.
\end{definition}

\begin{remark}
We will sometimes use the terminology \emph{chunk} to refer to an open of finite measure of a measure locale, a terminology that would be consistent with \cite{segal1951}. In case the measure locale $L$ is a \emph{probability locale}, i.e. $L$ is Boolean and $\mu(1) = 1$, it is also sensible to refer to opens as \emph{events}.
\end{remark}

\begin{remark}
We remark that a measure on a locale $L$ can be thought of as a decategorification of the concept of a cosheaf on $L$. The intuition for the definition of a locally finite measure $\mu$ is that while an arbitrary open $U$ can still have infinite measure, if this does happen, there is a \emph{reason} coming from statements about finite measure below $U$. To be more concrete, let $U$ be an open. If $\mu$ is a locally finite and continuous valuation, then
$$\mu( U ) = \sup_{K \leq U, \mu(K) < \infty} \mu(K)$$
is obtained as a supremum of finite values. The intuition for the notion of a \emph{faithful} measure frame comes from the requirement that in practice, one often wants to disregard objects of measure zero. 
\end{remark}

Some simple observations.
\begin{itemize}
\item A valuation $\mu$ on a frame (or more generally bounded distributive lattice) is finite iff $\mu(1) < \infty$.
\item If $\mu$ is a faithful valuation, then $\mu(U) = 0$ implies $U = 0$.
\item If $\mu$ is a faithful and finite valuation, then $\mu(U) = \mu(1)$ implies $U = 1$.
\end{itemize}

Let $f : L \rightarrow L'$ be a partial continuous map of locales and $\mu$ a valuation on $\mathcal{O}(L)$. Define the \emph{pushforward} $f_* \mu : \mathcal{O}(L') \rightarrow [0, \infty]$ via
$f_* \mu(U) = \mu( f^*(U) ).$

\begin{proposition} \label{pushforwardproperties}
Let $f : L \rightarrow L'$ be a partial continuous map of locales and $\mu$ a valuation on $\mathcal{O}(L)$.
\begin{itemize}
\item If $\mu$ is a measure, then $f_* \mu$ is a measure.
\item If $f$ is \emph{weakly proper}, by which we mean that $f_*$ preserves directed suprema, and $\mu$ is locally finite, then $f_* \mu$ is also locally finite.
\end{itemize}
\end{proposition}

We note that the requirement for $f$ to be weakly proper in the second statement cannot be dropped in general. As a counterexample consider the projection $\pi_1 : \mathbb{R}^2 \rightarrow \mathbb{R}$ with the domain equipped with the Lebesgue measure.

\begin{proof}
Suppose $\mu$ is continuous, and $U = \bigvee_{i \in I} U_i$ a directed supremum of opens in $L'$. Then
$$f_* \mu( \bigvee_{i \in I} U_i ) = \mu( f^* \bigvee_{i \in I} U_i ) = \mu( \bigvee_{i \in I} f^* (U_i) ) = \sup_{i \in I} f_* \mu( U_i ),$$
in other words $f_* \mu$ is continuous.

Now assume $f_*$ preserves directed suprema and $\mu$ is locally finite. We know $f^*(1) = \bigvee_{ U \leq f^*(1) ~:~ \mu(U) < \infty} U$ since $\mu$ is locally finite. Adjoining over, and using that $f_*$ preserves directed suprema we have
$$1 \leq  \bigvee_{ U \leq f^*(1) ~:~ \mu(U) < \infty} f_*(U).$$
But $f_* \mu( f_*(U) ) = \mu( f^* f_*(U) ) \leq \mu(U) < \infty$, therefore we have obtained $1$ as a supremum of finite measure elements.
\end{proof}

\begin{example}
Consider a locale $L$ and open $U$. If $\mu$ is a measure on $L$, then the pushforward along the partial map $ j : L \rightarrow U$ equips $U$ with the \emph{restriction} $j_* \mu$ of $\mu$. If $\mu$ is locally finite or faithful, then so is the restriction. The resulting measure-preserving partial frame homomorphism is one of the reasons to consider the category of measure locales and \emph{partial} measure-preserving maps, rather than the more restrictive notion of globally defined measure-preserving maps.

Conversely, one can consider the inclusion of the closed complement $i : U^c \hookrightarrow L$. If $\nu$ is a measure on $U^c$, then $i_* \nu$ is the measure on $L$ given by \emph{extension by zero}. If $\nu$ is locally finite, then so is $i_* \nu$, but faithfulness is not inherited.
\end{example}

\begin{example}
If $x : \mathrm{pt} \rightarrow L$ is a point of a locale $L$, we can equip $L$ with the \emph{Dirac measure} $\delta_x$ centered at $x$, defined as the pushforward $\delta_x = x_* \delta$, where $\delta$ is the unique measure on the point $\mathrm{pt}$ with value $\delta(\mathrm{pt}) = 1$.
\end{example}

\begin{lemma} \label{pullbacklocallyfinite}
Let $f : (L, \mu) \rightarrow (L', \mu')$ be a global continuous measure-preserving map. If $\mu'$ is locally finite, then so is $\mu$.
\end{lemma}

\begin{proof}
Pull back the cover $1 = \bigvee_{ U' ~:~ \mu'(U') < \infty} U'$ of $L'$ using $f^*$ to get a cover of $f^*(1) = 1$ in $L$ by opens of finite measure.
\end{proof}

Measures on a frame can always be constructed from local data using a basis. Suppose $D \subset F$ is a basis that is closed under finite joins. Then any valuation $\mu$ on $F$ restricts to a valuation $\mu|_D : D \rightarrow [0,\infty] $ on $D$, and since any open in $F$ is obtained as a directed supremum of basic opens in $D$, by applying continuity of $\mu$, we see that $\mu$ is completely determined by its restriction. Let us moreover call $\mu$ \emph{$D$-finite} if $\mu|_D$ is a finite valuation when restricted to $D$. We note that the existence of such a basis $D$ for which $\mu$ is $D$-finite implies that $\mu$ is locally finite.

Conversely, suppose $\nu$ is a functor $D \rightarrow [0, \infty)$. When does $\nu$ extend to a measure on $F$? Note that if $\nu$ arose as the restriction of a continuous valuation, we would have the condition that
$$ \nu(p) = \sup_{i \in I} \nu(p_i)$$
whenever $p = \bigvee_{i \in I} p_i$ is a directed supremum in $F$ with $p, p_i \in D$. But this is by definition just saying that $\{ p_i \leq p ~|~ i \in I \}$ is a directed covering in the induced Grothendieck pretopology $\tau$ on $D$. Call such valuations on $D$ $\tau$-compatible. 

Given a $\tau$-compatible valuation $\nu$, one might try the definition
$$\nu_*(U) = \sup_{ d \in D, d \leq U} \nu(d)$$
for $U \in F$. The next theorem argues that this indeed produces a measure.

\begin{theorem} \label{valuationbasis}
Let $F$ be a frame and $D \subset F$ a basis closed under finite joins (including the bottom element). Then there is a bijection
$$(-)|_D : \mathrm{Meas}(F) \cong \mathrm{Val}(D,\tau)$$
where the set $\mathrm{Val}(D,\tau)$ is the set of valuations on $D$ that are compatible with the induced Grothendieck pretopology $\tau$. The inverse is given by $\nu \mapsto \nu_*$. The bijection restricts to a bijection
$$(-)|_D : \mathrm{Meas}(F)^{D\mathrm{-fin}} \cong \mathrm{Val}_{\mathrm{fin}}(D,\tau)$$
where:
\begin{itemize}
\item The set $\mathrm{Meas}(F)^{D\mathrm{-fin}}$ is the set of $D$-finite measures on $F$.
\item The set $\mathrm{Val}_{\mathrm{fin}}(D,\tau)$ is the set of finite valuations on $D$ that are compatible with the induced Grothendieck pretopology $\tau$.
\end{itemize}
\end{theorem}

\begin{proof}
It is straightforward to see that restriction $(-)|_D$ is well-defined and that the inverse map is given by $\nu \mapsto \nu_*$. The main difficulty lies in verifying that $\nu_*$ is actually a measure for a given $\tau$-compatible valuation $\nu$. We note that it is also clear that $\nu_*(0) = \nu(0) = 0$ and that $\nu_*(U) \leq \nu_*(V)$ whenever $U \leq V$ in $F$. We need to show that $\nu_*$ is continuous and that it satisfies the modularity condition. \\

\noindent Continuous:  Suppose $U = \bigvee_{i \in I} U_i$ is a directed supremum in $F$. We always have
$$ \sup_{i \in I} \nu_*( U_i ) = \sup_{ i \in I} \sup_{q \leq U_i, q \in D} \nu(q) \leq \sup_{ p \leq U, p \in D } \nu(p) = \nu_*(U). $$
To see the reverse inequality, note that for a given $p \leq U$ with $p \in D$ we have
$$ p = \bigvee_{i \in I} p \wedge U_i = \bigvee_{i \in I} \bigvee_{q \leq U_i, q \in D} p \wedge q.$$
The latter is a directed supremum with elements in $D$. Since $\nu$ is compatible with $\tau$, this means that
$$ \nu(p) = \sup_{i \in I} \sup_{q \leq U_i, q \in D} \nu( p \wedge q)$$
Passing to suprema over $p$ we see the reverse inequality.

\noindent Modularity: Let $U, V$ in $F$ be given. Then $U = \bigvee_{p \leq U, p \in D} p$ and $V = \bigvee_{q \leq V, q \in D} q$ are directed suprema. Combining them we get the directed suprema
$$ U \vee V = \bigvee_{p \leq U, q \leq V, p,q \in D} p \vee q \text{ and } U \wedge V = \bigvee_{p \leq U, q \leq V, p,q \in D} p \wedge q.$$
We verify that
$$\begin{array}{rcl}
\nu_*(U) + \nu_*(V) &=& \sup_{p \leq U, q \leq V, p,q \in D} (\nu(p) + \nu(q)) \\
&=& \sup_{p \leq U, q \leq V, p,q \in D} (\nu(p \vee q) + \nu(p \wedge q)) \\
&\leq&  \sup_{p \leq U, q \leq V, p,q \in D} \nu(p \vee q) +  \sup_{p' \leq U, q' \leq V, p',q' \in D} \nu(p' \wedge q') \\
&=& \nu_*( U \vee V ) + \nu_*(U \wedge V)
\end{array}$$
where we have used that $\nu_*$ is continuous. To see the reverse inequality, choose $\epsilon > 0$, and $p,p' \leq U$, $q, q' \leq V$ with $p,p',q,q' \in D$ such that
$$\begin{array}{rcl}
\nu(p \vee q) + \epsilon &\geq & \nu_*( U \vee V ) \\
\nu(p' \wedge q') + \epsilon &\geq & \nu_*(U \wedge V).
\end{array}$$
Define $p'' = p \vee p'$ and $q'' = q \vee q'$. Then $p'' \leq U, p'' \in D$ and $q'' \leq V, q'' \in D$. Adding both sides we get
$$\begin{array}{rcccl}
\nu_*( U \vee V ) + \nu_*(U \wedge V) &\leq& \nu(p \vee q) + \nu(p' \wedge q') + 2 \epsilon &\leq& \nu( p'' \vee q'') + \nu( p'' \wedge q'') + 2\epsilon \\
 &=& \nu(p'') + \nu(q'') + 2\epsilon &\leq& \nu_*(U) + \nu_*(V) + 2\epsilon.
\end{array}$$
Since $\epsilon > 0$ was arbitrary, this completes the proof. (We remark that the proof remains valid even if $\nu_*( U \vee V )$ or $\nu_*(U \wedge V)$ have infinite value.)
\end{proof}


\begin{example} Let $D$ be lower bounded distributive lattice, equipped with the Grothendieck pretopology $fin$. In this situation, all finite valuations $\mu : D \rightarrow [0,\infty)$ extend uniquely to locally finite measures $\mu_*$ on the locally coherent locale $L(D, fin)$, which are finite on compact opens.
\end{example}

\begin{example} \label{scissorscongruence1}
An interesting sub-example of the previous example is obtained by letting $X$ be a geometry, such as Euclidean space $E^n$, the sphere $S^n$ or hyperbolic space $H^n$, of dimension $n$, and considering the associated lower bounded distributive lattice $D(X)$ given by $n$-dimensional polytopes. (Compare with \cite[Section 7.1]{lehner2025algebraicktheorycoherentspaces}.) The corresponding locally coherent space $L(D(X), fin) = {{X_{\mathrm{poly}}}}$ is central for the study of scissors-congruences. The assignment of $n$-dimensional volume $\mathrm{Vol}_n : D(X) \rightarrow [0,\infty)$ is a finite valuation, hence induces a measure $\mathrm{Vol}_n$ on ${{X_{\mathrm{poly}}}}$.
\end{example}

\begin{example}
Let $G$ be a (discrete) group. Then the group $G$ is amenable iff there exists a finitely additive $G$-invariant probability measure $\lambda$ on the powerset $\mathcal{P}(G)$. Since $\mathcal{P}(G)$ is a Boolean algebra, $\lambda$ is equivalently a valuation on $\mathcal{P}(G)$ with $\lambda(G) = 1$. In this case $L(  \mathcal{P}(G), fin ) = \beta(G)$ can be identified with the Stone-\v{C}ech compactification of $G$. We get the result that $G$ is amenable iff there exists a $G$-invariant probability measure on $\beta(G)$.
\end{example}

\begin{remark}
The requirement that a basis $P \subset F$ is closed under finite joins could also be dropped. Since $\mu$ is a valuation, i.e.\ $\mu(0) = 0$ and $\mu$ satisfies modularity, the value of $\mu$ on a finite join of base elements is completely determined, hence $\mu$ is completely determined by its value on $P$ as a functor $\mu|_P : P \rightarrow [0, \infty]$. The conditions required for such a functor to then induce a continuous valuation on $F$ are slightly more complicated to state. We refrain from doing so since the usage of this more general case is not required for the rest of this paper. The reader is encouraged to work out the necessary conditions themselves in case the need arises.
\end{remark}

A measure can be built from local information in the same way a continuous function can. This is summarized in the statement that the set of measures, varied over opens of a locale, forms itself a sheaf. We verify this in steps.

\begin{corollary} \label{measuresaresheaffiltered}
Let $L$ be a locale and $U_i, i \in I$ a directed system of opens of $L$. Suppose $\mu_i : \mathcal{O}(U_i) \rightarrow [0, +\infty]$ is a collection of compatible measures, i.e.\ whenever $U_i \subset U_j$, then $\mu_j |_{U_i} = \mu_i$. Then there exists a unique measure $\mu : \mathcal{O}(U) \rightarrow [0, +\infty]$ such that $\mu|_{U_i} = \mu_i$. 
Furthermore:
\begin{itemize}
\item $\mu$ is locally finite iff $\mu_i$ is locally finite for all $i \in I$.
\item $\mu$ is faithful iff  $\mu_i$ is faithful for all $i \in I$.
\end{itemize}
\end{corollary}

\begin{proof}
Let $D \subset \mathcal{O}(U)$ be the set of $V \leq U$ such that $V \leq U_i$ for some $i \in I$. This set forms an ideal, as $I$ is directed, and since $U = \bigvee_{i \in I} U_i$ it is a basis for $U$. Define $\widetilde{\mu} : D \rightarrow [0,\infty]$ via $\widetilde{\mu}(V) = \mu_i(V)$ for $V \leq U_i$. This definition is independent of choice of $i \in I$, since the measures $\mu_i$ are assumed compatible. It is also clear that $\widetilde{\mu}$ is a $\tau$-compatible valuation on $D$. Define $\mu$ as the unique extension, provided by Theorem \ref{valuationbasis}.

For the statement about local finiteness, we can use the same argument by replacing $D$ by the set of $V$ such that $V \leq U_i$ for some $i \in I$ as well as $\mu(V) < \infty$.

Finally, if $\mu$ is faithful, then clearly its restrictions are faithful. Conversely, assume $\mu_i$ is faithful. Let $V \leq W \leq U$, and assume for all $K \leq U$ with $\mu(K) < \infty$ we have that $\mu( V \wedge K ) = \mu( W \wedge K )$. Then in particular for all $i \in I$ and $K \leq U_i$ with $\mu(K) < \infty$, we have
$$\mu_i( V \wedge K ) = \mu_i( W \wedge K )$$
and hence $V \wedge U_i = W \wedge U_i$ for all $i \in I$. But then
$$ V = \bigvee_{i\in I} V \wedge U_i =  \bigvee_{i\in I} W \wedge U_i = W.$$
\end{proof}

\begin{proposition} \label{measuresaresheaffinite}
Let $L$ be a locale and $U, V$ opens of $L$. Assume $\mu_U : \mathcal{O}(U) \rightarrow [0, \infty]$ and $\mu_V : \mathcal{O}(V) \rightarrow [0, \infty]$ are two measures on $U$, respectively $V$, such that $\mu_U |_{U \wedge V} = \mu_V |_{U \wedge V}$. Then there exists a unique measure $\mu : \mathcal{O}(U \vee V) \rightarrow [0, \infty]$ on $U \vee V$ such that $\mu |_U = \mu_U$ and $\mu |_V$. Furthermore:
\begin{itemize}
\item $\mu$ is locally finite iff $\mu_U$ and $\mu_V$ are locally finite.
\item $\mu$ is faithful iff  $\mu_U$ and $\mu_V$ are faithful.
\end{itemize}
\end{proposition}

\begin{proof}
Define $\mu : \mathcal{O}(U \vee V) \rightarrow [0, \infty]$ as
$$\mu(Z) = \mu_U( Z \wedge U) + \mu_V( Z \wedge V) - \mu_U( Z \wedge U \wedge V )$$
if $\mu_U( Z \wedge U \wedge V ) < \infty$, and otherwise as  $\mu(Z ) = \infty$ for $Z \leq U \vee V$.

We verify that $\mu$ is a measure:
\begin{itemize}
\item $\mu( 0 ) = 0$ is clear.
\item \emph{Monotone:} Assume $Z_1 \leq Z_2 \leq U \vee V$ and w.l.o.g.\ that $\mu_U( Z_i \wedge U \wedge V ) < \infty$ for $i =1,2$ (otherwise we are already done). Then
$$\mu_U( Z_1 \wedge U ) - \mu_U( Z_1 \wedge U \wedge V ) \leq \mu_U( Z_2 \wedge U ) - \mu_U( Z_2 \wedge U \wedge V )$$
by Lemma \ref{valuationlemma1}, and clearly also $\mu_V( Z_1 \wedge V ) \leq \mu_V( Z_2 \wedge V )$. Adding both inequalities we get
$ \mu(Z_1) \leq \mu(Z_2)$.
\item \emph{Modularity:} Let $Z_1, Z_2 \leq U \vee V$ and again w.l.o.g.\ that $\mu_U( Z_i \wedge U \wedge V ) < \infty$ for $i =1,2$. Then
$$\begin{array}{rl}
& \mu(Z_1) + \mu(Z_2) \\ =& \mu_U( Z_1 \wedge U) + \mu_V( Z_1 \wedge V) - \mu_U( Z_1 \wedge U \wedge V ) \\ & + \mu_U( Z_2 \wedge U) + \mu_V( Z_2 \wedge V) - \mu_U( Z_2 \wedge U \wedge V ) \\
=& \mu_U( (Z_1 \vee Z_2) \wedge U) + \mu_V( (Z_1 \vee Z_2) \wedge V) - \mu_U( (Z_1 \vee Z_2) \wedge U \wedge V ) \\
& + \mu_U( (Z_1 \wedge Z_2) \wedge U) + \mu_V( (Z_1 \wedge Z_2) \wedge V) - \mu_U( (Z_1 \wedge Z_2) \wedge U \wedge V ) \\
=& \mu(Z_1 \vee Z_2) + \mu( Z_1 \wedge Z_2 ).
\end{array}$$
\item \emph{Continuity:} Let $Z_i, i\in I$ be a directed system such that $Z_i \leq U \vee V$ for all $i \in I$. Again, w.l.o.g.\ we may assume that $\mu_U( Z_i \wedge U \wedge V ) < \infty$ for $i \in I$. Then
$$\begin{array}{rcl}
\mu( \bigvee_{i \in I} Z_i ) &=& \mu_U( \bigvee_{i \in I} Z_i \wedge U ) + \mu_V( \bigvee_{i \in I} Z_i \wedge V ) - \mu_U( \bigvee_{i \in I} Z_i \wedge U \wedge V ) \\
&=& \lim_{i \in I} \mu_U( Z_i \wedge U ) + \lim_{i \in I} \mu_V( Z_i \wedge V ) - \lim_{i \in I} \mu_U( Z_i \wedge U \wedge V ) \\
&=& \lim_{i \in I} \mu( Z_i ) = \sup_{i \in I} \mu(Z_i)
\end{array}$$
since $\mu(Z_i)$ is a monotone net.
\item \emph{Local finiteness:} Now assume that $\mu_U$ and $\mu_V$ are locally finite. Combining coverings of $U$ and $V$ with elements of finite measure individually produces a covering of $U \vee V$, hence $\mu$ is locally finite. The converse statement that local finiteness of $\mu$ implies that same for $\mu_U$ and $\mu_V$ is clear.
\item \emph{Faithfulness:} Lastly, assume $\mu_U$ and $\mu_V$ are faithful. Let $W_1 \leq W_2 \leq U \vee V$ be two opens and assume that 
$$\mu( W_1 \wedge K ) = \mu( W_2 \wedge K )$$
for all $K$ with $\mu(K) < \infty$. Then by Lemma \ref{valuationlemma1} we also have
$$\mu_U( W_1 \wedge U \wedge K ) = \mu( W_1 \wedge U \wedge K ) = \mu( W_2 \wedge U \wedge K ) = \mu_U( W_2 \wedge U \wedge K )$$
hence $W_1 \wedge U = W_2 \wedge U$, and similarly with $V$. But then
$$ W_1 = (W_1 \wedge U) \vee (W_1 \wedge V) = (W_2 \wedge U) \vee (W_2 \wedge V) = W_2.$$
and hence $\mu$ is faithful. The converse statement that faithfulness of $\mu$ implies faithfulness of $\mu_U$ and $\mu_V$ is clear as well.
\end{itemize}
\end{proof}

The following theorem is a repackaging of the previous statements. Note that we have a functor
$$ \mathrm{Meas} : \mathrm{Loc}_{\mathrm{part}} \rightarrow \mathrm{Set}$$
which sends a locale $L$ to its set of measures, and a map $f : L \rightarrow M$ to its action via pushforward of measures, by Proposition \ref{pushforwardproperties}. We can furthermore precompose it with the fully faithful inclusion
$$(\mathrm{Loc}_{\mathrm{open~emb}})^{op} \hookrightarrow \mathrm{Loc}_{\mathrm{part}} $$
where $\mathrm{Loc}_{\mathrm{open~emb}}$ is the category of locales and open embeddings as morphisms, that sends an open embedding $i : U \rightarrow L$ to the partial frame homomorphism $i_!$, to obtain a functor 
$ \mathrm{Meas} : (\mathrm{Loc}_{\mathrm{open~emb}})^{op} \rightarrow \mathrm{Set}.$

\begin{theorem} \label{measuresaresheaves}
The functor
$$ \mathrm{Meas} : (\mathrm{Loc}_{\mathrm{open~emb}})^{op} \rightarrow \mathrm{Set}$$
is a sheaf with respect to the open cover topology, where $\mathrm{Loc}_{\mathrm{open~emb}}$ is the category of locales and open embeddings as morphisms. In particular, for every locale $L$, the functor
$$ \mathrm{Meas} : \mathcal{O}(L)^{op} \rightarrow \mathrm{Set}$$
is a sheaf on $L$. The analogous statements are true with $\mathrm{Meas}$ replaced by:
\begin{itemize}
\item $\mathrm{Meas}_{\mathrm{loc. fin}}$, the set of locally finite measures, and
\item $\mathrm{Meas}_{\mathrm{loc. fin., faithf}}$, the set of locally finite and faithful measures.
\end{itemize}
\end{theorem}

We remind the reader that a sheaf $\mathcal{F}$ on a locale $L$ is a functor $\mathcal{O}(L)^{op} \rightarrow \mathrm{Set}$ such that the following conditions are satisfied:
\begin{enumerate}
\item $\mathcal{F}(0) = \mathrm{pt}$.
\item For all opens $U, V$ the square
$$\xymatrix{
\mathcal{F}(U \vee V) \ar[r] \ar[d] & \mathcal{F}(V) \ar[d] \\
\mathcal{F}(U) \ar[r] & \mathcal{F}(U \wedge V)
}$$
is a pullback.
\item For all directed systems $U_i, i \in I$ with $U = \bigvee_{i \in I} U_i$ the natural map
$$\mathcal{F}( U ) \rightarrow \lim_{i \in I} \mathcal{F}(U_i)$$
is a bijection.
\end{enumerate}
In the case of the functor $\mathrm{Meas}$, Condition (2) is taken care of by Proposition \ref{measuresaresheaffinite}. Condition (3) is given by Corollary \ref{measuresaresheaffiltered}. Condition (1) translates to the statement that the empty set $\emptyset$ has exactly one measure on it, which is trivially true.

\subsection{Grothendieck topologies from valuations}

In the following section, we will turn Theorem \ref{valuationbasis} on its head. Instead of starting with a Grothendieck pretopology, and looking at compatible valuations, we \emph{start with a valuation} and use it to induce a Grothendieck pretopology.

\begin{definition} \label{muinnertopology}
Let $D$ be a lower bounded distributive lattice and let $\mu$ be a finite valuation. The $\mu$-\emph{inner} Grothendieck pretopology $inn$ on $D$ generated via two cases:
\begin{itemize}
\item Finite unions cover: $\{p_i \leq p ~|~ i \in I\}$ is a cover whenever $$\bigvee_{i \in I} p_i = p$$ for $I$ finite.
\item $\mu$-approximations cover: $\{p_i \leq p ~|~ i \in I\}$  is a cover whenever $p_i, i \in I,$ is directed and $$ \mu(p) = \sup_{i \in I} \mu(p_i).$$
\end{itemize}
Denote the resulting locale by $L(D,\mu)$, called the \emph{inner measure locale} of $D$ and the corresponding frame of propositional sheaves by $\mathrm{Sh}(D,\mu;\mathbf{2})$. If $U$ is an open in $L(D,\mu)$, equivalently described as an $inn$-ideal, we define
$$\mu_*(U) = \sup_{p \in U} \mu(p).$$
We call $\mu_*$ the \emph{inner measure}.
\end{definition}

We note that both types of covering conditions can be combined into the statement that $\{p_i \leq p\}_{i \in I}$ is a cover iff for all $\epsilon > 0$ there exists a finite subset $F \subset I$ such that
$$\mu( p ) - \mu( \bigvee_{i \in F} p_i ) < \epsilon.$$
A specific case worth pointing out is that of the inclusion of a single element $\{ q \leq p \}$, which is a $\mu$-approximation if $\mu(q) = \mu(p)$.

\begin{proposition} \label{mutopology}
The coverings described in the above definition do in fact define a Grothendieck topology.
\end{proposition}

\begin{proof}
We need to check stability under pullback and composition. Let us begin with pullbacks. Assume $\{p_i \leq p\}_{i \in I}$ is a covering and $q \leq p$. Let $\epsilon > 0$ and choose $F \subset I$ finite such that 
$$\mu( p ) - \mu( \bigvee_{i \in F} p_i ) < \epsilon.$$
But then also
$$\mu( q ) - \mu( \bigvee_{i \in F} q \wedge p_i ) < \epsilon $$
by Lemma \ref{valuationlemma1}.

We now check stability under composition. Assume $\{p_i \leq p\}_{i \in I}$ is a covering and $\{p_{ij} \leq p_i\}_{j \in J_i}$ is a covering for each $i \in I$. Let $\epsilon > 0$ and choose $F \subset I$ finite such that 
$$\mu( p ) - \mu( \bigvee_{i \in F} p_i ) < \frac{\epsilon}{2}.$$
Write $n = |F|$. Then choose $F_i \subset J_i$ finite for each $i \in F$ such that
$$\mu( p_i ) - \mu( \bigvee_{j \in F_i} p_{ij} ) < \frac{\epsilon}{2 n}.$$
Now it follows by induction from Lemma \ref{valuationlemma2} that 
$$\mu( \bigvee_{i \in F} p_i ) - \mu( \bigvee_{i \in F} \bigvee_{j \in F_i} p_{ij} ) < \frac{\epsilon}{2}$$
and hence
$$\mu( p ) - \mu( \bigvee_{i \in F} \bigvee_{j \in F_i} p_{ij} ) < \frac{\epsilon}{2} + \frac{\epsilon}{2} = \epsilon,$$
in other words $\{ p_{ij} \leq p \}$ is a covering.
\end{proof}

Since the inner topology on $D$ always includes all finite covers, we have that propositional sheaves for the inner topology embed into propositional sheaves for the finite topology, as elaborated in Example \ref{finitarytopology}. More specifically we have the adjunction
\[\begin{tikzcd}
	{\mathrm{Sh}(D,\mu;\mathbf{2})} & {\mathrm{Sh}(D,fin;\mathbf{2})} & {\mathrm{Idl}(D)} & {\mathrm{Ind}(D).}
	\arrow[""{name=0, anchor=center, inner sep=0}, curve={height=12pt}, hook', from=1-1, to=1-2]
	\arrow[""{name=1, anchor=center, inner sep=0}, "{(-)^{sh}}"', curve={height=12pt}, from=1-2, to=1-1]
	\arrow["\cong"{description}, draw=none, from=1-2, to=1-3]
	\arrow["{ \cong }"{description}, draw=none, from=1-3, to=1-4]
	\arrow["\dashv"{anchor=center, rotate=-90}, shift right, draw=none, from=1, to=0]
\end{tikzcd}\]
In this sense we identify $\mu$-ideals with special types of ideals, namely those that are closed under $\mu$-approximation. The sheafification of an ideal $U$ can be directly computed as
$$U^{sh} = \{ p \in D ~|~ \forall \epsilon > 0 ~\exists q \in U, p \leq q \text{ s.t. } \mu(p) - \mu(q) < \epsilon \}$$

\begin{theorem} \label{innermeasureproperties}
Let $(D,\mu)$ be a valuation site. The inner measure $\mu_*$ is a locally finite and faithful measure on $L(D,\mu)$, and the Yoneda functor $$[-] : D \rightarrow \mathrm{Sh}(D,\mu;\mathbf{2})$$
is a valuation preserving lattice homomorphism preserving the bottom element.
\end{theorem}

\begin{proof} 
Since the inner topology contains the finite topology, the Yoneda functor $[-] : D \rightarrow \mathrm{Sh}(D,\mu;\mathbf{2})$ is a homomorphism of lower bounded distributive lattices. Moreover, it is clear that $\mu(p) = \mu_*([p])$. In particular, $\mu_*$ restricted to the image of $[-]$ is a finite valuation. It is also compatible with the induced Grothendieck pretopology, almost by definition. If $\{p_i \leq p ~|~ i \in I\}$ is a $\mu$-approximation, then we have
$$\mu_*([p]) = \mu(p) = \sup_{i \in I} \mu(p_i) = \sup_{i \in I} \mu_*([p_i]).$$
Applying Theorem \ref{valuationbasis} we see that $\mu_*$ restricted to the image of $D$ extends to a locally finite measure, whose value by definition agrees with $\mu_*$.

We still need to argue faithfulness. Let $U \leq V$ be two $\mu$-ideals such that
$$\mu_*(U \wedge K) = \mu_*(V \wedge K)$$
for any $\mu$-ideal $K$ with $\mu_*(K) < + \infty$. If this is the case, we may in particular choose for $p \in V$ the $\mu$-ideal $K = [p] \leq V$, and obtain
$$\mu_*(U \wedge [p]) = \sup \{ \mu(q \wedge p) ~|~ p \in U \} =  \mu(p).$$
But this means that $p$ can be approximated arbitrarily well from below by elements in $U$ and hence $p \in U$. We conclude that $U = V$.
\end{proof}

\begin{remark}
Since $\mathrm{Sh}(D,\mu;\mathbf{2})$ is a frame, it in particular has the top element $1$, given as the maximal ideal $D$. Thus the valuation $\mu_*$ is finite on $\mathrm{Sh}(D,\mu;\mathbf{2})$ iff
$$\mu_*(1) = \sup_{p \in D} ~\mu(p) < \infty,$$
which is the case iff the valuation $\mu$ on $D$ is bounded. A particular special case where this emerges is when $D$ is assumed to have a maximal element $1$, i.e.\ when $D$ is a bounded distributive lattice, and $\mu$ is a \emph{probability valuation} on $D$, i.e. $\mu(1) = 1$. In this case $\mu_*$ is again a probability measure on $\mathrm{Sh}(D,\mu;\mathbf{2})$.
\end{remark}

\begin{example}
An illustrative simple example is given by considering a set $S$, and taking $P = \mathcal{P}(S)^{\mathrm{fin}}$, the set of finite subsets of $S$, together with the counting measure, i.e.\
$$\mu(M) = \#M$$
for $M \subset S$ a finite subset. Then $\mathrm{Sh}(P,\mu;\mathbf{2})$ recovers the powerset $\mathcal{P}(S)$ of all subsets of $S$, and hence $L(P,\mu)$ is the discrete locale corresponding to the set $S$.
\end{example}

\begin{example}
Let $S$ be a set. The Stone-\v{C}ech compactification $\beta(S)$ of $S$ is obtained by equipping $\mathcal{P}(S)$ with the finite covering topology, and has as its points the set of ultrafilters on $S$ \cite[Example 3.14.]{lehner2025algebraicktheorycoherentspaces}. Let $\mathcal{U}$ be an ultrafilter on $S$, in other words a point of $\beta(S)$, which can be interpreted as a $\{0,1\}$-valued valuation $\mu_{\mathcal{U}}$ on $\mathcal{P}(S)$, see \cite[5]{goldbring2022ultrafilters}. Then
$\mathrm{Sh}(\mathcal{P}(S),\mu_{\mathcal{U}}, \mathbf{2}) \cong \mathbf{2}$ and the induced valuation preserving functor $$\mathcal{P}(S) \rightarrow \mathrm{Sh}(\mathcal{P}(S),\mu_{\mathcal{U}}, \mathbf{2}) \cong \mathbf{2}$$ agrees with the flat functor that induces the continuous map $\mathcal{U} : \mathrm{pt} \rightarrow \beta(S)$.
\end{example}

\begin{example}
More generally, for a lower bounded distributive lattice $D$, we have that $\mathrm{Sh}(D,\mu;\mathbf{2}) \cong \mathbf{2}$ for a valuation $\mu$ iff $\mu$ is a non-trivial, $\{0,1\}$-valued valuation on $D$. In this case, the natural functor $D \rightarrow \mathrm{Sh}(D,\mu;\mathbf{2}) \cong \mathbf{2}$ can be identified with $\mu$ and corresponds to the choice of a point of the locale given by the frame $\mathrm{Ind}(D)$.
\end{example}

\begin{example}
Consider the lower bounded distributive lattice $([0,\infty), \leq )$, excluding $\infty$. Then the identity is trivially a finite valuation. The corresponding inner locale is given by $\overrightarrow{[0,+\infty)}$, as defined in Example \ref{lowerreals}.
\end{example}

\begin{example} \label{Randomsequences} Consider the Cantor space $C = \prod_{n \in \mathbb{N}} \mathbf{2}$ as a model for infinite sequences of coin tosses, and consider the poset of compact opens $\mathcal{O}$ of $C$. It is obtained as
$$\mathcal{O} = \mathrm{colim}_{n \in \mathbb{N}} \mathcal{P}( \mathbf{2}^n ),$$
with the maps $\mathcal{P}( \mathbf{2}^n ) \rightarrow \mathcal{P}( \mathbf{2}^{n+1} )$ being given as $\pi_n^*$, for $\pi_n : \mathbf{2}^{n+1} \rightarrow \mathbf{2}^n$ being the projection onto the first $n$ factors. As a filtered colimit of Boolean algebras, it is itself a Boolean algebra. Equipping $\mathcal{P}( \mathbf{2}^{n} )$ with the uniform probability measure $\mu_{2^n}$, we see that the maps $\pi_n^* : \mathcal{P}( \mathbf{2}^n ) \rightarrow \mathcal{P}( \mathbf{2}^{n+1} )$ are valuation preserving, and we thus obtain a well defined, finite valuation $\mu$ on $\mathcal{O}$. The corresponding inner locale $\mathrm{Ran}(\mathbf{2}) = L(\mathcal{O}, \mu)$ is a sublocale of $C$, since $C$ is obtained from $\mathcal{O}$ by using the finite cover pretopology, and all finite covers are contained in the $\mu$-inner pretopology. It has been suggested by Simpson \cite{SIMPSON20121642} to call $\mathrm{Ran}(\mathbf{2})$ the \emph{locale of random sequences}.\footnote{The fact that $\mathrm{Ran}(\mathbf{2})$ actually agrees with Simpson's definition as the smallest sublocale of $C$ of measure $1$ can be seen by observing that the uniform measure on $C$ becomes faithful when restricted to $\mathrm{Ran}(\mathbf{2})$, compare with \cite[Proposition 6.3]{SIMPSON20121642}.} The locale $\mathrm{Ran}(\mathbf{2})$ is countably based, zero-dimensional, and has no points, and is what Simpson calls \emph{random}: It is equipped with the faithful probability measure $\mu_*$. It is however not Boolean, i.e.\ it does not define a measurable locale in the sense of Section \ref{measurablelocales}. (See the follow-up discussion in Example \ref{Randomsequences2} for an argument.)
\end{example}

In general, the construction $\mathrm{Sh}(D,\mu, \mathbf{2})$ behaves much like a $\mu$-weighted version of the Ind-completion of $D$, something that will be made more precise in Proposition \ref{filteredapproximation}. Thinking geometrically, we can also think of the represented locale as a $\mu$-generic point. The next proposition states that a certain converse to Theorem \ref{innermeasureproperties} holds.

\begin{proposition} \label{valuationframedeterminedbyfinite}
Suppose $L$ is a locale together with a locally finite and faithful measure $\mu$. Then $L \cong L(\mathcal{O}(L)^{fin}, \mu)$, where $\mathcal{O}(L)^{fin}$ is the sub-poset of opens $U$ of $L$ such that $\mu(U) < +\infty$.
\end{proposition}

\begin{proof} Write $F = \mathcal{O}(L)$ for the frame of opens of $L$ and $F^{fin} = \mathcal{O}(L)^{fin}$ for the subposet of opens of finite measure. By Lemma \ref{idealoffinitemeasure} $F^{fin}$ is an ideal on $F$ and in particular a lower bounded distributive lattice, and $\mu$ restricts to a finite valuation on $F^{fin}$. Furthermore, since $\mu$ is locally finite, we know that
$$1 = \bigvee_{U \text{ s.t. } \mu(U) < +\infty } U,$$
and thus any element $V$ of $F$ can be written as a supremum of elements in $F^{fin}$. This means that $F^{fin}$ is a basis for $F$ and $F$ can be obtained from $F^{fin}$ together with its induced Grothendieck topology, by Theorem \ref{basistheorem}. We only need to verify that this topology agrees with the $\mu$-inner topology.

The induced coverings on $F^{fin}$ can be split up into two cases:
\begin{itemize}
\item Finite covers: $\{U_i \subset U\}_{i \in I}$ cover if $I$ finite and $\bigvee U_i = U$.
\item Directed covers: $\{U_i \subset U\}_{i \in I}$ cover if $I$ filtered and $\bigvee U_i = U$.
\end{itemize}
Finite covers tautologically agree with finite covers in the $\mu$-topology. Now let $\{U_i \subset U\}_{i \in I}$ be a directed set. Suppose $\bigvee_{i \in I} U_i = U$. Then by continuity of $\mu$ we see that
$$\mu(U) = \mu(\bigvee_{i \in I} U_i) = \sup_{i \in I} \mu(U_i).$$
and hence $\{U_i \subset U\}_{i \in I}$ is a $\mu$-cover. Conversely, suppose
$$\sup_{i \in I} \mu(U_i) = \mu(U).$$
Since $\mu(U) < + \infty$, by faithfulness of $\mu$ we conclude that
$$\bigvee_{i \in I} U_i =  U.$$
Hence $\mu$-covers are also induced covers.
\end{proof}

We now elaborate on the countability condition that locales constructed from valuation sites naturally inherit: Any proposition with finite measure can be obtained by an ascending sequence of basic propositions.

\begin{proposition} \label{filteredapproximation}
Let $(D,\mu)$ be a valuation site, and let $U$ be a $\mu$-ideal with $\mu(U) < \infty$. Assume $U_i, i \in I$ is a directed collection of $\mu$-ideals such that $U_i \leq U$ for all $i \in I$. Then
$$\begin{array}{ccc}
\bigvee_{i \in I} U_i = U & \text{ iff } & \sup_{i \in I} \mu_*(U_i) = \mu_*(U). 
\end{array}$$
Furthermore, for any $\mu$-ideal $U$ with $\mu(U) < \infty$ there exists an ascending sequence $d_0 \leq d_1 \leq d_2 \leq \hdots$ in $D$ with $d_n \in U$ such that $U = \bigvee_{n \in \mathbb{N}} [d_n]$.
\end{proposition}

We remark that the latter statement places a strong countability condition on the measure frame $\mathrm{Sh}(D,\mu, \mathbf{2})$. We call a given sequence $d_0 \leq d_1 \leq d_2 \leq \hdots$ such that $U = \bigvee_{n \in \mathbb{N}} [d_n]$ an \emph{exhaustion} of $U$.

\begin{proof}
By continuity of $\mu_*$, it is clear that if $\bigvee_{i \in I} U_i = U$, then also $\sup_{i \in I} \mu_*(U_i) = \mu_*(U)$. We now prove the converse. Note that we always have
$$\bigvee_{i \in I} U_i \leq U.$$
Suppose $\alpha = \sup_{i \in I} \mu_*(U_i) = \mu_*(U)$, which means that
$$\alpha = \sup_{i \in I} \sup_{d \in U_i} \mu(d) = \sup_{d \in U} \mu(d).$$
Let $\epsilon > 0$ and choose $i \in I$ and $d \in U_i$ such that $ \alpha - \mu(d) < \epsilon$.

Now let $d' \in U$. Then 
$$\mu(d') - \mu(d' \wedge d) = \mu( d' \vee d ) - \mu( d ) \leq \alpha - \mu(d) < \epsilon.$$
But this means that $d'$ can be $\mu$-approximated by elements of $\bigvee_{i \in I} U_i$, hence $\bigvee_{i \in I} U_i = U.$

For the second statement, assume $U$ such that $\mu(U) < \infty$. By induction define a sequence $d_n$ of elements of $D$ smaller than $U$ such that $d_n \leq d_{n+1}$ for $n \in \mathbb{N}$:
\begin{itemize}
\item Set $\epsilon = 1$. Then there exists $d_1$ such that $$\mu_*(U) - \mu(d_1) = \sup_{d \in U} \mu(d) - \mu(d_1) < 1.$$
\item Assume by induction that we have constructed $d_n$ such that $\mu_*(U) - \mu(d_n) \leq \frac{1}{n}$. Choose $d'$ such that $\mu_*(U) - \mu(d') \leq \frac{1}{n+1}$. Define $d_{n+1} = d_n \vee d'$. We verify that
$$  \mu_*(U) - \mu(d_{n+1}) \leq \mu_*(U) - \mu(d') \leq \frac{1}{n+1}.$$
\end{itemize}
We now have that $\sup_{n \in \mathbb{N}} \mu( d_n ) = \mu(U)$, hence by the previous point we conclude that $ U = \bigvee_{n \in \mathbb{N}} [d_n]$.
\end{proof}

Let us add some useful computations in the frame $\mathrm{Sh}(D,\mu, \mathbf{2})$, viewed as a sublocale of $\mathrm{Idl}(D)$.

\begin{lemma} \label{implicationformulas}
Let $(D, \mu)$ be a valuation site, and $p \in D$. Then:
\begin{enumerate}
\item $[p] = \{ q \in D ~|~ \mu(q \wedge p) = \mu(q) \}.$
\item The zero object in $\mathrm{Sh}(D,\mu, \mathbf{2})$  is given by the null ideal $N = \{ q \in D ~|~ \mu(q) = 0\}$.
\item $y_p \rightarrow N = [p] \rightarrow N = \{ q \in D ~|~ \mu(q \wedge p) = 0 \}.$
\item For any ideal $J$ we have
$ J \rightarrow N = \{ q \in D ~|~ \mu( q \wedge p ) = 0 \text{ for all } p \in J \}.$
\end{enumerate}
\end{lemma}

\begin{proof}
Let us prove (1). Since it is clear that
$$\{ q \in D ~|~ \mu(q \wedge p) = \mu(q) \} \subset [p],$$
we only need to verify that the left-hand set is downward closed, and closed under finite unions and $\mu$-approximations, since it already contains $p$.

\emph{Downward closed:} Suppose $q$ is such that $\mu(q) = \mu( q \wedge p )$, and suppose $q' \leq q$. Then
$$\mu(q') - \mu( q' \wedge p ) \leq \mu(q) - \mu( q \wedge p ) = 0,$$
by Lemma \ref{valuationlemma1}, hence $\mu(q') = \mu( q' \wedge p )$.

\emph{Finite covers:} Suppose we have $q_1, q_2$ with $q_i$ such that $\mu(q_i) = \mu(q_i \wedge p )$. Then
$$\mu(q_1 \vee q_2) = \mu(q_1) + \mu(q_2) - \mu(q_1 \wedge q_2 ) = \mu(q_1 \wedge p) + \mu(q_2 \wedge p) - \mu(q_1 \wedge p \wedge q_2 \wedge p) = \mu( (q_1 \vee q_2) \wedge p ).$$

$\mu$-\emph{approximations:} Suppose $q_i, i \in I$ is a directed set with $q_i \leq q$ for all $i \in I$, and $\sup_{i \in I} \mu( q_i ) = \mu( q )$. We know that $\mu$-approximations are stable under intersection by Proposition \ref{mutopology} therefore $\sup_{i \in I} \mu( q_i \wedge p) = \mu( q \wedge p )$ as well. But then
$$\mu( q ) = \sup_{i \in I} \mu( q_i ) =  \sup_{i \in I} \mu( q_i \wedge p) = \mu( q \wedge p ).$$

The statement (2) is a special case of (1). To show (3), by adjunction the statement $[q] \leq ( y_p \rightarrow N )$ is equivalent to 
$$[q \wedge p] = [q] \wedge [p] \leq N,$$
which is the case iff $\mu(q \wedge p)$.

The statement (4) now follows from (3) by writing $J = \bigvee_{ p \in J} y_p$, and using
$$ J \rightarrow N = ( \bigvee_{ p \in J} y_p ) \rightarrow N = \bigwedge_{ p \in J}  y_p \rightarrow N = \{ q \in D ~|~ \mu( q \wedge p ) = 0 \text{ for all } p \in J \}.$$
\end{proof}

Recall that a Grothendieck pretopology $(D,\tau)$ is called subcanonical if all representables $y_p$ are propositional sheaves.

\begin{proposition} \label{faithfulsite}
Let $(D,\mu)$ be a valuation site. The $\mu$-inner topology on $D$ is subcanonical iff $\mu$ is faithful, i.e.\ $\mu : D \rightarrow [0,\infty)$ is a conservative functor.
\end{proposition}

\begin{proof}
The $\mu$-inner topology on $D$ is subcanonical iff for all $p \in D$ we have
$$\{ q ~|~ q \leq p \} = y_p = [p] = \{ q ~|~ \mu(q \wedge p) = \mu(q) \}.$$
Therefore, if $p \leq q$ such that $\mu(p) = \mu(q)$, we also have $q \leq p$, hence $\mu$ is conservative.

Conversely, assume $\mu$ is conservative. Then $\mu(q \wedge p) = \mu(q)$ implies $q \wedge p = q$, which is the case iff $q \leq p$.
\end{proof}

Combined with Lemma \ref{subcanonicalfaithful}, this means that the natural functor $[-] : D \rightarrow \mathrm{Sh}(D,\mu, \mathbf{2})$ is injective iff $\mu$ is faithful. If this is not the case, we can always enforce faithfulness of $\mu$
by taking a suitable quotient of $D$. Define a congruence relation on $D$ by
$$p \sim q ~\text{ iff }~ \mu(p) = \mu( p \wedge q ) = \mu(q)$$
for $p,q \in D$. We note that by modularity this implies $\mu(p \vee q) = \mu(p)$ as well.

\begin{lemma}
Let $(D, \mu)$ be a valuation site. The relation $\sim$ is in fact a congruence relation, $\mu$ descends to a faithful valuation on $D/\sim$. Furthermore, the induced quotient functor $(D, \mu) \rightarrow (D/\sim, \mu)$ is valuation preserving and induces an isomorphism 
$$\mathrm{Sh}(D,\mu, \mathbf{2}) \cong \mathrm{Sh}(D/\sim,\mu, \mathbf{2}).$$
\end{lemma}

\begin{proof}
We need to show that $\sim$ is a congruence relation. Let $p,q, r \in D$ and assume $p \sim q$. W.l.o.g.\ assume that $q \leq p$, otherwise repeat the argument with $q$ replaced by $p \wedge q$. Then
$$\mu( p \wedge r ) - \mu( q \wedge r ) \leq \mu( p ) - \mu( q ) = 0$$
hence $p \wedge r \sim q \wedge r$. Now we can use modularity to also deduce
$$\begin{array}{rcl}
\mu( p \vee r ) &=& \mu( p ) + \mu( r ) - \mu( p \wedge r ) \\
               = \mu( q ) + \mu (r ) - \mu (q \wedge r ) &=& \mu( q \vee r )
\end{array}
$$ hence also $p \vee r \sim q \vee r$. Finally, we need to check that the relation preserves $0$, i.e.\ the kernel $\{p \in D ~|~ p \sim 0 \}$ is an ideal. But $p \sim 0$ simply means $\mu(p) = 0$ and it is clear that the set of null elements is downward closed and closed under finite joins.

The statements that $\mu$ is well-defined and the quotient map $D \rightarrow D/\sim$ is valuation preserving are clear. For the last statement, note that the Yoneda functor factors as
$$[-] : D \rightarrow D/\sim \rightarrow \mathrm{Sh}(D,\mu, \mathbf{2}).$$
The induced right-hand functor is injective, therefore the claimed statement that
$$\mathrm{Sh}(D,\mu, \mathbf{2}) \cong \mathrm{Sh}(D/\sim,\mu, \mathbf{2})$$
follows from the Basis Theorem \ref{basistheorem}.
\end{proof}

One can also ask in what case one can recover $D$ entirely from the associated frame? Note that $[-]$ always has image in the set $\mathrm{Sh}(D,\mu, \mathbf{2})^{fin}$ of propositional sheaves of finite measure. However, while faithfulness of $\mu$ guarantees that the map 
$$[-] : D \rightarrow \mathrm{Sh}(D,\mu, \mathbf{2})^{fin}$$
is injective, it need not be surjective. Proposition \ref{filteredapproximation} gives that the target is given by the closure under suprema of ascending chains $d_1 \leq d_2 \leq \cdots $ in $D$ such that $\sup_{n \in \mathbb{N}} \mu(d_n) < \infty$. We can turn this condition into a definition.

\begin{definition}
A valuation site $(D, \mu)$ is called \emph{bounded $\sigma$-complete}, if whenever $d_1 \leq d_2 \leq \cdots $ in $D$ is a sequence such that $\sup_{n \in \mathbb{N}} \mu(d_n) < \infty$, the supremum $d = \bigvee d_i$ exists in $D$ and has valuation $\mu(d) = \sup_{n \in \mathbb{N}} \mu(d_n)$.
\end{definition}

\begin{corollary} \label{sigmacompleteness}
Let $(D,\mu)$ be a valuation site. The natural functor
$$[-] : D \rightarrow \mathrm{Sh}(D,\mu, \mathbf{2})^{fin}$$
is an isomorphism iff $(D,\mu)$ is faithful and bounded $\sigma$-complete.
\end{corollary}

\subsection{Functoriality}

Let us now investigate the functoriality of the construction of $(D, \mu) \mapsto \mathrm{Sh}(D,\mu, \mathbf{2})$.

\begin{lemma}
Let $\mathbf{f} : (D,\mu) \rightarrow (D',\mu')$ be a morphism of valuation sites. Then $\mathbf{f}$ is a partial morphisms of sites, with $P$ and $P'$ being equipped with the $\mu$-inner respectively $\mu'$-inner topology.
\end{lemma}

\begin{proof}
By definition $\mathbf{f}$ is a homomorphism of distributive lattices, and in particular preserves all existing finite limits. Since it is valuation-preserving, it also preserves $\mu$-coverings.
\end{proof}

Thus, any homomorphism of $\mathbf{f} : (D,\mu) \rightarrow (D',\mu')$ of valuation sites extends to a partial frame homomorphism $f^* : (D,\mu)^{inn} \rightarrow (D',\mu')^{inn}$, in the sense that the square
\[\begin{tikzcd}
	D & {D'} \\
	{\mathrm{Sh}(D,\mu, \mathbf{2})} & {\mathrm{Sh}(D',\mu', \mathbf{2})}
	\arrow["\mathbf{f}", from=1-1, to=1-2]
	\arrow[from=1-1, to=2-1]
	\arrow[from=1-2, to=2-2]
	\arrow["{f^*}", from=2-1, to=2-2]
\end{tikzcd}\]
commutes, by Theorem \ref{functorialitymorphismsites}. 

\begin{corollary}
There exists a functor
$$(-)^{inn} : \mathrm{ValSite} \rightarrow \mathrm{MeasFrm}$$
which sends a valuation site $(D,\mu)$ to $(D,\mu)^{inn} = \mathrm{Sh}(D,\mu, \mathbf{2})$.
\end{corollary}

\begin{proof} Let $\mathbf{f} : (D,\mu) \rightarrow (D',\mu')$ be a homomorphism of valuation sites. We only need to check that $f^* : \mathrm{Sh}(D,\mu, \mathbf{2}) \rightarrow \mathrm{Sh}(D',\mu, \mathbf{2})$ preserves the inner valuations. Recall that
$$f^*(U) = \bigvee_{p \in U} \mathbf{f}(p).$$ We compute that
$$\mu_*( f^*(U) ) = \sup_{p \in J} \mu(f(p)) =  \sup_{p \in U} \mu(p) = \mu_* (U).$$
\end{proof}

We also have a functor
$$(-)^{fin} : \mathrm{MeasFrm} \rightarrow \mathrm{ValSite}$$
which sends a frame with measure $\mu$ to its ideal of $\mu$-finite elements, see Lemma \ref{idealoffinitemeasure}. It is clear that a measure-preserving partial frame homomorphism $ f^* : F \rightarrow F'$ restricts to a valuation-preserving homomorphism of distributive lattices $\mathbf{f} : F^{fin} \rightarrow (F')^{fin}.$

\begin{theorem} \label{innervaluationadjunction}
There is an adjunction
\[\begin{tikzcd}
	{\mathrm{ValSite}} & {\mathrm{MeasFrm}}
	\arrow[""{name=0, anchor=center, inner sep=0}, "{(-)^{inn}}", curve={height=-12pt}, from=1-1, to=1-2]
	\arrow[""{name=1, anchor=center, inner sep=0}, "{(-)^{fin}}", curve={height=-12pt}, from=1-2, to=1-1]
	\arrow["\dashv"{anchor=center, rotate=-90}, draw=none, from=0, to=1]
\end{tikzcd}\]
with the counit of the adjunction for a measure frame $(F,\mu)$ given by 
$$ i_! : \mathrm{Sh}(F^{fin},\mu, \mathbf{2}) \rightarrow F $$
induced by left Kan extending the canonical inclusion $i : F^{fin} \rightarrow F$,
and unit of the adjunction for a valuation site $(D,\mu)$ given by the Yoneda functor
$$[-] : D \rightarrow  \mathrm{Sh}(D,\mu, \mathbf{2})^{fin}.$$
This adjunction is idempotent, inducing an equivalence of categories
\[\begin{tikzcd}
	{\mathrm{ValSite}_{\sigma , \mathrm{faithf}}} & {\mathrm{MeasFrm}_{\mathrm{l.f.~faithf}}}
	\arrow[""{name=0, anchor=center, inner sep=0}, "{(-)^{inn}}", curve={height=-12pt}, from=1-1, to=1-2]
	\arrow[""{name=1, anchor=center, inner sep=0}, "{(-)^{fin}}", curve={height=-12pt}, from=1-2, to=1-1]
	\arrow["\dashv"{anchor=center, rotate=-90}, draw=none, from=0, to=1]
\end{tikzcd}\]
where
\begin{itemize}
\item ${\mathrm{ValSite}_{\sigma , \mathrm{faithf}}}$ is the full subcategory of $\mathrm{ValSite}$ given by bounded $\sigma$-complete and faithful valuation sites, and
\item ${\mathrm{MeasFrm}_{\mathrm{l.f.~faithf}}}$ is the full subcategory of ${\mathrm{MeasFrm}}$ given by faithful and locally finite measure frames.
\end{itemize}
\end{theorem}

\begin{proof}
Let $f^* : \mathrm{Sh}(D,\mu, \mathbf{2}) \rightarrow (F, \nu)$ be a measure-preserving partial frame homomorphism, with $(D,\mu)$ a valuation site and $(F,\nu)$ a measure frame. Then $f^*$ is determined by the partial flat functor $ [f^*] = \mathbf{f} : D \rightarrow F$, which factors through $F^{fin}$ since $f^*$ is measure-preserving. In other words, we have a natural injection
$$\mathrm{Hom}_{\mathrm{MeasFrm}}( \mathrm{Sh}(D,\mu, \mathbf{2}), (F, \nu) ) \rightarrow \mathrm{Hom}_{\mathrm{ValSite}}( (D,\mu), (F^{fin}, \nu) ).$$
Surjectivity is provided by Theorem \ref{universalpropertysheaves}: A valuation preserving homomorphism of lower bounded lattices $\mathbf{f} : D \rightarrow F^{fin}$ gives a partial morphism of sites
$$ \mathbf{f} : D \rightarrow F^{fin} \hookrightarrow F,$$
which extends to partial frame homomorphism $$f^* : \mathrm{Sh}(D,\mu, \mathbf{2}) \rightarrow F$$. It is straightforward to verify that $f^*$ preserves valuations, as every propositional sheaf in $\mathrm{Sh}(D,\mu, \mathbf{2})$ is obtained as a directed supremum of basic propositions.

The fact that this adjunction is idempotent follows from the observation that $\mathrm{Sh}(D,\mu, \mathbf{2})$ is a faithful measure frame, and for faithful measure frame $(F,\mu)$, the counit
$$\mathrm{Sh}(F^{fin},\mu, \mathbf{2}) \rightarrow F$$
is an isomorphism, by Proposition \ref{valuationframedeterminedbyfinite}.

The statement about the unit being an equivalence iff $(D,\mu)$ is faithful and bounded $\sigma$-complete is taken care of by Corollary \ref{sigmacompleteness}.
\end{proof}

\section{Zero-dimensional and Boolean locales} \label{boolean}

A particular feature that separates measure/probability from point-set topology is that it is a very natural assumption that for any chunk/event a complement exists - Philosophically speaking, if one knows how to measure a chunk $U$ of some locale $X$, then the measure of the complement $U^c$ must be given as $\mu(X) - \mu(U)$ (assuming the measure of $X$ is finite) and is thus determined. A similar situation holds for probability theory. Therefore, one would expect (at least classical) measure/probability theory to work best in a locale theoretic context when combined with \emph{Booleanness}. In fact, the ability to work with complete Boolean algebras as topological objects is one of the selling points of the locale theoretic approach, something that is impossible in classical point-set topology. Let us begin with some remarks on Boolean locales.

\begin{definition}
Let $(D,\leq)$ be a bounded distributive lattice and $U \in D$ an element. A \emph{complement} of $U$ is an element $U^c$ such that
$$ U \wedge U^c = 0 \text{ and } U \vee U^c = 1. $$
\end{definition}

It is an elementary exercise to show that in a bounded distributive lattice complements, if they exist, are unique. A bounded distributive lattice is called a Boolean algebra if every element has a complement. \\

Recall that in a frame $F$, we already have the notion of Heyting implication $U \rightarrow V$, as defined in Section \ref{Heytingimplication}.

\begin{definition}
Let $L$ be a locale and $U$ an open. The \emph{negation} of $U$ is defined as
$$ \neg U = U \rightarrow 0.$$
\end{definition}

We note that by definition, the negation $\neg U$ of $U$ is the largest open such that $\neg U \wedge U = 0$. However, we do not in general have that $\neg U \vee U = 1$. The following lemma is a straightforward verification.

\begin{lemma}
Let $L$ be a locale and $U$ an open. Then $U$ is complement iff $U \vee \neq U = 1$. In particular the complement of $U$, if it exists is given by $\neg U$.
\end{lemma}

\begin{definition}
Let $L$ be a locale. We call an open $U$ of $L$ \emph{closed and open}, or abbreviated \emph{clopen}, if $U$ has a complement. 
\end{definition}

The locale of sublocales $\mathfrak{Sl}(L)$ together with its continuous map $\mathrm{can} : \mathfrak{Sl}(L) \rightarrow L$ has a universal property with respect to maps $f : M \rightarrow L$ such that $f^*(U)$ is clopen for all opens $U$ of $L$.

\begin{theorem}[\cite{picado_pultr}, Proposition 6.3.1.] \label{lifttosublocales2} 
Let $f : M \rightarrow L$ be a map of locales such that $f^*(U)$ is complemented in $\mathcal{O}(M)$ for each open $U$ of $L$. Then there exists a unique lift $\tilde{f} : M \rightarrow \mathfrak{Sl}(L)$ along $\mathrm{can}$.
\end{theorem}

Of course, the condition on $f : M \rightarrow L$ necessary for Theorem \ref{lifttosublocales2} is automatic if every open of $M$ is already clopen. We call this condition \emph{Boolean}.

\begin{definition}
Let $L$ be a locale.
\begin{itemize}
\item $L$ is called \emph{zero-dimensional}, or also \emph{totally disconnected}, if it has a basis of clopens. 
\item $L$ is called \emph{Boolean}, if every open is clopen.
\end{itemize}
We will use the same terminology when talking about the associated frame. Denote by $\mathrm{BoolLoc} \subset \mathrm{Loc}$ the full subcategory spanned by Boolean locales.
\end{definition}

\begin{example}
A natural example of a Boolean locale is given by $S^{disc}$, obtained by equipping a set $S$ with the discrete topology, or equivalently by considering the frame $\mathcal{P}(S)$ given by the power set of $S$. More generally, for any topological space $X$, we have the natural continuous map $X^{disc} \rightarrow X$. By Theorem \ref{lifttosublocales2} this map lifts along $\mathrm{can} : \mathfrak{Sl}(L) \rightarrow L$. We discuss this in Remark \ref{remarksubspaces}.
\end{example}

\begin{example}
The locale $\mathfrak{Sl}(L)$ for a locale $L$ is always zero-dimensional, but usually not Boolean \cite[VI 4]{picado_pultr}. The basis of clopens is given by sublocales of the form $U^c \vee V$, for $U$ and $V$ open. One verifies that for an open sublocale $U$, the complement is indeed given by the closed sublocale $U^c$. 
\end{example}

\begin{remark} By definition, Boolean locales correspond to complete Boolean algebras, viewed as frames. The ability to view complete Boolean algebras as space-like objects is unique to locale theory and does not exist when one works in classical point-set topology: A Boolean locale is spatial if and only if it is atomic, in which case the corresponding locale simply corresponds to a discrete set. Boolean locales can be thought of as generalized discrete spaces. They share many separation properties discrete spaces also satisfy. A remarkable fact is that the classical Stone-\v{C}ech compactification of a discrete set has a natural extension to all locales, given by sending a frame $F$ to the frame $\mathrm{Ind}(F)$. When restricted to complete Boolean algebras, this functor produces an equivalence between the category of Boolean locales and \emph{Stonean spaces}, or in other words extremally disconnected compact Hausdorff spaces.\footnote{This category has seen a recent surge in interest due to work on condensed/pyknotic mathematics, see \cite{ScholzeCondensed} and \cite{barwick2019pyknoticobjectsibasic}.} Stonean spaces which are not given as Stone-\v{C}ech compactifications of discrete sets are often viewed as mysterious, or also superfluous. When viewing complete Boolean locales as natural extensions of discrete locales, this mystery vanishes somewhat and it becomes true that any Stonean space \emph{is} a Stone-\v{C}ech compactification. We remark that Stonean spaces are not themselves Boolean, unless they are finite, a common point of confusion.
\end{remark}

\begin{remark}
A classical example of a totally disconnected locale that is not Boolean is given by a pro-finite set $S$, also called Stone space, such as e.g. $\mathbb{N} \cup \{\infty\}$, the Cantor set $C$, and the Stone-\v{C}ech compactification $\beta(M)$ for a set $M$. In fact, there is a classical Stone duality between (not necessarily complete!) Boolean algebras and pro-finite spaces, see \cite{johnstone1982stone}, also \cite[Section 3.2]{lehner2025algebraicktheorycoherentspaces}, under which a pro-finite set corresponds to its Boolean algebra of compact open subsets, and which extends the mentioned Stone duality between complete Boolean algebras and Stonean spaces mentioned in the previous remark. However, just as with Stonean spaces, pro-finite sets only give examples of Boolean locales if they correspond to finite sets.
\end{remark}

Since complements are unique, it is clear that a frame $F$ is Boolean iff $\neg : F^{op} \rightarrow F$ is an isomorphism. This self-duality implies the following.

\begin{lemma}
A Boolean frame is automatically a coframe. A Boolean coframe is automatically a frame.
\end{lemma}

An \emph{atom} in a Boolean frame is defined to be a minimal non-zero element. This gives a direct characterization of the set of points of a Boolean locale.

\begin{lemma} \label{pointsareminimalelements}
Let $L$ be a Boolean locale. Then there is a bijection
$$\begin{array}{rcl}
\mathrm{pts}(L) & \cong & \mathrm{Atom}(\mathcal{O}(L)) \\
P & \mapsto & \bigwedge P
\end{array} $$
\end{lemma}

\begin{proof}
Let $x \in \mathcal{O}(L)$ be a minimal non-zero open. Then $P = \{ U \in \mathcal{O}(L) ~|~ x \leq U \}$ is clearly a proper filter. It is completely prime: If $x \leq \bigvee_{i \in I} U_i$ we have
$$ 0 \neq x = \bigvee_{i \in I} x \wedge U_i. $$
Minimality of $x$ means that all the elements $x \wedge U_i$ are either $0$ or $x$, therefore at least $x \wedge U_i \neq 0$ for at least one $i \in I$.

Conversely, let $P$ be a completely prime filter. Define $x = \bigwedge P$. We claim that $x \in P$, which would imply that $P$ is principal, generated by $x$. Assume not. Since $x \vee \neg x = 1 \in P$, we then would have $\bigvee_{U \in P} \neg U = \neg x \in P$. Since $P$ is completely prime, we therefore find $ U \in P$ such that $\neg U \in P$. But then $ 0 = U \wedge \neg U \in P$, which cannot be.
\end{proof}

Furthermore, the description of implication in a Boolean locale simplifies.

\begin{lemma} \label{booleanimplication}
Let $L$ be a Boolean locale and $U,V$ opens. Then $U \rightarrow V = \neg U \vee V$.
\end{lemma}

\begin{proof}
We claim that the functors $U \rightarrow - $ and $\neg U \vee - $ agree, for which we equivalently need to show that
$$ U \wedge W \leq V ~\text{ iff }~ W \leq \neg U \vee V$$
for all $V, W$ open. First assume $ U \wedge W \leq V $ holds. Apply $\neg U \vee -$ to this inequality to obtain
$$ W \leq \neg U \vee W = \neg U \vee ( U \wedge W ) \leq \neg U \vee V. $$
Now for the converse assume $W \leq \neg U \vee V$ holds. Apply $ U \wedge - $ to obtain
$$ U \wedge W \leq U \wedge (\neg U \vee V) = U \wedge V \leq V. $$
Hence both statements are equivalent.
\end{proof}

\begin{lemma}
Any sublocale of a Boolean locale is Boolean.
\end{lemma}

\begin{proof}
Let $i_* : L \rightarrow M$ be an embedding, with $M$ Boolean. Let $ U \in L$. Then $i_*(U)$ has a complement $V$ in $M$, i.e. $i_*(U) \wedge V = 0$ and
$i_*(U) \vee V = 1$. Applying $i^*$ and using $i^* i_*(U) = U$ gives that $i^*(V)$ is the complement to $U$.
\end{proof}

Perhaps surprisingly, Boolean locales can be generated fairly easily. As an example, to every locale $L$ one can associate the \emph{smallest dense sublocale} $L_{\neg \neg} \hookrightarrow L$ given by
$$\{ \neg U ~|~ U \in F \} \subset F$$
where $F$ is the frame of $L$, see \cite[III.8.3]{picado_pultr}.\footnote{Also compare with Theorem 3 in \cite{maclane1992sheaves} in case of topoi.} To see that this is in fact a sublocale, use Lemma \ref{sublocales} and Proposition \ref{superheytingprop}:\begin{itemize}
\item $\{ \neg U ~|~ U \in F \}$ is closed under arbitrary infima, as
$$ \bigwedge_{ i \in I } ( U_i \rightarrow 0 ) = ( \bigvee_{i \in I} U_i ) \rightarrow 0. $$
\item $\{ \neg U ~|~ U \in F \}$ is closed under implication, as
$$ V \rightarrow ( U \rightarrow 0 ) = ( V \wedge U ) \rightarrow 0.$$
\end{itemize}
The associated nucleus is given by $U \mapsto \neg \neg U$. We note that this means in particular that $\neg \neg \neg U = \neg U$ for any open $U$. Also, suprema of opens in $L_{\neg \neg}$ are computed by applying $\neg \neg$ to their suprema in $L$. The sublocale $L_{\neg \neg}$ is in fact Boolean, which can be verified directly.
\begin{itemize}
\item Applying $\neg \neg$ to $ \neg U \vee U$ we get that
$$ \neg \neg( \neg U \vee U  ) = \neg ( \neg U \wedge \neg \neg U ) = \neg( 0 ) = 1. $$
\item We also have
$\neg \neg U \wedge \neg U = 0$
by definition.
\end{itemize}
Hence $\neg \neg U$ is the complement of $\neg U$ in $L_{\neg \neg}$. (But not necessarily in $L$!)

We can generalize this example quite a bit. Recall that given any open $N$, one can associate the inclusion of the closed complement $N^c \hookrightarrow L$, see Example \ref{closedsublocale} given on the level of frames as 
$$\{ V \in F ~|~ N \leq V \} \subset F,$$
which treats $N$ as the new bottom element. We can compose these two inclusions to get the sublocale $$b(N) = (N^c)_{\neg \neg} \hookrightarrow L.$$ Note that the frame for $b(N)$ is simply given as
$$\{ U \rightarrow N ~|~ U \in F \} \subset F,$$
because by Proposition \ref{superheytingprop} (4), any element of the form $U \rightarrow N$ automatically lives above $N$. Since any sublocale needs to be closed under implication, we also see that $b(N)$ is equivalently the smallest sublocale containing $N$. In fact, this type of example is the only example for a Boolean sublocale.

\begin{proposition}[\cite{picado_pultr}, III.10.4] \label{booleansublocalescharacterization}
Let $L$ be a locale and $N$ an open in $L$. Then $b(N) \rightarrow L$ is a Boolean sublocale of $L$. Conversely, any Boolean sublocale of $L$ is of the form $b(N)$ for some open $N$.
\end{proposition}

\begin{proof} Let us reproduce the argument presented in \cite{picado_pultr} for completeness. We have seen that $b(N)$ is a Boolean sublocale. For the converse, suppose $S \subset F$ describes a Boolean sublocale of $L$. Let $N = \bigwedge S$, which is an element of $S$ by closure under infima. It is clear that $b(N) \subset S$, since $S$ is closed under implication. Conversely, if $U \in S$, since $S$ is Boolean we have $U = (U \rightarrow N) \rightarrow N$, and so we see that $U \in b(N)$.
\end{proof}

If $L$ is in particular a Boolean locale, then every sublocale is of the form $b(N) \hookrightarrow L$ for some open $N$. But we can identify $b(N) = \{ U \rightarrow N \} = \{ \neg U \vee N \}$
with the closed sublocale given by $N^c$.

\begin{corollary}
Suppose $L$ is a Boolean locale. Then all sublocales of $L$ are given by open, or equivalently closed, inclusions, i.e.\ the natural map 
$\mathfrak{Sl}(L) \rightarrow L$ is an isomorphism.
\end{corollary}

\begin{lemma} \label{openpullbackdoublenegation}
Let $L$ be a locale and $U, N$ opens of $L$. Then the square
$$\begin{tikzcd}
b(U \wedge N) \ar[r, hook] \ar[d, hook] \arrow[dr, phantom, "\scalebox{1.0}{$\lrcorner$}", very near start, color=black] & b(N) \ar[d, hook] \\
U \ar[r, hook] & L
\end{tikzcd}$$
is a pullback square of embeddings.
\end{lemma}

\begin{proof}
By using that 
$$\begin{tikzcd}
(U \wedge N)^c \ar[r, hook] \ar[d, hook] \arrow[dr, phantom, "\scalebox{1.0}{$\lrcorner$}", very near start, color=black] & N^c \ar[d, hook] \\
U \ar[r, hook] & L
\end{tikzcd}$$
is a pullback square, we can reduce the proof to the case of $N = 0$. Since pullbacks of open embeddings are open embeddings, we have the pushout square of frames
$$\begin{tikzcd}
\mathcal{O}(L_{\neg \neg})_{ \neg \neg U }   & \mathcal{O}(L_{\neg \neg}) \ar[l] \\
\mathcal{O}(L)_ U  \ar[u]^{ i^* } & \mathcal{O}(L) \ar[l] \ar[u]^{ j^* }
\end{tikzcd}$$
Then $\mathcal{O}(L_{\neg \neg})_{ \neg \neg U }$ is a sublocale of the Boolean locale $L_{\neg \neg}$, hence it is itself Boolean, and therefore of the form $b(i_*(0))$. The functor $i^*$ is given by the restriction of $j^* = \neg \neg$ and hence $i_* = j_* \wedge U$. But then $i_*(0) = j_*(0) \wedge U = 0 \wedge U = 0$, and therefore $\mathcal{O}(L_{\neg \neg})_{ \neg \neg U } \cong \mathcal{O}(U_{\neg \neg})$, which is what we wanted to show.
\end{proof}

Let us now come to maps between Boolean locales.

\begin{lemma} \label{allmapsopen}
Let $f : L \rightarrow M$ be a continuous map between Boolean locales. Then $f$ is open.
\end{lemma}

\begin{proof}
First observe that $f^*$ preserves negation: Let $U$ be an open in $M$. Applying $f^*$ to the equations
$$\begin{array}{rl}
U \vee \neg U = 1 & \text{ and} \\
U \wedge \neg U = 0
\end{array}$$
we get
$$\begin{array}{rl}
f^*(U) \vee f^*(\neg U) = 1 & \text{ and} \\
f^*(U) \wedge f^*(\neg U) = 0
\end{array}$$
and hence $f^*(\neg U) = \neg f^*(U)$. We can now verify that $f^*$ preserves infima and implications.
$$f^*( \bigwedge_{i \in I} U_i ) = f^*( \neg \bigvee_{i \in I} \neg U_i ) = \neg \bigvee_{i \in I} \neg f^*( U_i ) = \bigwedge_{i \in I} f^*( U_i ),$$
and
$$f^*( U \rightarrow V ) = f^*( \neg U \vee  V ) = \neg f^*(U) \vee f^*(V) = f^*(U) \rightarrow f^*(V),$$
by Lemma \ref{booleanimplication}.
\end{proof}

\begin{lemma} \label{openmapsdoublenegation}
Let $f : L \rightarrow M$ be an open map between locales. Then $f$ descends to a well-defined map on double negation sublocales, i.e.\ there is a commuting square
\[\begin{tikzcd}
	L & M \\
	{L_{\neg \neg}} & {M_{\neg \neg}}.
	\arrow["f", from=1-1, to=1-2]
	\arrow[hook, from=2-1, to=1-1]
	\arrow["\tilde{f}", from=2-1, to=2-2]
	\arrow[hook, from=2-2, to=1-2]
\end{tikzcd}\]
If this is the case, we have $\tilde{f}^* = \neg \neg f^*$.
\end{lemma}

\begin{proof}
We need to show that $\neg \neg f^* \neg \neg =  \neg \neg f^*$. Since $f$ is open, the functor $f^*$ preserves implication and $0$, hence in particular negation. Therefore we have
$$\neg \neg f^* \neg \neg =  \neg \neg \neg \neg f^* = \neg \neg f^*, $$
since triple negation is the same as negating once.
\end{proof}

\begin{remark}
The condition for a map $f : L \rightarrow M$ to descend to double negation sublocales is slightly weaker than being open and called \emph{skeletal}. Skeletal maps were studied originally in \cite{Banaschewski1994}. See also \cite{JOHNSTONE2006240}, in particular Section 3.
\end{remark}

\begin{theorem} \label{booleanadjunction}
There exists an adjunction
\[\begin{tikzcd}
	{\mathrm{Loc}_{\mathrm{open}}} & {\mathrm{BoolLoc}}
	\arrow[""{name=0, anchor=center, inner sep=0}, "{(-)_{\neg \neg}}"', curve={height=12pt}, from=1-1, to=1-2]
	\arrow[""{name=1, anchor=center, inner sep=0}, "{}"', curve={height=12pt}, hook', from=1-2, to=1-1]
	\arrow["\dashv"{anchor=center, rotate=-90}, draw=none, from=1, to=0]
\end{tikzcd}\]
with the left adjoint given by the inclusion of the full subcategory of Boolean locales, and the right adjoint being given by $L \mapsto (L)_{\neg \neg} \hookrightarrow L$.
\end{theorem}

\begin{proof}
The left adjoint is well-defined by Lemma \ref{allmapsopen}. The adjunction follows directly from Lemma \ref{openmapsdoublenegation}.
\end{proof}

We now try to integrate markings into the picture. Recall the definition of a marked locale $(L,U)$ given in Definition \ref{markedlocales}, as a locale $L$ together with a choice of open $U$.

\begin{corollary} \label{doublenegationrelativefactoring}
Let $f : L \rightarrow M$ be an open map between locales and $N$ an open of $M$. Then $f$ induces a well-defined map
$$\tilde{f} : b(f^*(N)) \rightarrow b(N).$$
The inverse image part is given by the formula $\tilde{f}^*(U) = f^*((U \rightarrow N) \rightarrow N) = (f^*(U) \rightarrow f^*(N)) \rightarrow f^*(N)$.
\end{corollary}

\begin{proof} The claim follows at once by combining Lemma \ref{closedsublocaleuniversalopen}, to get a well-defined open restriction $f| : f^*(N^c) \rightarrow N^c$, and Lemma \ref{openmapsdoublenegation}, to descend further to double negation sublocales. The claimed formula follows from the observation that negation on the closed sublocale $f^*(N)^c$ is given by $- \rightarrow f^*(N)$.
\end{proof}

We summarize the above statement into the existence of a useful functor for constructing Boolean locales.

\begin{corollary} \label{booleanmarkedadjunction}
There exists an adjunction
\[\begin{tikzcd}
	{\mathrm{MarkLoc}_{\mathrm{open}}} & {\mathrm{BoolLoc}}
	\arrow[""{name=0, anchor=center, inner sep=0}, "b"', curve={height=12pt}, from=1-1, to=1-2]
	\arrow[""{name=1, anchor=center, inner sep=0}, "{(-,0)}"', curve={height=12pt}, hook', from=1-2, to=1-1]
	\arrow["\dashv"{anchor=center, rotate=-90}, draw=none, from=1, to=0]
\end{tikzcd}\]
with the left adjoint $L \mapsto (L,0)$ obtained by equipping a Boolean locale with the marking given by the bottom element $0$, and the right adjoint being given by $(L,U) \mapsto b(U) = (U^c)_{\neg \neg} \hookrightarrow L$. Furthermore, the left adjoint is fully faithful.
\end{corollary}

\begin{proof} The adjunction is obtained by composing the adjunction given by Theorem \ref{markingadjunction} with the adjunction provided by Theorem \ref{booleanadjunction}
Fully faithfulness of the left adjoint follows because it is obtained by a composition of fully faithful left adjoints.
\end{proof}

\subsection{The analogy between topological spaces and matrix groups}

The following is a side remark that will not play a role in the rest of the text, other than (potentially) aid the reader in intuition. Since the usage of locales might seem foreign to a reader educated in classical point-set topology, we want to highlight an analogy between the theory of groups and the theory of locales. While it is not true that the functor $\mathrm{Loc}: \mathrm{Top} \rightarrow \mathrm{Loc}$ that associates to a topological space $X$ its corresponding frame of opens is fully faithful, it \emph{is} true that the functor $\mathrm{Top} \rightarrow \mathrm{Loc}^\rightarrow$ into the category of maps between locales, that sends a topological space $X$ to the map
$$\mathrm{Loc}(X^\mathrm{disc}) \rightarrow \mathrm{Loc}(X)$$
is fully faithful. In other words the information of a topological space can be recovered from its set of opens, together with the information on how it embeds into the frame $\mathcal{O}(X^\mathrm{disc}) = \mathcal{P}(X)$.
The functor $\mathcal{O}(X) \rightarrow \mathcal{P}(X)$ is by definition an injective frame homomorphism, or equivalently a quotient map of locales.

This is analogous to how one would define a \emph{matrix group}: As a subgroup $G \subset \mathrm{GL}_n(k)$. If, under this analogy, we think of general linear groups as analogous to Boolean locales, we get the following table of corresponding terms.

\begin{center}
\begin{tabular}{ |c|c| } 
 \hline
 Group theory & Topology \\
 \hline
 Group $G$ & Locale $L$ \\
 \hline
 Presentation of $G \cong \left\langle S ~|~ R \right\rangle$ & $\mathcal{O}(L) \cong \mathrm{Sh}(P,\tau; \mathbf{2})$ \\
 via generators and relations & for Grothendieck pretopology $(P,\tau)$ \\
 \hline
 Representation $G \rightarrow \mathrm{GL}_n(k)$ & Map $M \rightarrow L$ with $M$ Boolean \\
 \hline
 Zero representation $G \rightarrow  0$ & Inclusion of empty locale $\emptyset \rightarrow L$ \\ 
 \hline
 One-dimensional representation $G \rightarrow k^\times$ & Point $\mathrm{pt} \rightarrow L$ \\
 \hline
 Faithful representation $G \hookrightarrow \mathrm{GL}_n(k)$ & Quotient map $M \twoheadrightarrow L$ \\
 \hline
 Matrix group $G \subset \mathrm{GL}_n(k)$ & Topological space $X^{disc} \twoheadrightarrow X$ \\
 \hline
 Regular representation $G \rightarrow k[G]^\times$ & $\mathfrak{Sl}(L)_{\neg \neg} \twoheadrightarrow L$ \\
 \hline
\end{tabular}
\end{center}

In the upcoming sections, particularly Section \ref{radonvaluations}, we use nice enough valuations on a class of spaces to construct maps $X^\mu \rightarrow X$, with $X^\mu$ being a Boolean locale. This can be thought of as giving a \emph{different representation} of $X$, which encapsulates the measure theoretic rather than point-set topological aspects of $X$.

\section{Measurable locales}  \label{measurablelocales}

In this section we will describe a class of locales which have the right properties for classical measure and integration theory: Measurable locales. These are defined to be Boolean locales that allow enough measures to distinguish opens. The category of measurable locales is anti-equivalent via a variant of Gelfand duality to the category of commutative von Neumann algebras, which is a major reason for the usefulness of the point-free approach to measure theory.

In this section, we will use the word \emph{chunk} instead of open. Readers with an eye towards probability theory may also want to read the word \emph{chunk} as \emph{event}.

\begin{definition} \label{measurablelocale}
A Boolean locale $L$ is called \emph{measurable} if one of the following equivalent statements holds:
\begin{enumerate}
\item There exists a locally finite and faithful measure $\mu$ on $L$.
\item $1$ is the supremum over chunks $a$ such that $a$ admits a finite faithful measure.
\item For any $b \in L, b \neq 0$ there exists a finite measure $\mu$ on $L$ such that $\mu(b) \neq 0$.
\end{enumerate}
We define $\mathrm{MblLoc} \hookrightarrow \mathrm{Loc}$ as the full subcategory spanned by measurable locales.
\end{definition}

We note that the statement that either of (2) or (3) imply (1) is non-trivial and requires the axiom of choice. Before we give the proof of the equivalence of the three conditions, we record two easy lemmata about measures on Boolean locales.

\begin{lemma}
Let $B$ be a complete Boolean algebra. A locally finite measure $\mu$ on $B$ is faithful iff the condition
$$\mu(U) = 0 \text{ iff } U = 0$$
holds.
\end{lemma}

\begin{proof}
Suppose $\mu$ is faithful, and let $\mu(U) = 0$. Since $0 \leq U$ are both opens with finite measure, we conclude $U = 0$. Conversely, suppose 
$$\mu(U) = 0 \text{ iff } U = 0$$
holds. Now let $V \leq V'$ with $\mu(V) = \mu(V') < + \infty$. By Booleanness we have
$$V \vee ( V' \setminus V ) = V'$$
where $V' \setminus V = V' \wedge \neg V$. But then by modularity we have $\mu( V' \setminus V ) = 0$, therefore $V' \setminus V = 0$, and hence $V = V'$. Now let $V \leq V'$ with
$$\mu(V \wedge K) = \mu(V' \wedge K).$$
Use continuity of $\mu$ together with local finiteness, we can reduce to the case of finite measure and conclude $V = V'$.
\end{proof}

\begin{remark}
A basic observation is that for a Boolean locale $L$ with faithful measure $\mu$, all points $x$ of $L$ (identified with atoms of $\mathcal{O}(L)$ under Lemma \ref{pointsareminimalelements}) need to satisfy $\mu(x) > 0$. This gives a direct criterion that shows that many useful measurable locales in practice cannot be spatial.
\end{remark}

The presence of Booleanness also allows one to deduce statements about co-continuity of measures from continuity.

\begin{lemma} \label{infformulaboolean}
Let $L$ be a Boolean locale and $\mu$ a measure on $L$. Let $U_i, i \in I,$ be a downwards directed system of chunks such that there exists $i_0 \in I$ such that $\mu(U_{i_0}) < \infty$. Then
$$\mu( \bigwedge U_i ) = \inf_{i \in I} \mu( U_i ).$$
\end{lemma}

\begin{proof} Write $m = \mu(U_{i_0}) < \infty$. Using directedness, we may assume w.l.o.g.\ that $U_i \leq U_{i_0}$ for all $i \in I$, and therefore $\mu(U_i) \leq m $ for all $i \in I$. Then compute
$$ \mu( \bigwedge_{i \in I} U_i ) + \mu( \bigvee_{i \in I} U_{i_0} \setminus U_i ) =  m = \mu( U_i ) + \mu( U_{i_0} \setminus U_i )$$
for all $i \in I$, therefore 
$$ \mu( \bigwedge_{i \in I} U_i ) + \mu( \bigvee_{i \in I} U_{i_0} \setminus U_i ) = \inf_{i \in I}  \mu( U_i ) +  \sup_{i \in I} \mu( U_{i_0} \setminus U_i ).$$
Hence $\mu( \bigwedge U_i ) = \inf_{i \in I} \mu( U_i )$ follows from continuity of $\mu$.
\end{proof}

\begin{theorem} \label{allequivalent}
The three conditions stated in Definition \ref{measurablelocale} are equivalent.
\end{theorem}

\begin{proof}
Condition (1) implies (2): This is clear, as local finiteness of $\mu$ means that $1$ can be obtained as a supremum of chunks $b$ of finite measure, and $\mu$ restricts to a finite faithful measure on $b$. \\

\noindent Condition (2) implies (1):
Consider the poset $P$ given by pairs $(a,\mu_a)$, where $a$ is a chunk of $L$, and $\mu_a$ is a locally finite and faithful measure on $a$, with $(a,\mu_a) \leq (a',\mu_a')$ if $a \leq a'$ and the restriction of $\mu_a'$ to $a$ agrees with $\mu_a$. This poset is non-empty, as otherwise $1$ would not be the supremum of all $a$ such that $a$ admit a finite faithful measure. Let $(a_i, \mu_{a_i})_{i \in I}$ be a chain in $P$. Define $a = \bigvee_{i \in I} a_i$, and $\mu_a(b) = \sup_{i \in I} \mu_{a_i}( b \wedge a_i )$. Then $\mu_a$ is again a locally finite and faithful measure, by Corollary \ref{measuresaresheaffiltered}. Now use Zorn's Lemma to obtain a maximal element $(a, \mu_a)$ of $P$. Suppose $a \neq 1$. Since $1$ is obtained as the supremum of $b$ such that $b$ admits a finite faithful measure, the same is true for $\neg a \neq 0$. (Here we have used Booleanness.) Choose $b \leq \neg a, b \neq 0$ together with a finite faithful measure $\mu_b$ on $b$. Then $(a \vee b, \mu)$ with $\mu(c) = \mu_a( c \wedge a) + \mu_b( c \wedge b)$ for $c \leq a \vee b$ satisfies $(a, \mu_a) \leq (a \vee b, \mu)$ in $P$, a contradiction to maximality. \\

\noindent Condition (2) and (3) are equivalent: This is treated in \cite[Lemma 2.53]{PAVLOV2022106884}.
\end{proof}

Aside from the obvious analogy with measurable spaces, another justification for calling these locales measurable comes from the following theorem. Given a locale $L$, we define the ring of bounded functions
$$C_b(L;\mathbb{C}) = \mathrm{colim}_{r \rightarrow + \infty} \mathrm{Map}(L, B_r(0)),$$
where $B_r(0) \subset \mathbb{C}$ is the disc of radius $r > 0$ centered at the origin. We obtain a functor
$$ C_b : \mathrm{Loc}^{op} \rightarrow *\mathrm{CAlg},$$
where $*\mathrm{CAlg}$ is the category of commutative $*$-algebras over $\mathbb{C}$. We call the restriction of this functor to the category of measurable locales
$$L^\infty : \mathrm{MblLoc}^{op} \rightarrow *\mathrm{CAlg}.$$
We remark on the difference to classical measure theory: Whereas for a measure space $(X,\mathcal{L},\mu)$, an element of $L^\infty(X)$ is given as an \emph{equivalence class} of measurable, essentially bounded functions, in the point-free setup an element of $L^\infty$ \emph{is} a bounded continuous function.

\begin{theorem}[\cite{PAVLOV2022106884}] \label{gelfanddualityvonneumann}
The functor $L^\infty$ induces an equivalence
$$L^\infty : \mathrm{MblLoc}^{op} \simeq \mathrm{CVNA},$$
where $\mathrm{CVNA}$ is the category of commutative von Neumann algebras and normal $*$-morphisms. The inverse functor sends a commutative von Neumann algebra to its locale of projections.
\end{theorem}

This means already developed tools in measure theory can be readily imported to the world of measurable locales.

\begin{remark} It seems plausible that the inclusion 
$$\mathrm{MblLoc} \rightarrow \mathrm{Loc}$$
has a right adjoint
$$(-)_{\mathrm{Bor}} : \mathrm{Loc} \rightarrow \mathrm{MblLoc}$$
which would be suitably referred to as the Borel locale of a locale. This would stand in contrast to the case of the inclusion $\mathrm{BoolLoc} \rightarrow \mathrm{Loc}$. (The non-existence of such a right adjoint is discussed in \cite[57]{johnstone1982stone}.) However, it is not clear to the author at the time of writing how the functor $(-)_{\mathrm{Bor}}$ would be described explicitly.
\end{remark}

\begin{remark}
Any commutative von Neumann algebra $\mathcal{A}$ is of the form $L^\infty(\Gamma, \mu)$ for some locally compact space $\Gamma$ together with Radon measure $\mu$, see e.g.\ \cite[V Theorem 1.18.]{Takesaki1979}. We will independently reprove this fact, under the above Gelfand-type duality, by showing that any measurable locale is obtained from a regular content on a locally compact Hausdorff space, see Theorem \ref{representationtheorem}.
\end{remark}

Given a measurable locale $M$, and a locale $L$, and an $L$-valued random variable on $M$ is just defined as a continuous function $M \rightarrow L$.\footnote{In practice this matches the standard definition of random variable up to almost everywhere equivalence.}

Since we can use valuation sites to generate locally finite and faithful measures, an important question is when the frame $\mathrm{Sh}(D,\mu; \mathbf{2})$ generated from a valuation site $(D,\mu)$ is actually a Boolean frame, in which case the corresponding locale is a measurable locale. This leads us to the notion of an \emph{almost Boolean} valuation site.

\begin{definition} \label{almostboolean}
Let $(D,\mu)$ be a valuation site.
\begin{itemize}
\item We call $(D,\mu)$ \emph{almost disconnected}, if for any $c_0 \leq c \in D$, and $\epsilon > 0$ there exists $d \leq c$ such that
$$\begin{array}{l}
 \mu( d \wedge c_0 ) = 0, \text{ and}  \\
 \mu( c ) - \mu( d \vee c_0 ) < \epsilon. 
\end{array}$$
\item We call $(D,\mu)$ \emph{almost Boolean}, if for any $c \in D$ and for any ascending sequence $c_0 \leq c_1 \leq \hdots \leq  c$ of elements contained in $c$, and for any $\epsilon > 0$ there exists $N \in \mathbb{N}$ and $d \leq c$ such that
$$\begin{array}{l}
\mu( d \wedge c_n ) = 0  \text{ for all } n \in \mathbb{N}, \text{ and} \\ 
\mu( c ) - \mu( d \vee c_N ) < \epsilon.
\end{array}$$
\end{itemize}
\end{definition}

Intuitively, we think of the element $d$ as an approximate complement relative to $c$ to the element $c_0$, respectively to the subobject represented by the union of $c_n, n \in \mathbb{N}$.

\begin{example}
The set $P = P(S)^{\mathrm{fin}}$ of finite subsets of a given set $S$, together with the counting measure, is almost Boolean. To see this, let $A$ be a finite subset of $S$ and $B_n \subset A, n \in \mathbb{N}$ an ascending sequence of subsets. Since $A$ is finite, there exists a number $n_0 \in \mathbb{N}$ where the maximum $B = B_{n_0}$ is achieved. Now set $C = A \setminus B$ as the almost complement.
\end{example}

\begin{example}
Let $(L, \delta_x)$ be a locale with a point $x$, equipped with the Dirac measure. Then $\mathcal{O}(L)$ is almost Boolean. To see this, let $U$ be an open of $L$. If $x \not\in U$, there is nothing to show. Assume $x \in U$. If $F_n \subset U, n \in \mathbb{N}$ is an ascending chain, then $x \in \bigvee_{n \in \mathbb{N}} F_n$ iff $x \in F_{n_0}$ for some $n_0 \in \mathbb{N}$. If this is the case, take $C = 0$. Otherwise, $x \not\in F_n$ for all $n \in \mathbb{N}$, in which case it suffices to take $C = U$ as the approximate complement.
\end{example}

\begin{example}
The poset $([0,\infty), \leq)$ equipped with the tautological valuation is not almost disconnected.
\end{example}

\begin{example}
The Sierpinski space $\mathbb{S}$ given by the frame $\mathcal{O}(\mathbb{S}) = \{0 \leq U \leq 1 \}$ together with the valuation $\mu_p$ defined via $\mu_p(U) = p$ and $\mu_p(1)=1$ is an example of a valuation site that is \emph{not} almost disconnected for $p < 1$.
\end{example}

\begin{example} Consider a valuation $\mu$ on a Boolean algebra $B$. Then $(B,\mu)$ is almost disconnected, since for any $c \leq d$, we can take $d \setminus c = d \wedge \neg c$ as the relative complement. However, in general $(B,\mu)$ is not almost Boolean, as can be seen from the following example.
\end{example}

\begin{example} \label{Randomsequences2} We continue with Example \ref{Randomsequences}, given by Cantor space $C$. The poset of compact opens of $C$,
$$\mathcal{O} = \mathrm{colim}_{n \in \mathbb{N}} \mathcal{P}( \mathbf{2}^n )$$
is a Boolean algebra. Therefore $(\mathcal{O}, \mu)$, where $\mu$ is the uniform valuation, is an example of an almost disconnected valuation site. However, it is not almost Boolean, as seen by the following counterexample, which is a combinatorial version of (the open complement of) the classical Smith–Volterra–Cantor set, also lovingly referred to as \emph{Fat Cantor set}.

First of all, by cofinality we can write 
$$\mathcal{O} \cong \mathrm{colim}_{n \in \mathbb{N}} \mathcal{P}( \mathbf{2}^{n^2} ).$$
We have the projection $p_n : \mathbf{2}^{n^2} \rightarrow \mathbf{2}^{(n-1)^2}$. Note that $p_n$ is measure-preserving. We will give a sequence of subsets $A_n \in \mathcal{P}( \mathbf{2}^{n^2} )$ constructed inductively.
\begin{itemize}
\item The set $A_1$ is chosen to be a one-element subset of $\mathbf{2}^{1^2} = \mathbf{2}$. We have $\mu(A_1) = \frac{1}{2}$.
\item Define $A_n$ as a union of $p_{n}^{-1}(A_{n-1}) \subset \mathbf{2}^{n^2}$ together with a choice of elements $x_y \in \mathbf{2}^{n^2}$ for each of the sets $p_{n}^{-1}({y})$ for $y \not\in A_{n-1}$. A calculation yields
$$\mu(A_n) = \mu(A_{n-1}) + \frac{1-\mu(A_{n-1})}{2^{2n-1}}.$$
\end{itemize}
The sequence $(A_n, n \in \mathbb{N})$ can be identified with an ascending sequence in $\mathcal{O}$. We verify by estimation that
$$\sup_{n\in \mathbb{N}} \mu(A_n) \leq \sum_{n=1}^\infty \frac{1}{2^{2 n - 1}} = \frac{2}{3} < 1.$$
However, there cannot exist an $\epsilon$-complement for this sequence for any $\epsilon$. Assume $B$ is an element, such that $\mu(B \cap A_n) = 0$ for all $n \in \mathbb{N}$. The set $B$ must live in some finite stage $\mathcal{P}( \mathbf{2}^{k^2} )$ for some $k \in \mathbb{N}$. By assumption, we have $B \cap A_k = \emptyset$. But, assuming that $B$ is non-empty, we must have $p_{k+1}^{-1}(B) \cap A_{k+1} \neq \emptyset$, as the intersection contains at least one element by construction of $A_{k+1}$. Therefore $\mu( p_{k+1}^{-1}(B) \cap A_{k+1} ) > 0$, a contradiction.
\end{example}

\begin{remark} The existence of Fat Cantor sets hints at a deeper result about the potential structure of measurable locales. As a consequence of Maharam's theorem \cite[Ch. 33]{fremlin2000measure}, there cannot exist a non-discrete measurable locale with a countable basis. This also implies indirectly that the locale $\mathrm{Ran}(\mathbf{2})$ of random sequences cannot be Boolean, as we've seen explicitly from the computation above.
\end{remark}

\begin{proposition}
Let $(D,\mu)$ be an almost disconnected valuation site. Then for every $c \in D$, the propositional sheaf $[c] \in \mathrm{Sh}(D, \mu; \mathbf{2})$ is clopen. In particular, the $\mu$-inner locale $L(D,\mu)$ is zero-dimensional.
\end{proposition}

\begin{proof}
Let $e \in D$. We need to show that $[e]$ has a complement, which is the case iff
$$ 1 = [e] \vee ( [e] \rightarrow N ).$$
where $1 = D$ is the maximal $\mu$-ideal. Recall that $[e] \rightarrow N = \{ d ~|~ \mu(d \wedge e) = 0 \}$ by Lemma \ref{implicationformulas}. In other words, we need to show that for any $c \in D$, we have
$$[c] = [c \wedge e] \vee ( [c] \wedge ( [e] \rightarrow N ) ).$$
Denote $c_0 = c \wedge e$. Then the above holds iff for all $\epsilon > 0$ we find $d \leq c$ such that $\mu(d \wedge e ) = \mu( d \wedge c_0) = 0$ and $\mu(c) - \mu(d \vee c_0) < \epsilon$. But this is just the condition for $(D, \mu)$ to be almost disconnected.
\end{proof}

Thus we see that almost disconnectedness of $(D,\mu)$ leads to the locale $L(D,\mu)$ being zero-dimensional. In order to guarantee that it is Boolean, we need the stronger condition of $(D,\mu)$ being almost Boolean.

\begin{theorem} \label{almostbooleangivesboolean}
Let $(D,\mu)$ be a valuation site. Let $N \subset D$ be the ideal of null sets. The following are equivalent:
\begin{enumerate}
\item $(D,\mu)$ is almost Boolean.
\item $\mathrm{Sh}(D,\mu; \mathbf{2})$ is a Boolean frame.
\item $L(D,\mu)$ is a measurable locale.
\item $Sh(D,\mu; \mathbf{2}) = b(N) \subset \mathrm{Sh}(D,fin; \mathbf{2}) \cong \mathrm{Ind}(D)$.
\end{enumerate}
\end{theorem}

\begin{proof}
The equivalence of (2) and (3) is by definition. The equivalence of (2) and (4) is an immediate consequence from Proposition \ref{booleansublocalescharacterization}. 

Assume (1), i.e.\ that $(D,\mu)$ is almost Boolean. Let $I \subset D$ be a $\mu$-ideal. We need to show that $I$ has a complement, which is the case iff
$$ 1 = I \vee ( I \rightarrow N ).$$
Keeping in mind that $1 = D$ is the maximal $\mu$-ideal, this is equivalent to the statement:

For all $ c \in C, \epsilon > 0$ there exist elements $k, d \leq c$ such that:
\begin{itemize}
\item $k \in I.$
\item $d$ satisfies $\mu(d \wedge k' ) = 0 \text{ for all } k' \in I$.
\item $\mu(c) - \mu( d \vee k ) < \epsilon.$
\end{itemize}

Let us prove this claim. Let $c \in C$ and $\epsilon > 0$. Since $\mu_*( [c] \wedge I ) \leq \mu( c ) < \infty$, we can use Proposition \ref{filteredapproximation} to get an exhaustion $c_1 \leq c_2 \leq \hdots $ such that 
$$[c] \wedge I = \bigvee_{n \in \mathbb{N}} [c_n].$$

Using that $(D,\mu)$ is almost Boolean, we can find $m \in \mathbb{N}$ and $d \in D$ such that
$$\mu( d \wedge c_n ) = 0 \text{ for all } n \in \mathbb{N}, \text{ and } \mu( c ) -  \mu( c_m \vee d ) < \epsilon.$$

Setting $k = c_m$ we are almost done with the argument. We are left to show that $\mu( d \wedge k' ) = 0$ for all $k' \in I$, or in other words $\mu_*( [d] \wedge I ) = 0$. This follows as $d \leq c$ by
$$[d] \wedge I = [d] \wedge [c] \wedge I = \bigvee_{n \in \mathbb{N}} [d] \wedge [c_n] = \bigvee_{n \in \mathbb{N}} N = N. $$

To show (2) implies (1), let $c \in C$, $c_0 \leq c_1 \leq \hdots \leq  c$ and $\epsilon > 0$. Define $I = \bigvee_{n \in \mathbb{N}} [c_n]$. By construction $I \leq [c]$. Since $Sh(D,\mu; \mathbf{2})$ is assumed to be Boolean, we have 
$$ [c] = I \vee ( (I \rightarrow N) \wedge [c] ) = I \vee \{ d \in D ~|~\forall n \in \mathbb{N} ~ \mu(d \wedge c_n) = 0  \text{ and } \mu(d) = \mu(d \wedge c) \}$$
where we used the description of $I \rightarrow N$ provided by Lemma \ref{implicationformulas}.

This can only be the case if $c$ itself has a $\mu$-approximation by elements from the right-hand side, which is equivalent to $(D,\mu)$ being almost Boolean.
\end{proof}

\begin{example}
A major source of examples of almost Boolean valuation sites is given by the notion of a \emph{regular content} $\lambda$ on the set $\mathcal{K}(X)$ of compact subsets of a Hausdorff space $X$, as discussed later with Theorem \ref{regularcontentalmostboolean}. A particular classical example to keep in mind is the $d$-dimensional Lebesgue measure $\lambda_d$ on $\mathbb{R}^d$ when restricted to the set $\mathcal{K}^d = \mathcal{K}(\mathbb{R}^d)$ of compact subsets of $\mathbb{R}^d$, which we elaborate on in Section \ref{lebesguereals}. This produces the measurable locale of \emph{Lebesgue reals} $\mathbb{R}^d_{Leb} = L( \mathcal{K}^d, \mu)$. The corresponding complete Boolean algebra $\mathcal{O}(\mathbb{R}^d_{Leb})$ is also called \emph{random algebra}, and more classically obtained by modding out the Boolean algebra of Lebesgue measurable sets by the ideal of $\lambda_d$-null sets. (We will prove this isomorphism later in Theorem \ref{radonmeasures}.) The random algebra was originally studied by von Neumann, see \cite[p.253, Example 2]{von1998continuous}.
\end{example}

\begin{example} Keeping with the theme of measures on $\mathbb{R}^d$, one may instead of $(\mathcal{K}^d,\mu)$ look at the sub-valuation site $(\mathcal{J}^d,\mu)$ spanned by finite unions of \emph{boxes}, where a \emph{box} is a subset of the form $[a, b] \subset \mathbb{R}^d$ for $a,b \in \mathbb{R}^d$ and
$$[a, b] = \{ x \in \mathbb{R}^d ~|~ a_i \leq x_i \leq b_i \text{ for all } i = 1, \hdots, d \}.$$
The locale $\mathbb{R}^d_{Jor} = L(\mathcal{J}^d,\mu)$ might be suitably called the \emph{Jordan reals} as it captures the essence of the so called \emph{Jordan measure}, see \cite[ Section 1.1.2.]{Tao2011-TAOAIT-2}, and comes with a continuous, measure-preserving map $\mathbb{R}^d_{Jor} \rightarrow \mathbb{R}^d$. Every inclusion $D \subset D'$ of finite unions of boxes admits an $\epsilon$-complement, however $(\mathcal{J}^d,\mu)$ is not almost Boolean. As an example, consider the construction of a Smith-Volterra-Cantor set, also referred to as Fat Cantor set, as for example described in \cite[141]{aliprantis1998principles}, for some $\epsilon > 0$. This is obtained as the complement of a countable union of ascending open sets $U_n \subset [0,1]$, each being given by a finite union of open intervals,  such that for $U = \bigcup_{n \in \mathbb{N}} U_n$ we have $\mu(U) = 1- \epsilon < 1$, yet the intersection of any interval $[a,b]$ with $U$ will always have positive measure. Simply doing the same construction with closed intervals provides an ascending sequence of sets $A_n$, each consisting of a finite union of closed intervals, that cannot have approximate complements. 
\end{example}

\section{The locally coherent space associated to a Hausdorff space}

An important example of a measure is that of a Radon measure on a Hausdorff space. Before we begin with the measure-theoretic aspects, we will take care of some categorical constructions that are possible for Hausdorff spaces in this section. Since Hausdorff spaces are automatically sober \cite[Ch. I 1.2]{picado_pultr}, we obtain a fully faithful embedding
$$ L : \mathrm{HausSpc} \hookrightarrow \mathrm{Loc}$$
Henceforth we will identify a Hausdorff topological space with its corresponding locale. Hausdorffness of a topological space $X$ guarantees that the poset $\mathcal{K}(X)$ of compact subsets of $X$ interacts favourably with the locale corresponding to $X$. Let us state a lemma that will be useful later on.

\begin{lemma}[Separation lemma, \cite{cohn1994measure} 7.1.2] \label{compactseparationlemma}
Let $X$ be a Hausdorff space. Assume $K_1, K_2$ are two compact subsets of $X$ such that $K_1 \cap K_2 = \emptyset$. There exist $U_1, U_2$ open sets of $X$ such that $U_1 \cap U_2 = \emptyset$ and $K_1 \subset U_1$ as well as $K_2 \subset U_2$.
\end{lemma}

Since Hausdorffness implies that intersections of compact subsets are still compact we see that $\mathcal{K}(X)$ is a lower bounded distributive lattice. We can equip it with the \emph{finite covering} topology, where
$$\{ K_i \subset K ~|~ i \in I \}$$
is a covering if $I$ is finite and $\bigcup_{i \in I} K_i = K$.

\begin{definition}
Let $X$ be a Hausdorff space. Define $X^\mathcal{K}$ to be the locale corresponding to the $0$-site $\mathcal{K}(X)$ equipped with the finite covering topology $fin$.
\end{definition}

\begin{remark}Recall that this means concretely that
$$\mathcal{O}(X^\mathcal{K}) \cong \mathrm{Idl}( \mathcal{K}(X) ) \cong \mathrm{Ind}( \mathcal{K}(X) ).$$
The locale $X^\mathcal{K}$ is an example of a locally coherent locale. This means it is spatial and corresponds to a locally compact, but not necessarily Hausdorff, space. For more detail, see \cite{lehner2025algebraicktheorycoherentspaces}. 
\end{remark}

For any topological space $X$, there is a functorial assignment
$$\begin{array}{rcl}
\theta_X^* : \mathcal{O}(X)  &\rightarrow & \mathcal{O}(X^\mathcal{K}) \\
U & \mapsto & \{ K \subset U ~|~ K \text{ compact} \}
\end{array}$$
It is clear that $\theta^*$ maps every open $U$ to a $fin$-proposition, and that $\theta^*$ preserves finite meets.

\begin{proposition} \label{framehomforlocallycompacthausdorff}
Let $X$ be a Hausdorff space. Then $\theta_X^*$ is a frame homomorphism. If $X$ is locally compact then the resulting map of locales
$$ \theta : X^\mathcal{K} \rightarrow X$$
is an epimorphism of locales.
\end{proposition}

We remark that a morphism $f : X \rightarrow Y$ of locales is an epimorphism iff $f^* : \mathcal{O}(Y) \rightarrow \mathcal{O}(X)$ is injective, see e.g.\ \cite[IV]{picado_pultr}. For the proof of this proposition we will need an elementary lemma about compact subsets in Hausdorff spaces.

\begin{lemma}[\cite{cohn1994measure}, 7.1.10] \label{compactdecompositionlemma}
Let $X$ be a Hausdorff space. Assume $K$ is a compact subset and $K \subset U_1 \cup \hdots \cup U_n$ with $U_i$ being open sets for $i = 1, \hdots, n$. Then $K = K_1 \cup \cdots \cup K_n$ with $K_i$ compact and $K_i \subset U_i$ for all $i = 1, \hdots, n$.
\end{lemma}


\begin{proof}[Proof of Proposition \ref{framehomforlocallycompacthausdorff}]
We need to show that $\theta^*$ preserves suprema. Let $U = \bigcup_{i \in I} U_i$ be an arbitrary union of open subsets of $X$. Then
$$\theta^*( U ) = \{ K \subset U ~|~ K \text{ compact} \} = \{ K ~|~ K \text{ compact and } K \subset \bigcup_{i \in F} U_i, ~F \subset I \text{ finite} \},$$
whereas
$$\bigvee_{i \in I} \theta^*( U_i ) = \{ \bigcup_{i \in F}K_{i} ~|~ F \subset I \text{ finite}, K_i \subset U_i \text{ for all } i \in F \}.$$
These two sets agree for $X$ Hausdorff because of Lemma \ref{compactdecompositionlemma}.

To verify that $\theta^*$ induces an epimorphism of frames in the case that $X$ is locally compact, we need to argue that it is injective. Assume
$$\theta^*( U ) = \{ K \subset U ~|~ K \text{ compact} \} = \{ K \subset V ~|~ K \text{ compact} \} = \theta^*( V ).$$
But then $U$ and $V$ contain the same points, as for any $x \in U$ by local compactness we can simply choose a compact neighborhood of $x$ contained in $U$ and vice versa. Therefore $U = V$.
\end{proof}

\begin{remark}
The injectivity of $\theta^*$ in the case of a locally compact Hausdorff space $X$ is a localic analogue of the known fact that the $\infty$-category of $K$-sheaves can be identified with the $\infty$-category sheaves on $X$, as discussed e.g.\ in \cite[Section 7.3.4]{luriehtt} or \cite[Proposition 6.5.]{efimov2025ktheorylocalizinginvariantslarge}. In the locale theoretic language, if $X$ is locally compact Hausdorff, then a propositional sheaf $\mathcal{J} \in \mathcal{O}(X^\mathcal{K})$ is of the form $\mathcal{J} = \theta^*(U)$ iff it satisfies the condition
$$ K \in \mathcal{J} \text{ iff } \exists K' \in \mathcal{J} : K \ll K', $$
where $K \ll K'$ means that there exists an open $V \subset X$ such that $K \subset V \subset K'$. The corresponding open set $U$ is then given as $U = \bigcup_{V \subset K, K \in \mathcal{J}} V$.
\end{remark}

\begin{remark} \label{pullbackofopens} The map $\theta : X^\mathcal{K} \rightarrow X$ is compatible with open inclusions, in the sense that whenever $U \subset X$ is an open subset, we have the pullback square
\[\begin{tikzcd}
	{U^\mathcal{K}} & {X^\mathcal{K}} \\
	U & X.
	\arrow[from=1-1, to=1-2]
	\arrow[from=1-1, to=2-1]
	\arrow["\lrcorner"{anchor=center, pos=0.125}, draw=none, from=1-1, to=2-2]
	\arrow[from=1-2, to=2-2]
	\arrow[from=2-1, to=2-2]
\end{tikzcd}\]
To see this note that pullback along $\theta$ of $U$ gives the open sublocale given by $\theta^*(U)$ of $X^\mathcal{K}$ by Lemma \ref{preimageofopen}. Observe that
$$ \mathcal{O}(X^\mathcal{K})_{ / \theta^*(U) } = \mathrm{Idl}( \mathcal{K}(X) )_{ / \theta^*(U) } \cong \mathrm{Idl}( \mathcal{K}(U) ) = \mathcal{O}(U^\mathcal{K})$$
since an ideal $\mathcal{J} \subset \theta^*(U) = \{ K \subset U ~|~ K \text{ compact} \}$ is by definition just an ideal of compact sets contained in $U$.
\end{remark}

\subsection{Functoriality with respect to  partial proper maps}

We now analyse the functoriality of the assignment $X \mapsto X^\mathcal{K}$. There are two distinct cases we would like to separate. 
\begin{itemize}
\item The case where $f : X \rightarrow Y$ is partially defined with open support, and \emph{proper}, that is $f^{-1}$ preserves compact sets. This case will allow us to construct measure-preserving partial maps.
\item The case where $f : X \rightarrow Y$ is continuous, with global support. This case induces an \emph{open} map $f^\mathcal{K} : X^\mathcal{K} \rightarrow Y^\mathcal{K}$, which will be useful in constructing maps between Boolean locales without reference to a choice of measure.
\end{itemize}

We will analyse the first case in this section.

\begin{definition}
A partial continuous map $f : X \rightarrow Y$ between Hausdorff spaces $X,Y$ with open support is called \emph{proper} if $f^{-1}(K) \subset X$ is compact for every compact subset $K \subset Y$. 
\end{definition}

\begin{theorem} \label{Hausdorfffunctorialityproper}
There exists a functor
$$\begin{array}{rcl}
(-)^\mathcal{K} : \mathrm{HausSpc}_{\mathrm{part}~\mathrm{prop}}  &\rightarrow& \mathrm{Loc}_{\mathrm{part}} \\
X &\mapsto& X^\mathcal{K}
\end{array}$$
whereby a partial proper map $f : X \rightarrow Y$ is sent to the partial map $f^\mathcal{K} : X^\mathcal{K} \rightarrow Y^\mathcal{K}$ obtained by extending the partial morphism of sites $f^{-1} : \mathcal{K}(Y) \rightarrow \mathcal{K}(X)$. Moreover, we have a commuting square
\[\begin{tikzcd}
	{X^\mathcal{K}} & {Y^\mathcal{K}} \\
	X & Y,
	\arrow["{f^\mathcal{K}}", from=1-1, to=1-2]
	\arrow["{\theta_X}"', from=1-1, to=2-1]
	\arrow["{\theta_Y}", from=1-2, to=2-2]
	\arrow["f", from=2-1, to=2-2]
\end{tikzcd}\]
in other words, $\theta$ is a natural transformation $(-)^\mathcal{K} \rightarrow L$, where $L : \mathrm{HausSpc} \rightarrow \mathrm{Loc}$ is the fully faithful inclusion of Hausdorff spaces into locales.
\end{theorem}

\begin{proof}
The functoriality of $(-)^\mathcal{K}$ is immediate, as $f^{-1} : \mathcal{K}(Y) \rightarrow \mathcal{K}(X)$ preserves finite unions and intersections and is thus a partial morphism of sites. We obtain the extension from Theorem \ref{functorialitymorphismsites}.

We still need to show that $\theta$ is a natural transformation. Consider the square
\[\begin{tikzcd}
	{\mathcal{O}(X^\mathcal{K})} & {\mathcal{O}(Y^\mathcal{K})} \\
	{\mathcal{O}(X)} & {\mathcal{O}(Y).}
	\arrow["{(f^\mathcal{K})^*}"', from=1-2, to=1-1]
	\arrow["{\theta^*}", from=2-1, to=1-1]
	\arrow["{\theta^*}"', from=2-2, to=1-2]
	\arrow["{f^{-1}}", from=2-2, to=2-1]
\end{tikzcd}\]
We want to show that this square commutes. Let $V \subset Y$ be open. Tracing through both composites we get
$$\begin{array}{r}
\theta(f^{-1}(V)) = \{ K \subset f^{-1}(V) \text{ compact} \} = \{ K \subset \mathrm{dom}(f) \subset X \text{ compact} ~|~ f(K) \subset V \} \\
= \{ K \subset X \text{ compact} ~|~ f(K) \subset C \subset V, C \text{ compact} \} \\ = \{ K \subset X \text{ compact} ~|~ K \subset f^{-1}(C), C \subset V \text{ compact} \} \\ = (f^\mathcal{K})^*(\theta(V))
\end{array}$$
where we note that $f(K)$ is well-defined, as $K \subset f^{-1}(V) \subset \mathrm{dom}(f)$, hence we conclude the proof.
\end{proof}

\begin{remark}
The class of partially defined proper maps with open support may seem unusal to some readers. The choice is fairly natural however at least in the case of locally compact Hausdorff spaces, as there exist equivalences
$$ \mathrm{LocCHaus}_{\mathrm{part}~\mathrm{prop}} \simeq \mathrm{CHaus}_* \simeq \mathrm{Comm}C^*\mathrm{-Alg}^{op}$$
with $\mathrm{CHaus}_*$ being the category of pointed compact Hausdorff spaces and base-point preserving continuous maps, and $\mathrm{Comm}C^*\mathrm{-Alg}$ being the category of \emph{non-unital} commutative $C^*$-algebras, with the equivalences given by one-point compactification and Gelfand duality, respectively. (See e.g.\ \cite[\nopp IV 4.1]{johnstone1982stone}.) Hence partially defined proper maps simply correspond to $*$-homomorphisms.
\end{remark}

\subsection{Functoriality with respect to globally defined continuous map}

Let $f : X \rightarrow Y$ be a globally defined, continuous map between Hausdorff spaces. Since $f$ maps compact sets to compact sets, an assignment which preserves finite unions, we have an associated left adjoint functor $f^\mathcal{K}_! :  \mathcal{O}(X^\mathcal{K}) \rightarrow \mathcal{O}(Y^\mathcal{K})$, given by left Kan extending the assignment $K \mapsto f(K)$ for $K \subset X$ compact. We claim this is the direct image part of an open map $f^\mathcal{K} : X^\mathcal{K} \rightarrow Y^\mathcal{K}$. 

\begin{theorem} \label{Hausdorfffunctoriality}
There exists a functor
$$\begin{array}{rcl}
(-)^\mathcal{K} : \mathrm{HausSpc} &\rightarrow& \mathrm{Loc}_{\mathrm{open}} \\
X &\mapsto& X^\mathcal{K}
\end{array}$$
whereby a map $f : X \rightarrow Y$ is sent to the open map $f^\mathcal{K} : X^\mathcal{K} \rightarrow Y^\mathcal{K}$ with:
\begin{itemize}
\item The direct image part $f^\mathcal{K}_! :  \mathcal{O}(X^\mathcal{K}) \rightarrow \mathcal{O}(Y^\mathcal{K})$ is obtained by extending the covering preserving functor $\mathcal{K}(X) \rightarrow \mathcal{O}(Y^\mathcal{K}), K \mapsto [f(K)]$.
\item The inverse image part $(f^\mathcal{K})^* :  \mathcal{O}(Y^\mathcal{K}) \rightarrow \mathcal{O}(X^\mathcal{K})$ is obtained by extending the flat functor $\mathcal{K}(Y) \rightarrow \mathcal{O}(X^\mathcal{K}), C \mapsto \{ K \subset f^{-1}(C) \}$.
\end{itemize}
Moreover, we have a commuting square
\[\begin{tikzcd}
	{X^\mathcal{K}} & {Y^\mathcal{K}} \\
	X & Y,
	\arrow["{f^\mathcal{K}}", from=1-1, to=1-2]
	\arrow["{\theta_X}"', from=1-1, to=2-1]
	\arrow["{\theta_Y}", from=1-2, to=2-2]
	\arrow["f", from=2-1, to=2-2]
\end{tikzcd}\]
in other words, $\theta$ is a natural transformation $U(-)^\mathcal{K} \rightarrow L$, where $U : \mathrm{Loc}_{\mathrm{open}} \rightarrow \mathrm{Loc}$ is the forget functor, and $L : \mathrm{HausSpc} \rightarrow \mathrm{Loc}$ is the fully faithful inclusion of Hausdorff spaces into locales.
\end{theorem}

\begin{proof} First note that both $f^\mathcal{K}_!$ and $(f^\mathcal{K})^*$ are left adjoints. This is clear for $f^\mathcal{K}!$, as $K \mapsto f(K)$ preserves finite unions. To observe the same for $(f^\mathcal{K})^*$, note that for $C_1, C_2 \subset Y$ compact we have
$$\begin{array}{rcccl}
f^*(C_1 \cup C_2) &=& \{ K \subset f^{-1}(C_1 \cup C_2) \} &=& \{ K \subset f^{-1}(C_1) \cup f^{-1}(C_2) \} \\
&=& \{ K_1 \cup K_2 ~|~ K_1 \subset f^{-1}(C_1), K_2 \subset f^{-1}(C_2) \} &=& f^*(C_1) \vee f^*(C_2).
\end{array}$$
For the middle equality, note that any $K \subset f^{-1}(C_1) \cup f^{-1}(C_2)$ can be written as
$$ K = ( K \cap f^{-1}(C_1) ) \cup ( K \cap f^{-1}(C_2) ).$$
Note that Hausdorffness of $X$ is crucial here, as $f^{-1}(C_1)$ is in general not compact, but always closed and $K \cap f^{-1}(C_1)$ a closed subset of a compact set, therefore compact.

We next establish that $f^\mathcal{K}_!$ is in fact left adjoint to $(f^\mathcal{K})^*$. We note that this already implies that $(f^\mathcal{K})^*$ is a frame homomorphism, as it is then both a left and a right adjoint, therefore preserves both arbitrary suprema as well as infima. Let $U \in \mathcal{O}(X^\mathcal{K})$ and $V \in \mathcal{O}(Y^\mathcal{K})$. We need to show
$$f^\mathcal{K}_!(U) \leq V ~\text{ iff }~ U \leq (f^\mathcal{K})^*( V ).$$
\begin{enumerate}
\item We have $f^\mathcal{K}_!(U) = \bigvee_{K \in U} [f(K)]$, therefore
$$f^\mathcal{K}_!(U) \leq V ~\text{ iff }~ \forall K \in U : f(K) \in V$$
\item We have
$$(f^\mathcal{K})^*(V) = \{ K \subset X \text{ compact} ~|~ K \subset f^{-1}(C) \text{ for some } C \in V \},$$
therefore
$$U \leq (f^\mathcal{K})^*( V ) ~\text{ iff }~ \forall K \in U \exists C \in V : K \subset f^{-1}(C).$$
\end{enumerate}
Assume (1) holds, i.e.\ $f(K) \in V$. Then setting $C = f(K)$ we see that $K \subset f^{-1}(C)$ for $C \in V$, hence condition (2) holds. For the converse, assume (2), i.e.\ $f(K) \subset f^{-1}(C)$. This is the case iff $f(K) \subset C$. By downward closure of $V$, since $C \in V$, this also implies $f(X) \in V$.

Let us know show that $f^\mathcal{K}_! \dashv (f^\mathcal{K})^*$ satisfies the Frobenius identity
$$ f^\mathcal{K}_!( U \wedge (f^\mathcal{K})^*(U) ) = f^\mathcal{K}_!(U) \wedge V $$
for $U \in \mathcal{O}(X^\mathcal{K})$ and $V \in \mathcal{O}(Y^\mathcal{K})$, which then concludes the statement that $f^\mathcal{K} : X^\mathcal{K} \rightarrow Y^\mathcal{K}$ is an open map. Since both $f^\mathcal{K}_!$ and $(f^\mathcal{K})^*$ are functors preserving suprema, it suffices to show the claim for $U = [K], V = [C]$, where $K \subset X$ and $C \subset Y$ are compact. We compute:
$$\begin{array}{rcccl}
f^\mathcal{K}_!( [K] \wedge \bigvee_{ K' \subset f^{-1}(C)} [K'] ) &=& f^\mathcal{K}_!( \bigvee_{ K' \subset f^{-1}(C)} [K \cap K']
&=&  f^\mathcal{K}_!( \bigvee_{ K'' \subset K \cap f^{-1}(C)} [K''] \\ = f^\mathcal{K}_!( [ K \cap f^{-1}(C)] )
&=& [ f( K \cap f^{-1}(C) ) ] &=& [ f(K) \cap C ], 
\end{array}$$
where we used that $K \cap f^{-1}(C)$ is itself compact, hence a maximal element for the set of all $K'' \subset K \cap f^{-1}(C)$ compact, and in the last line the Frobenius identity for image and preimage of functions between sets.

The argument that $\theta$ is a natural transformation is analogous to the argument given in the proof of Theorem \ref{Hausdorfffunctorialityproper}.
\end{proof}

\begin{remark}
Since the locales $X^\mathcal{K}$ are locally coherent---and hence spatial---the theorem above could, in principle, be phrased entirely in classical point-set topology. 
\end{remark}


Theorem \ref{Hausdorfffunctoriality} will prove useful in providing a clean construction of a Boolean locale out of a Hausdorff space. Recall the notion of a marked locale defined in \ref{markedlocales}. We will use a similar notion for Hausdorff spaces.

\begin{definition} \label{compactlymarked}
A pair $(X,N)$ where $X$ is a Hausdorff space and $N \subset \mathcal{K}(X)$ is an ideal of compact subsets is called a \emph{compactly marked Hausdorff space}. If $(X,N), (Y,M)$ are both compactly marked, we say that a continuous map $f : X \rightarrow Y$ is \emph{compatible} if:
\begin{itemize}
\item For all $K \in N$ we have $f(K) \in M$.
\item For all $C \in M$ we have $\{ K \subset f^{-1}(C) \text{ compact} \} \subset N$. 
\end{itemize}
We denote the resulting category of compactly marked Hausdorff spaces and compatible maps by $\mathrm{CMHS}$.
\end{definition}

We remark right away that in the upcoming section the central example of a compactly marked Hausdorff space will be that of a Hausdorff space $X$ equipped with a measure $\mu$, and $N$ being the set of compact $\mu$-null sets.

\begin{corollary} \label{booleanoutofcompact}
There is a functor
$$\begin{array}{rcl}
b : \mathrm{CMHS} & \rightarrow & \mathrm{BoolLoc} \\
     (X,N) &\mapsto & b(X,N)
\end{array}$$
where $b(X,N) = (N^c)_{\neg \neg} \hookrightarrow X^\mathcal{K}$. 
\end{corollary}

\begin{proof}
The functor $(-)^\mathcal{K} : \mathrm{HausSpc} \rightarrow \mathrm{Loc}_{\mathrm{open}}$ refines to a functor $(-)^\mathcal{K} :  \mathrm{CMHS} \rightarrow \mathrm{MarkLoc}_{\mathrm{open}}$ by construction. Our wanted functor is now obtained by composing with the functor $b: \mathrm{MarkLoc}_{\mathrm{open}} \rightarrow \mathrm{BoolLoc}$ from Theorem \ref{booleanadjunction}.
\end{proof}

\section{Radon measures on Hausdorff spaces - The locale theoretic approach} \label{radonvaluations}

We now come to the core example of a Boolean measure locale: The \emph{Radon locale} constructed from a \emph{Radon valuation} on a Hausdorff space $X$. This notion is a slight generalization of that of a \emph{Radon measure}. This notion will give us access to almost all measure spaces useful for mathematical practice within the context of localic measure theory.

\begin{definition} \label{definitionradonvaluation}
Let $X$ be a Hausdorff space. A \emph{Radon valuation} on $X$ is a finite and almost Boolean valuation $\mu : \mathcal{K}(X) \rightarrow [0, \infty)$ on the set of compact subsets of $X$. Denote by $X^\mu$ the associated inner locale together with its locally finite and faithful measure
$$\mu_* : \mathcal{O}(X^\mu) = \mathrm{Sh}( \mathcal{K}(X), \mu; \mathbf{2}) \rightarrow [0, \infty].$$
provided by Theorem \ref{innermeasureproperties}.
\end{definition}

\begin{example}
The main examples of Radon valuations will come from \emph{regular contents}, which we discuss in the following in Section \ref{regularcontents}. In the situation of a regular content $\mu$, there is an associated Radon measure $(X, \mathcal{B}, \mu)$. The frame of the locale $X^\mu$ is isomorphic to the complete Boolean algebra given by measurable sets modulo null sets $\mathcal{B}/\mathcal{N}$ of $X$, a fact that will be proven in Theorem \ref{radonmeasures}.
\end{example}

Let us discuss some consequences of the existence of a Radon valuation. Since the $\mu$-inner topology on $\mathcal{K}(X)$ includes all coverings of the topology $fin$, we have a continuous embedding $X^\mu \hookrightarrow X^\mathcal{K}$. We also have the continuous map $\theta : X^\mathcal{K} \rightarrow X$ by Proposition \ref{framehomforlocallycompacthausdorff}, resulting in the composition 
$$p_\mu : X^\mu \hookrightarrow X^\mathcal{K} \rightarrow X.$$
Since $X^\mu$ is Boolean this further lifts uniquely by Theorem \ref{lifttosublocales2},
\[\begin{tikzcd}
	&& {\mathfrak{Sl}(X)} \\
	{X^\mu} & {X^\mathcal{K}} & X.
	\arrow["{\mathrm{can}}", two heads, from=1-3, to=2-3]
	\arrow["{\exists !}", dashed, from=2-1, to=1-3]
	\arrow[hook, from=2-1, to=2-2]
	\arrow[from=2-2, to=2-3]
\end{tikzcd}\]
	
Pushing forward the inner measure $\mu_*$ on $X^\mu$ equips all locales in this picture with measures and makes all maps involved measure-preserving. In particular, we obtain a measure $\mu_*$ on the frame of sublocales $\mathrm{Sl}(X)^{op}$ of $X$.\footnote{Technically speaking we obtain a \emph{co-measure} on the coframe $\mathrm{Sl}(X)$, as we obtain a contravariant functor $\mu_* :  \mathrm{Sl}(X)^{op} \rightarrow [0, \infty]$ that satisfies modularity, $\mu_*(1) = 0$ and $\mu_*( \bigwedge_{i \in I} S_i ) = \sup_{i \in I} \mu_*(S_i)$. This does not cause many issues in practice however, as for \emph{complemented sublocales}, which contains finite unions and intersections of closed and open sublocales, we can simply pass to the co-measure of the complement.} The value at an open set $U$ of $X$ is computed as
$$\mu_*(U) = \sup_{ K \subset U \text{ compact}} \mu(K),$$
and the value at a closed set $C$ of $X$ is given as
$$\mu_*(C) = \sup \{ \mu(K) ~|~ K \text{ s.t. } \mu(K \cap K') = 0 \text{ for all } K' \subset C^c \text{ compact} \}.$$

\begin{remark} \label{lifttosubsets}
Any subset $S$ of $X$ will be assigned a consistent value that we can reasonably call the \emph{measure} of $S$. To see this, note that we have a lift
\[\begin{tikzcd}
	& {\mathfrak{Sl}(X)} \\
	{X^{disc}} & X
	\arrow[from=1-2, to=2-2]
	\arrow["j", dashed, from=2-1, to=1-2]
	\arrow[from=2-1, to=2-2]
\end{tikzcd}\]
since $X^{disc}$ is Boolean, by the universal property of the map $\mathfrak{Sl}(X) \rightarrow X$ given in Theorem \ref{lifttosublocales2}. Therefore we can use the functor $j_* : \mathcal{P}(X) \rightarrow \mathrm{Sl}(X)^{op}$ to define $\mu(S) = \mu_*(j_*(S))$ for a subset $S \subset X$. This will however not give a measure on $\mathcal{P}(X)$ in the traditional sense, as $j_*$ only preserves infima, and not necessarily joins. We will see later, that if $\mu$ is a regular content, this induced measure agrees with the classical notion of measure, see Corollary \ref{measureofsubset}.
\end{remark}

Moreover, if $A \hookrightarrow X$ is a sublocale of $X$, define the map $p|_A : A^\mu \rightarrow A$ via pullback along $p_\mu$, i.e.\ as
\[\begin{tikzcd}
	{A^\mu} & {X^\mu} \\
	A & X.
	\arrow[hook, from=1-1, to=1-2]
	\arrow["{p|_A}"', from=1-1, to=2-1]
	\arrow["\lrcorner"{anchor=center, pos=0.125}, draw=none, from=1-1, to=2-2]
	\arrow["{p_\mu}", from=1-2, to=2-2]
	\arrow[hook, from=2-1, to=2-2]
\end{tikzcd}\]
Since the pullback of an inclusion of a sublocale along any map is again a sublocale, and $X^\mu$ is Boolean and thus any sublocale of $X^\mu$ is open, this means that $A^\mu \hookrightarrow X^\mu$ is an open inclusion, and $A^\mu$ is, as a sublocale of a Boolean locale, again Boolean. The measure $\mu_*$ on $X^\mu$ thus restricts to a measure ${\mu|_A}_*$ on $A^\mu$, which is again locally finite and faithful. Pushing forward along $p|_A$ equips the sublocale $A$ with a measure $\mu|_A$, which we call the restriction of $\mu$.

We call a continuous map $f : X^\mu \rightarrow Y$, for $Y$ a locale, a $Y$-\emph{valued random variable} on $X^\mu$. The map $p_\mu : X^\mu \rightarrow X$ is a particular such example, and induces the function
$$ \mathrm{Map}( X, Y ) \rightarrow \mathrm{Map}( X^\mu, Y )$$
via precomposition, which realizes continuous functions on $X$ as random variables. 

More generally, suppose we have a partition of $X^\mu$ into pieces of the form $X^\mu = \coprod_{i \in I} A_i^\mu$. Then any choice of individual random variables $f_i : A_i^\mu \rightarrow Y$ for each $i \in I$ determines a unique random variable $f : X^\mu \rightarrow Y$. Thus for example piece-wise defined continuous functions always give associated random variables, assuming the domains of definition cover all of $X$, up to null sets. On top of that, from these basic examples of random variables, assuming the usual completeness assumptions on the target $Y$, more examples of random variables can be constructed via limit procedures. This covers practically all useful measurable functions considered in mathematical analysis, constructed in a point-free manner. Unlike in the classical approach to measure theory, no passage to equivalence classes up to almost everywhere equivalence is necessary. Moreover, any construction made to random variables in our sense will always be measurable in the classical sense. From here, the theory of integration could be developed. However, this remains outside of the scope of this article.

\begin{remark}
The theory works smoothest in the case $X$ is additionally locally compact, and $\lambda$ is a regular content. We will develop some of the consequences of both of this cases in the upcoming sections. However, measures on spaces, such as spaces of functions, that are not necessarily locally compact are important in practice.
\end{remark}

\subsection{Functoriality for Radon locales}

Let us discuss the functoriality of the assignment $X \mapsto X^\mu$. Here we want to consider two different cases of maps between Hausdorff spaces equipped with Radon valuations: That of (globally defined) continuous measure-preserving maps, as well as the case of measure-preserving partially defined proper maps with open support.

\begin{definition}
A continuous map $f : X \rightarrow Y$ between Hausdorff spaces $X,Y$ equipped with valuations $\mu, \nu$ on the set of compact sets $\mathcal{K}(X)$, respectively $\mathcal{K}(Y)$, is called \emph{compactly measure-preserving} if the induced map $f^\mathcal{K} : X^\mathcal{K} \rightarrow Y^\mathcal{K}$ is measure-preserving when equipping $X^\mathcal{K}$ and $Y^\mathcal{K}$ with the induced inner measures $\mu_*$ and $\nu_*$. Equivalently, $f$ is compactly measure-preserving iff for all compact sets $C \subset Y$ it holds that
$$ \mu_*((f^\mathcal{K})^*(C)) = \sup_{K \subset f^{-1}(C) \text{ compact}} \mu(K) = \nu(C). $$
Denote by $\mathrm{ValHausSpc}_{\mathrm{glob}}$ the resulting category of pairs $(X,\mu)$ with $X$ a Hausdorff spaces and $\mu$ a valuation on $\mathcal{K}(X)$, and compact measure-preserving maps between them, and $\mathrm{RadHausSpc}_{\mathrm{glob}}$ the full subcategory spanned by pairs $(X,\mu)$ such that $\mu$ is a Radon valuation.
\end{definition}

In the following, recall that $\mathrm{MeasLoc}$ is the category of locales equipped with measures, and measure-preserving partial maps. We denote by $\mathrm{MeasLoc}_{\mathrm{glob}}$ the wide subcategory given by restricting to globally defined measure-preserving maps. 

Note that any compactly measure-preserving map $f : X \rightarrow Y$ is in particular a measure-preserving map $(X, \mu_*) \rightarrow (Y, \nu_*)$. To see this, consider the commuting square
\[\begin{tikzcd}
	{X^\mathcal{K}} & {Y^\mathcal{K}} \\
	X & Y,
	\arrow["{f^\mathcal{K}}", from=1-1, to=1-2]
	\arrow["{\theta_X}"', from=1-1, to=2-1]
	\arrow["{\theta_Y}", from=1-2, to=2-2]
	\arrow["f", from=2-1, to=2-2]
\end{tikzcd}\]
given by Theorem \ref{Hausdorfffunctoriality}. Then $f$ being compactly measure-preserving is equivalent to $f^\mathcal{K}$ being measure-preserving. Since $\mu_*$ and $\nu_*$ are defined via pushforward along $\theta$, this means all maps in this square become measure-preserving. In particular, this gives a forget functor $L : \mathrm{ValHausSpc}_{\mathrm{glob}} \rightarrow \mathrm{MeasLoc}_{\mathrm{glob}}$, which sends $(X, \mu)$ to $(X, \mu_*)$.

\begin{theorem}  \label{functorialityradonvaluationglobal}
There exists a functor
$$\begin{array}{rcl}
\mathrm{Rad} : \mathrm{ValHausSpc}_{\mathrm{glob}}  &\rightarrow& \mathrm{MeasLoc}_{\mathrm{l.f.~faithf, glob}} \\
X &\mapsto& X^\mu
\end{array}$$
where $\mathrm{MeasLoc}_{\mathrm{l.f.~faithf, glob}}$ is the category of locally finite and faithful measure locales, and measure-preserving globally defined maps between them, together with a natural transformation $p : U \mathrm{Rad} \rightarrow L$, where $U :  \mathrm{MeasLoc}_{\mathrm{l.f.~faithf, glob}} \rightarrow \mathrm{MeasLoc}_{glob}$ is the canonical inclusion, given in components by the map
$$ p_\mu : X^\mu \rightarrow X.$$

The functor $\mathrm{Rad}$ restricts to a functor
$$\begin{array}{rcl}
\mathrm{Rad} : \mathrm{RadHausSpc}_{\mathrm{glob}}  &\rightarrow& \mathrm{BoolMeasLoc}_{\mathrm{l.f.~faithf, glob}} \\
\end{array}$$
where $\mathrm{BoolMeasLoc}_{\mathrm{l.f.~faithf, glob}}$ is the category of Boolean, locally finite and faithful measure locales, with measure-preserving globally defined maps between them.
\end{theorem}

\begin{proof}
Let $f : X \rightarrow Y$ be a continuous, compactly measure-preserving map with $X, Y$ both both Hausdorff spaces, and $\mu$, $\nu$ valuations on their respective sets of compact subsets of $X$ and $Y$. We have the natural commutative square given by Theorem  \ref{Hausdorfffunctoriality},
\[\begin{tikzcd}
	{X^\mathcal{K}} & {Y^\mathcal{K}} \\
	X & Y.
	\arrow["{f^\mathcal{K}}", from=1-1, to=1-2]
	\arrow["{\theta_X}"', from=1-1, to=2-1]
	\arrow["{\theta_Y}", from=1-2, to=2-2]
	\arrow["f", from=2-1, to=2-2]
\end{tikzcd}\]
To show that $f^\mathcal{K}$ further refines to a map $\mathrm{Rad}(f) : X^\mu \rightarrow  Y^\nu$ it suffices to show that the induced functor
$$\begin{array}{rcl}
(f^\mathcal{K})^*)^{sh} : \mathcal{K}(Y) & \rightarrow & \mathcal{O}( X^\mu ) = \mathrm{Sh}( \mathcal{K}(X), \mu ; \mathbf{2} ) \\
C & \mapsto & \{ K \subset f^{-1}(C) \text{ compact }\}^{sh}
\end{array}$$
is flat with respect to the $\nu$-inner pretopology on $\mathcal{K}(Y)$. The case of finite covers is clear. Let $\{ C_i \subset C ~|~ i \in I \}$ be a directed $\nu$-approximation. Using that $f$ is measure-preserving, we have
$$ \sup_{i \in I} \mu_*( (f^\mathcal{K})^*(C_i)^{sh} ) = \sup_{i \in I} \mu_*( (f^\mathcal{K})^*(C_i) ) = \sup_{i \in I} \nu( C_i ) = \nu( C ) =  \mu_*( (f^\mathcal{K})^*(C) )  =  \mu_*( (f^\mathcal{K})^*(C)^{sh} ) < \infty. $$
Since $\mu_*$ is faithful on $\mathrm{Sh}( \mathcal{K}(X), \mu ; \mathbf{2} )$, this means that
$$ \bigvee_{i \in I } (f^\mathcal{K})^*(C_i)^{sh} = (f^\mathcal{K})^*(C)^{sh}$$
in $\mathrm{Sh}( \mathcal{K}(X), \mu ; \mathbf{2} )$, in other words, $(f^\mathcal{K})^*)^{sh}$ is flat. It is also clear that the induced map $\mathrm{Rad}(f)$ is measure-preserving. We thus obtain the commutative diagram
\[\begin{tikzcd}
	{X^\mu} & {Y^\nu} \\
	{X^\mathcal{K}} & {Y^\mathcal{K}} \\
	X & Y.
	\arrow["{\mathrm{Rad}(f)}", from=1-1, to=1-2]
	\arrow[from=1-1, to=2-1]
	\arrow[from=1-2, to=2-2]
	\arrow["{f^\mathcal{K}}", from=2-1, to=2-2]
	\arrow[from=2-1, to=3-1]
	\arrow[from=2-2, to=3-2]
	\arrow["f", from=3-1, to=3-2]
\end{tikzcd}\]
Since all measures in question are defined via pushforward, all maps in this diagram are measure-preserving. This settles the question of naturality. For the last point, the restriction to $\mathrm{RadHausSpc}_{\mathrm{glob}}$ maps into Boolean locales by Theorem \ref{almostbooleangivesboolean}.
\end{proof}

We now come to the analogous situation when considering partially defined proper maps instead.

\begin{definition}
A partial proper map $f : X \rightarrow Y$ between Hausdorff spaces $X,Y$ equipped with valuations $\mu, \nu$ on the set of compact sets $\mathcal{K}(X)$, respectively $\mathcal{K}(Y)$, is called \emph{compactly measure-preserving} if $\mu(f^{-1}(K)) = \nu(K)$ for all compact $K \subset Y$. Denote by $\mathrm{ValHausSpc}$ the category of Hausdorff spaces $X$ equipped with such valuations $\mu$, and compactly measure-preserving partial proper maps between them. We denote the full subcategory given by $X$ equipped with a Radon valuation $\mu$ by $\mathrm{RadHausSpc}$.
\end{definition}

Similarly to the case of globally defined maps, any compactly measure-preserving partial proper map $f : X \rightarrow Y$ is in particular a measure-preserving partial map $(X, \mu_*) \rightarrow (Y, \nu_*)$, using the same reasoning with Theorem \ref{Hausdorfffunctorialityproper} instead of Theorem \ref{Hausdorfffunctoriality}. Denote the forget functor which sends $(X, \mu)$ to $(X, \mu_*)$ by $L : \mathrm{ValHausSpc} \rightarrow \mathrm{MeasLoc}$.

\begin{theorem} \label{functorialityradonvaluationproper}
There exists a functor
$$\begin{array}{rcl}
(-)^\mu : \mathrm{ValHausSpc}  &\rightarrow& \mathrm{MeasLoc}_{\mathrm{l.f.~faithf}} \\
X &\mapsto& X^\mu
\end{array}$$
where $\mathrm{MeasLoc}_{\mathrm{l.f.~faithf}}$ is the category of locally finite and faithful measure locales, and measure-preserving partial maps between them, together with a natural transformation $p : U (-)^\mu \rightarrow L$, where $U :  \mathrm{MeasLoc}_{\mathrm{l.f.~faithf}} \rightarrow \mathrm{MeasLoc}$ is the canonical inclusion, given in components by the map
$$ p_\mu : X^\mu \rightarrow X.$$

This functor restricts to a functor
$$\begin{array}{rcl}
(-)^\mu : \mathrm{RadHausSpc}  &\rightarrow& \mathrm{BoolMeasLoc}_{\mathrm{l.f.~faithf}} \\
\end{array}$$
where $\mathrm{BoolMeasLoc}_{\mathrm{l.f.~faithf}}$ is the category of Boolean, locally finite and faithful measure locales, with measure-preserving partial maps between them.
\end{theorem}

\begin{proof} Since $f : X \rightarrow Y$ being compactly measure-preserving means that $f^{-1} : \mathcal{K}(Y) \rightarrow \mathcal{K}(X)$ is a valuation-preserving homomorphism of lower bounded distributive lattices, we have a functor $\mathrm{ValHausSpc}^{op} \rightarrow \mathrm{ValSite}$. Composing with the functor $(-)^{inn}$ provided by Theorem \ref{innervaluationadjunction} gives the functor $(-)^\mu$.

Naturality of $p$ follows analogously to the proof given for Theorem \ref{functorialityradonvaluationglobal} by using that we have the natural diagram
\[\begin{tikzcd}
	{X^\mu} & {Y^\nu} \\
	{X^\mathcal{K}} & {Y^\mathcal{K}} \\
	X & Y.
	\arrow["{\mathrm{Rad}(f)}", from=1-1, to=1-2]
	\arrow[from=1-1, to=2-1]
	\arrow[from=1-2, to=2-2]
	\arrow["{f^\mathcal{K}}", from=2-1, to=2-2]
	\arrow[from=2-1, to=3-1]
	\arrow[from=2-2, to=3-2]
	\arrow["f", from=3-1, to=3-2]
\end{tikzcd}\]
with the bottom square given by Theorem \ref{Hausdorfffunctorialityproper}, and the top square given by the observation that $f^{-1} : \mathcal{K}(Y) \rightarrow \mathcal{K}(X)$ is a morphism of sites. Analogously to the same statement made in Theorem \ref{functorialityradonvaluationglobal}, the restriction to $\mathrm{RadHausSpc}$ maps into Boolean locales by Theorem \ref{almostbooleangivesboolean}.
\end{proof}

\subsection{Functoriality of the induced measure on the frame of sublocales}

We also have functoriality with respect to the induced measure on sublocales in the case of Radon valuation. This uses the universal property of the map $\mathrm{can} : \mathfrak{Sl}(X) \rightarrow X $, which requires the use of globally defined measure-preserving maps. Denote by $\mathrm{RadHausSpc}_{glob}$ the category of Hausdorff spaces equipped with Radon valuations, and globally defined measure-preserving proper maps between them.

\begin{theorem} \label{functorialitymeasuresublocales}
There exists a lift of the functor $\mathfrak{Sl} : \mathrm{RadHausSpc}_{glob} \rightarrow \mathrm{Loc}$ to a functor
$$\begin{array}{rcl}
\mathfrak{Sl} : \mathrm{RadHausSpc}_{glob} & \rightarrow & \mathrm{MeasLoc}_{glob} \\
 (X, \mu) & \mapsto & (\mathfrak{Sl}(X), \mu_*)
\end{array}$$
where $\mathrm{MeasLoc}_{glob}$ is the category of locales equipped with measures and globally defined measure-preserving maps between them.
\end{theorem}

\begin{remark}
This in particular means that the inner measure $\mu_*$ defined on the locale $\mathfrak{Sl}(X)$ of sublocales of $X$ is invariant under measure-preserving homeomorphisms of $X$.
\end{remark}

\begin{proof}
This is mostly a formal consequence of Theorem \ref{functorialityradonvaluationglobal}. As described in the beginning of Section \ref{radonvaluations}, for a given Hausdorff space $X$ with Radon valuation $\mu$, we have a natural lift
\[\begin{tikzcd}
	& {\mathfrak{Sl}(X)} \\
	{X^\mu}  & X.
	\arrow["{\mathrm{can}}", two heads, from=1-2, to=2-2]
	\arrow["{\exists !}", dashed, from=2-1, to=1-2]
	\arrow[from=2-1, to=2-2]
\end{tikzcd}\]
obtained by using Theorem \ref{lifttosublocales2}, since $X^\mu$ is Boolean. Since $p : U \mathrm{Rad} \rightarrow L$ is natural as proven in Theorem \ref{functorialityradonvaluationglobal}, so is the lift to $\mathfrak{Sl}$. All measures involved are induced via pushforward, and so all maps become measure-preserving.
\end{proof}

\begin{corollary} \label{invariancemeasuresublocale}
Let $X$ be a Hausdorff space equipped with a Radon valuation $\mu$. Then there exists a measure $\mu_*$ on the locale of sublocales $\mathfrak{Sl}(X)$ of $X$, which is invariant under measure-preserving homeomorphisms of $X$, and the natural continuous map $\mathfrak{Sl}(X) \rightarrow X$ is measure-preserving.
\end{corollary}

\begin{remark}
In spirit at least this result is analogous to the results that were obtained independently by Simpson \cite{SIMPSON20121642} and Leroy \cite{leroy2013theorielamesuredans}. In both cases, Simpson and Leroy put a valuation $\mu^*$ on $(\mathrm{Sl}(X), \subset)$, which mimics the construction of an outer measure. In the case of Simpson, $X$ is assumed to be a fitted $\sigma$-locale, in the case of Leroy, $X$ is assumed to be a regular locale, in both cases $\mu : \mathcal{O}(X) \rightarrow [0,\infty)$ is assumed to be a valuation satisfying further conditions. The valuation $\mu*$ is then shown to be cocontinuous, i.e.\
$$\mu( \bigwedge_{i\in I} S_i ) = \inf_{i\in I} \mu(S_i)$$
for any downward directed set of sublocales $S_i$ (assuming finiteness in the case of Simpson), as well as $\sigma$-continuous, i.e.\
$$\mu( \bigvee_{n \in \mathbb{N}} S_n ) = \sup_{n \in \mathbb{N}} \mu(S_n).$$
The main conceptual difference for the approach presented in this article is that we instead construct the corresponding \emph{co-measure} $\mu_*$ on $\mathrm{Sl}(X)^{op}$. This has the main conceptual advantage that we are never leaving the world of locales.
\end{remark}

\subsection{Regular contents} \label{regularcontents}

A particularly nice class of Radon valuations is obtained from the notion of a \emph{regular content}. Regular contents have long been used in the classical construction of Radon measures. It is therefore not surprising that employing them in the construction of measure locales yields a theory essentially equivalent to the classical one. Their use also makes it easy to transfer classical constructions into the point-free setting, since in practice the essential geometric or logical information of a given measure theoretic situation is already captured in the construction of a regular content.

\begin{definition}[See \cite{halmos2013measure}, Section 53] Let $X$ be a topological space. A \emph{content} on $X$ is a function $\lambda : \mathcal{K}(X) \rightarrow [0,+\infty)$ on the set of compact subsets such that the following conditions hold.
\begin{itemize}
  \item \emph{Monotone:} If $C \subset D$ are compact sets, then $\lambda(C) \leq \lambda(D)$.
    \item \emph{Subadditive for unions:} For all compact sets $C, D$ we have
  \[
    \mu(C \cup D) \leq \mu(C) + \mu(D).
  \]
  \item \emph{Additive for disjoint unions:} For all compact sets $C, D$ with $C \cap D = \emptyset$ we have
  \[
    \mu(C \cup D) = \mu(C) + \mu(D).
  \]
\end{itemize}
\end{definition}
Note that $\mu( \emptyset ) = \mu( \emptyset \cup \emptyset ) = \mu( \emptyset )  + \mu( \emptyset )$, therefore $\mu( \emptyset ) = 0$.

\begin{definition}
Let $X$ be a topological space, and let $K_1, K_2$ be compact subsets. We say $K_1 \ll K_2$ if there exists $U$ open such that $K_1 \subset U \subset K_2$.
\end{definition}

\begin{definition}[See \cite{halmos2013measure}, Section 54]
Let $X$ be a topological space. A content $\lambda : \mathcal{K}(X) \rightarrow [0,+\infty)$ is called \emph{regular} if $\lambda(K) = \inf \{ \lambda(K') ~|~ K \ll K',~ K' \text{ compact} \}$ for all compact subsets $K$.
\end{definition}

\begin{example} \label{examplesradon}
Every Radon measure $(X, \mathcal{B}, \mu)$ on a Hausdorff space $X$ gives rise to a regular content, see the upcoming Section \ref{radonmeasurespaces} for details. In fact, a standard procedure in order to construct a Radon measure on locally compact Hausdorff space is either to give a direct construction of a content, or equivalently to give a positive linear functional $C_c(X) \rightarrow \mathbb{R}$, where $C_c(X)$ is the space of continuous functions with compact support via the Riesz–Markov–Kakutani representation theorem (See e.g.\ \cite[Theorem 7.2.8]{cohn1994measure}). As such we can give many different examples of regular contents.
\begin{itemize}
\item A discrete set $S$ equipped with the counting measure on $\mathcal{K}(S) = \mathcal{P}(S)^{fin}$.
\item The Lebesgue measure $\lambda_d$ on $\mathbb{R}^d$. More on this in Section \ref{lebesguereals}.
\item More generally, the Haar measure $\mu$ on a locally compact Hausdorff group $G$, discussed in Section \ref{haarmeasure}.
\item The Dirac measure $\lambda_x$ for any point $x \in X$.
\item The Volume measure on a Riemannian manifold $(M,g)$.
\item The Wiener measure on $C([0,1];\mathbb{R})$, modelling Brownian motion (See e.g.\ \cite[VI.5]{schwartz1973radon}).
\end{itemize}
\end{example}

We now want to compare the notion of a regular content with that of a valuation. Note that $\mathcal{K}(X)$ need not be a lower bounded distributive lattice for general topological spaces, since intersections of compact sets need not be compact. For this reason we restrict our attention to $X$ being Hausdorff.

\begin{theorem} \label{regularcontentalmostboolean}
Let $X$ be a Hausdorff topological space and $\lambda : \mathcal{K}(X) \rightarrow [0,+\infty)$ a regular content. Then $\lambda$ is an almost Boolean valuation on $\mathcal{K}(X)$.
\end{theorem}

\begin{proof}
We need to verify two things for the function $\lambda$: 

\noindent \emph{Modularity:} Let $C_1, C_2$ be two compact subsets. We want to show that
$$ \lambda( C_1 ) + \lambda( C_2 ) = \lambda( C_1 \cup C_2 ) + \lambda( C_1 \cap C_2 ).$$
Using regularity of $\lambda$, let $\epsilon > 0$ and choose $C_1 \cap C_2 \subset V \subset K$ with $V$ open and $K$ compact, such that $\mu(K) - \mu( C_1 \cap C_2 ) < \epsilon$.

We estimate for $i = 1,2$:
$$
\lambda(C_i) \leq \lambda( C_i \setminus V \cup K ) \leq \lambda( C_i \setminus V ) + \lambda( K ) \leq \lambda( C_i \setminus V ) + \lambda( C_1 \cap C_2 ) + \epsilon \leq \lambda(C_i) + \epsilon
$$
and
$$\begin{array}{rcccl}
\lambda(C_1 \cup C_2) &\leq& \lambda( C_1 \setminus V \cup C_2 \setminus V \cup K ) &\leq& \lambda( C_1 \setminus V ) + \lambda( C_2 \setminus V ) + \lambda( K ) \\ &\leq& \lambda( C_1 \setminus V ) + \lambda( C_2 \setminus V ) + \lambda( C_1 \cap C_2 ) + \epsilon &\leq& \lambda(C_1 \cup C_2) + \epsilon.
\end{array}$$
Putting these inequalities together we obtain
$$\begin{array}{rcccl}
\lambda(C_1 \cup C_2) + \lambda(C_1 \cap C_2)  &\leq& \lambda( C_1 \setminus V ) + \lambda( C_2 \setminus V ) + 2\lambda( C_1 \cap C_2 ) + \epsilon &\leq& \lambda( C_1 ) + \lambda( C_2 ) + \epsilon \\ &\leq&  \lambda( C_1 \setminus V ) + \lambda( C_2 \setminus V ) + 2\lambda( C_1 \cap C_2 ) + 3\epsilon
&\leq& \lambda(C_1 \cup C_2) + \lambda(C_1 \cap C_2) + 3 \epsilon.
\end{array}$$
Letting $\epsilon$ tend to zero shows the claim. \\

\noindent \emph{Almost Boolean:} Let $\epsilon > 0$, $C$ a compact set and assume that $C_1 \subset C_2 \subset \hdots $ is an ascending chain of compact sets such that $C_n \subset C$ for all $n \in \mathbb{N}$. Using regularity of $\lambda$, for all $n \in \mathbb{N}$ we find
$C_n \subset U_n \subset K_n,$ where $U_n$ are open and $K_n$ compact, such that 
$$ \mu( K_n ) - \mu(C_n) \leq \frac{\epsilon}{2}. $$
Without loss of generality, we may also assume that $U_n \subset U_{n+1}$ and $K_n \subset K_{n+1}$. Define $U = \bigcup_{n \in \mathbb{N}} U_n$. We claim that the compact set $C \setminus U$ is an approximate complement to the sequence $C_n$ relative to $C$. To see this, again using regularity, choose $C \setminus U \subset W \subset L$ with $W$ open and $L$ compact such that
$$ \mu( L ) - \mu(C \setminus U ) \leq \frac{\epsilon}{2}. $$
The open sets $U_n, n \in \mathbb{N}$ together with $W$ form an open cover of $C$. By compactness, we conclude that there exists $N \in \mathbb{N}$ such that $C \subset U_N \cup W$. Since $U_N \subset K_N$ and $W \subset L$ we have
$$\mu(C) \leq \mu( K_N ) + \mu( L ) \leq \mu( C_N) +  \mu(C \setminus U ) + \epsilon,$$
which concludes the proof.
\end{proof}

If $\lambda$ is a regular content, the map $\tilde{p} : X^\lambda \rightarrow \mathfrak{Sl}(X)$ induces a measure $\lambda_*$ on the frame $\mathrm{Sl}(X)^{op}$, or equivalently a co-measure on $\mathrm{Sl}$. The measure of an open $U$, equivalently co-measure  of the sublocale $U^c \hookrightarrow X$ is given as $\lambda_*(U) = \sup_{ K \subset U \text{ compact}} \lambda(K)$, almost by definition. We claim that the same formula holds for the measure of a closed subset, or equivalently the co-measure of the open inclusion $U \hookrightarrow X$.

\begin{proposition} \label{closedsublocalemeasure} Let $X$ be a Hausdorff space and $\lambda$ a regular content on $X$. Let $C \subset X$ be a closed subset, identified with the open sublocale $C^c \hookrightarrow X \in \mathrm{Sl}(X)^{op}$. Then
$$ \tilde{p}^*(C) = \{ K \subset C ~|~ K \text{ compact} \}^{sh} \in \mathcal{O}(X^\lambda).$$
In particular
$$\mu_*(C) = \sup_{ K \subset C \text{ compact}} \lambda(K).$$
\end{proposition}

\begin{proof}
Let $U = C^c$. Under the inverse image functor of the map $X^\lambda \rightarrow \mathfrak{Sl}(X)$, the closed set $C$ is sent to
$$\neg p_\lambda^*(U) = \{ K ~|~ K \text{ compact, s.t. } \mu(K \cap K') = 0 \text{ for all } K' \subset U \text{ compact} \}.$$
by Lemma \ref{implicationformulas}. We claim that this ideal arises as the closure under $\mu$-approximations of the set 
$$\{ K \subset C ~|~ K \text{ compact} \}.$$
Let $M$ be compact such that $\mu(M \cap K') = 0$ for all  $M \subset U \text{ compact}$. Then $M \setminus U$ is again compact. Let $\epsilon > 0$ and use regularity of $\lambda$ to find $M \setminus U \subset V \subset K''$, with $V$ open, $K''$ compact, such that $\mu(K'') - \mu(M \setminus U ) < \epsilon$.
Then $M \setminus V \subset M \cap C$ and therefore $\mu(M \setminus V) = 0$. Therefore
$$ \mu(M) \leq \mu(K'') + \mu( M \setminus V ) <  \mu(M \setminus U) + \mu( M \setminus V ) + \epsilon = \mu(M \setminus U) + \epsilon.$$
Since $\epsilon$ was arbitrary, we conclude $\mu( K \setminus U ) = \mu(K)$. But this means that $K$ is in the $\mu$-closure of $\{ K \subset C ~|~ K \text{ compact} \}$.
\end{proof}

\begin{corollary} \label{compactcomplement}
Let $\lambda$ be a regular content on a Hausdorff space $X$, and $K \subset X$ compact. Then $[K]$ and $p^*(K^c)$ are complementary in $\mathcal{O}(X^\lambda)$.
\end{corollary}

\begin{proof} Proposition \ref{closedsublocalemeasure} applied to $K$ gives that $\tilde{p}^*(C) = [K]$, which is by construction of the map $\tilde{p}$ the complement of $p^*(K^c) = \tilde{p}^*(K^c)$.
\end{proof}

\begin{theorem} \label{measurablelocaleembedded}
Let $\lambda$ be a regular content on a Hausdorff space $X$. Then the induced map $\tilde{p} : X^\lambda \rightarrow \mathfrak{Sl}(X)$ is an embedding.
\end{theorem}

\begin{proof}
By definition of $\tilde{p} : X^\lambda \rightarrow \mathfrak{Sl}(X)$, it's inverse image functor $\tilde{p} : \mathrm{Sl}(X)^{op}$ sends an open sublocale $U \hookrightarrow X$ to the complement of $p^*(U)$ in $\mathcal{O}(X^\mu)$. Lemma \ref{compactcomplement} implies that all chunks $[K] \in \mathcal{O}(X^\mu)$ lie in the image of $\tilde{p}$, for $K \subset X$ compact. Since these generate all of $ \mathcal{O}(X^\mu)$ under suprema, we see that $\tilde{p}^*$ is surjective.
\end{proof}

\begin{remark}
The fact that the measure algebra can be obtained as a Boolean sublocale of $\mathfrak{Sl}(L)$ has been observed by Leroy in \cite[Théoremè 2]{leroy2013theorielamesuredans} in the context of a locally finite measure $\mu$ on a regular locale $L$. Leroy characterizes this sublocale of $\mathfrak{Sl}(L)$ as given by the set of $\mu$-reduced sublocales of $L$. Leroy's characterization provides an independent approach to the construction of the corresponding measurable locale.
\end{remark}

Let $X^{disc}$ denote the underlying set of $X$ equipped with the discrete topology. As remarked earlier in Remark \ref{lifttosubsets}, we have a continuous map $\mathrm{can} : X^{disc} \rightarrow X$. Since $X^{disc}$ is Boolean, there exists a lift to a map $\varphi$ into $\mathfrak{Sl}(X)$,
\[\begin{tikzcd}
	& {\mathfrak{Sl}(X)} \\
	{X^{disc}}  & X.
	\arrow["{\mathrm{can}}", two heads, from=1-2, to=2-2]
	\arrow["{\exists ! ~ \varphi}", dashed, from=2-1, to=1-2]
	\arrow[from=2-1, to=2-2]
\end{tikzcd}\]
Concretely, $\varphi$ gives an adjunction
\[\begin{tikzcd}
	{\mathcal{P}(X)} & {\mathrm{Sl}(X)^{op}}
	\arrow[""{name=0, anchor=center, inner sep=0}, "{\varphi_*}"', curve={height=12pt}, from=1-1, to=1-2]
	\arrow[""{name=1, anchor=center, inner sep=0}, "{\varphi^*}"', curve={height=12pt}, from=1-2, to=1-1]
	\arrow["\dashv"{anchor=center, rotate=-90}, draw=none, from=1, to=0]
\end{tikzcd}\]
where $\varphi^*$ sends a closed sublocale $U^c$ to the open subset $U \subset X$, and the open sublocale $U$ to the closed subset $U^c \subset X$. In particular, a sublocale of the form $U \vee V^c$ for opens $U,V$ is sent to the \emph{locally closed} subset $U^c \cap V \subset X$. The right adjoint $\varphi_*$ can also be described concretely. If $S \subset X$ is a subset, it is sent to
$$\varphi_*(S) = \bigvee_{ M ~|~ i^*(M) \subset S } M = \bigvee_{ U, V \text{ open}  ~|~ U^c \cap V \subset S } U \wedge V^c $$
with the suprema and meet taken in the frame $\mathrm{Sl}(X)^{op}$. (It corresponds to the infimum over $U \vee V^c$ in the co-frame $\mathrm{Sl}(X)$.) In other words, the sublocale associated to a subset can be understood by the locally closed sets contained in it.

\begin{corollary} \label{measureofsubset}
Let $\lambda$ be a regular content on a Hausdorff space $X$, and let $S \subset X$ be a subset. The induced measure of $S$ is given by
$$ \lambda_*(S) = \sup \{ \lambda(K) ~|~ K \subset S \text{ compact} \}. $$
\end{corollary}

\begin{remark}
In particular, the induced measure agrees with the classical measure when $S$ is a measurable subset of $X$.
\end{remark}

\begin{proof}
By the above discussion we have
$$ \lambda_*(S) = \sup \{ \lambda_*(U^c \cap V) ~|~ U^c \cap V \subset S \text{ s.t. } U,V \text{ open} \}. $$
Therefore the formula reduces to the case of locally closed sets. The case for open sets holds by definition, the case for closed sets is given by Proposition \ref{closedsublocalemeasure}. The measure of the intersection of both cases is computed via the measure of the intersection of the corresponding $\mu$-ideals of compact sets, hence the formula is true for locally closed sets.
\end{proof}

The map $X^\lambda \rightarrow X$ is actually universal among maps from Boolean locales equipped with locally finite and faithful measures, a result that was anticipated by Simon Henry in the mathoverflow post \cite{486494}.

\begin{theorem} \label{universalpropertyregularcontent}
Let $\lambda$ be a regular content on a Hausdorff space $X$ and assume that the induced measure $\lambda_*$ on $X$ is locally finite.\footnote{This condition is automatically satisfied if $X$ is locally compact, see Lemma \ref{localfinitenesslocallycompact}.} Suppose $ f : (Y,\mu) \rightarrow (X, \lambda_*)$ is a measure-preserving continuous map, with $Y$ being Boolean, and $\mu$ a locally finite and faithful measure on $Y$. Then there exists a unique measure-preserving map $ \tilde{f} : (Y,\mu) \rightarrow (X^\lambda, \lambda_*)$ making the triangle
\[\begin{tikzcd}
	Y \\
	{X^\lambda} & X
	\arrow["{\tilde{f}}"', dashed, from=1-1, to=2-1]
	\arrow["f", from=1-1, to=2-2]
	\arrow["p"', from=2-1, to=2-2]
\end{tikzcd}\]
commute.
\end{theorem}

\begin{proof} By virtue of $Y$ being Boolean, there exists a unique lift $\bar{f}$ of $f$ against $\mathrm{can} : \mathfrak{Sl}(X) \rightarrow X$ by Theorem \ref{lifttosublocales2},
\[\begin{tikzcd}
	Y & {\mathfrak{Sl}(X)} \\
	& X.
	\arrow["{\bar{f}}", dashed, from=1-1, to=1-2]
	\arrow["f"', from=1-1, to=2-2]
	\arrow["{\mathrm{can}}", from=1-2, to=2-2]
\end{tikzcd}\]
Since all measures involved are defined via pushforwards, this lift is automatically measure-preserving. Denote by $\tilde{\lambda}_*$ the induced measure on $\mathfrak{Sl}(X)$. Showing that $\bar{f}$ factors through the embedding $\tilde{p} : X^\lambda \rightarrow \mathfrak{Sl}(X)$ provided by Theorem \ref{measurablelocaleembedded} is equivalent to showing that for each pair $S \leq S'$ of sublocales in $\mathrm{Sl}(X)^{op}$ such that $\tilde{p}^*(S) = \tilde{p}^*(S')$ we have $\bar{f}^*(S) = \bar{f}^*(S')$. Note that
$$\mu( \bar{f}^*(S)) = \tilde{\lambda}_*( S ) = \lambda_*( \tilde{p}_*(S) ) = \lambda_*( \tilde{p}_*(S') ) = \tilde{\lambda}_*( S' ) =  \mu( \bar{f}^*(S')).$$
Using that $\mu$ is locally finite and faithful, this means that $\bar{f}^*(S) = \bar{f}^*(S')$ for each pair of sublocale $S \leq S'$ of $X$ of finite measure. To show the same even in the case of infinite measure, note that the condition that $\lambda_*$ is locally finite on $X$ implies that $\tilde{\lambda}_*$ is also locally finite on $\mathfrak{Sl}(X)$ by Lemma \ref{pullbacklocallyfinite}. But this means that any sublocale $S$ can be written as a supremum of sublocales of finite measure, and hence we can reduce to the case of finite measure sublocales. This shows that $\bar{f}^*$ factors through $\mathcal{O}(X^\mu)$.
\end{proof}

Finally, we give a criterion for identifying points of $X^\mu$.

\begin{proposition} \label{pointsregularcontent}
Let $X$ be a Hausdorff space, $\lambda$ a regular content and $X^\lambda$ the associated Boolean inner measure locale. Then the natural map $p_\lambda : X^\lambda \rightarrow X$ induces an injection
$$ \mathrm{pts}( X^\lambda )  \hookrightarrow \mathrm{pts}( X )$$
which identifies the left-hand side with the set of point $x$ of $X$ such that $\lambda(\{x\}) > 0$.
\end{proposition}

\begin{proof}
Since $X^\lambda$ is Boolean, by Lemma \ref{pointsareminimalelements} points of $X^\lambda$ can be identified with minimal, non-zero $\lambda$-ideals of $\mathcal{K}(X)$. By minimality, such a $\lambda$-ideal needs to be of the form $[K]$ for some compact $K$ with $\lambda(K) > 0$. If $K$ contains a point $x$ such that $\lambda(\{x\}) > 0$, we again by minimality conclude that $[K] = [ \{x\} ]$. This point is sent under $p_\lambda$ to the open neighbourhood filter of $X$, or equivalently just $x$. Now we argue by contradiction. Assume that for all points $x$ of $K$ we have $\lambda(\{x\}) = 0$	. Let $\epsilon > 0$, such that $2 \epsilon < \mu(K)$. Using regularity twice we find for each $x \in X$ a sequence $x \in U_x \subset K_x \subset V_x \subset C_x$ with $U_x, V_x$ open and $K_x, C_x$ compact, and
$$\begin{array}{l}
\lambda(K_x) < \epsilon \text{ as well as} \\
\lambda(C_x) < 2 \epsilon.
\end{array}$$
The sets $\{U_x\}_{x \in K}$ form an open cover of $K$, therefore there exist finitely many $x_1,\hdots, x_n \in K$ such that the sets $U_{x_1}, \hdots,U_{x_n}$ cover $K$. We can write
$$ K = (K \cap K_{x_1}) \cup  \hdots \cup (K \cap K_{x_1})$$
as a finite union of compact sets of measure less then $\epsilon$. Since $\lambda(K) > 0$, we must have that $\lambda( K \cap K_{x_k} ) > 0$ for at least one $k$. Observe that:
\begin{itemize}
\item $\lambda( K \setminus V_{x_k} ) < \lambda( K )$, as
$$ \lambda( K \setminus V_{x_k} ) + \lambda( K \cap K_{x_k} ) \leq \lambda( K )$$
and $\lambda( K \cap K_{x_k} ) > 0$.
\item $\lambda( K \setminus V_{x_k} ) > 0$, as
$$2 \epsilon < \mu(K) \leq \mu( C_{x_k} ) + \lambda( K \setminus V_{x_k} ) < 2 \epsilon + \lambda( K \setminus V_{x_k} ).$$
\end{itemize}
This means that $[K \setminus V_{x_k}] \neq 0$, and also $[K \setminus V_{x_k}] \leq [K]$ but $[K \setminus V_{x_k}]$ cannot be equal to $[K]$, since its measure is strictly smaller. This is a contradiction to minimality of $[K]$.
\end{proof}

\subsection{Density of $p_\mu$}
	
We will describe a useful property of the map $p_\mu : X^\mu \rightarrow X$ in the case of many practical situations.

\begin{definition}
Let $f : L \rightarrow M$ be a map of locales. Then $f$ is called \emph{dense} if $f_*(0) = 0$.
\end{definition}

\begin{proposition} \label{injectivity}
Let $X$ be a Hausdorff space and $\mu$ a Radon valuation. Assume that $\mu_*(U) > 0$ for all non-empty open sets $U \subset X$, or equivalently that for all non-empty open sets $U$ there exists $K \subset U$ compact with $\mu(K) > 0$. Then the map
$$p_\mu : X^\mu \rightarrow X$$
is dense.
\end{proposition}

\begin{example} An important class of examples of Radon valuations $\mu$ satisfying the condition $\mu(U) > 0$ for all non-empty open sets $U \subset X$ arises when $X$ is a locally compact Hausdorff group and $\mu$ is the Haar measure on $X$, as described in Section \ref{haarmeasure}.
\end{example}

\begin{example}
As a trivial example where Proposition \ref{injectivity} fails, consider the Dirac measure $\delta_x$ at a point $x \in X$. For this example we have that $X^\mu = \mathrm{pt}$, and density of $p_\mu$ only if $X$ itself is a point.
\end{example}

\begin{proof}
Under the composition
$$ \mathcal{O}(X^\mu) \hookrightarrow \mathcal{O}(X^\mathcal{K}) \xrightarrow{\theta_*} \mathcal{O}(X) $$
the element $0$ is first sent to the ideal $N = \{ K \subset X ~|~ \mu(K) = 0 \}$, then further to
$$\theta_*(N) = \bigvee_{ U ~:~ \theta^*(U) \leq N } U = \bigvee \{ U ~|~ \mu(K) = 0 \text{ for all } K \subset U \text{ compact} \} = \bigvee_{ U ~:~ \mu_*(U) = 0}  U. $$
This open is the empty set iff $\mu_*(U) > 0$ for all non-empty opens.
\end{proof}

Recall that dense continuous maps between topological spaces behave like epimorphisms with respect to target spaces that are Hausdorff. A similar situation holds for dense maps between locales, although the situation is not quite equivalent.

\begin{definition}
A locale $L$ is called \emph{Hausdorff}, (also \emph{Isbell-Hausdorff}), if the diagonal
$$ L \rightarrow L \times L$$
is a closed map of locales.
\end{definition}

\begin{remark}
Hausdorff topological spaces need not be Isbell-Hausdorff if they are considered as locales. The problem lies in the fact that products of topological spaces are not necessarily preserved by the functor $\mathrm{Top} \rightarrow \mathrm{Loc}$. This subtlety disappears for locally compact spaces, where the two notions of Hausdorffness agree. It is also true that Isbell-Hausdorffness of the corresponding locale of a topological space $X$ implies that $X$ is Hausdorff.
\end{remark}

\begin{proposition}[\cite{picado_pultr}, V 2.5.3]
Let $f : L \rightarrow M$ be a dense map between locales, and $g_1, g_2 : M \rightarrow N$ be two maps with $N$ a Hausdorff locale. Suppose $g_1 f = g_2 f$. Then $g_1 = g_2$.
\end{proposition}

We obtain as an immediate consequence.

\begin{corollary} \label{densityinjectivity}
Suppose the situation of Proposition \ref{injectivity} holds, and let $Y$ be a Hausdorff locale. Then the continuous map $ p_\mu : X^\mu \rightarrow X$ induces an injection
$$ \mathrm{Map}( X, Y ) \rightarrow \mathrm{Map}( X^\mu, Y )$$
of continuous maps into random variables.
\end{corollary}

\begin{remark} One common criticism towards a point-free approach to measure theory is the statement that modding out by null sets makes it awkward or even impossible to study the local behaviour of equivalence classes of measurable functions. Proposition \ref{injectivity} shows that this fear is mostly misplaced. Not only does the point-free approach avoid talking about equivalence classes altogether, but moreover in many practical applications it is simply a (local) property of a map $X^\mu \rightarrow Y$ to be induced by a continuous map $X \rightarrow Y$, assuming one asks for the existence of enough open neighbourhoods of non-zero measure.
\end{remark}

\section{Measures on locally compact Hausdorff spaces}

Quite a few things simplify when working with regular contents on \emph{locally compact} Hausdorff spaces. Recall that a space $X$ is called locally compact, if for each open $U$ and $x \in U$, there exist $V,K$ such that $x \in V \subset K \subset U$ with $V$ open and $K$ compact. It follows that the same holds when $x$ is replaced by a compact subset.

\begin{lemma} \label{compactopenseparation}
Let $X$ be locally compact Hausdorff and $K \subset U$ with $K$ compact and $U$ open. Then there exist $V$ open and $K'$ compact such that $K \subset V \subset K' \subset U$.
\end{lemma}

It follows that the relation $\ll$ between compact sets becomes \emph{interpolative}, which means that whenever $K \ll K'$ for two compact subsets, there exists $K''$ such that $K \ll K'' \ll K'$.

\begin{proof}
Choose for each $x \in K$ sets $V(x) \subset K(x) \subset U$ with $V(x)$ open and $K(x)$ compact. Since $K$ is compact, finitely many such $x$ suffice for the sets $V(x)$ to cover $K$. Define $V$ and $K'$ as these finite unions.
\end{proof}

\begin{remark}
Unlike the notion of Hausdorffness, the notion of locally compact Hausdorff spaces can be done in a completely point free fashion, as the category of locally compact Hausdorff spaces is equivalent to the category of completely regular and locally compact locales, where a locale $L$ is called locally compact if its frame $\mathcal{O}(L)$ is a \emph{continuous} frame. For more details, see \cite{picado_pultr}, Chapters V.5 and VII.6. 
\end{remark}

\begin{lemma} \label{localfinitenesslocallycompact}
Suppose $X$ is locally compact Hausdorff and $\mu$ a valuation on $\mathcal{K}(X)$. Then the associated measure $\mu_* : \mathcal{O}(X) \rightarrow [0,\infty]$ is locally finite. 
\end{lemma}

\begin{proof}
Recall that $\mu_*(U) = \sup_{K \subset U \text{ compact}} \mu(K)$ for an open $U$ by definition. If $U$ has compact closure $\bar{U}$, then $\mu(K) \leq \mu( \bar{U} ) < \infty$ for all $K \subset U$ compact and therefore also $\mu_*(U) \leq \mu( \bar{U} ) < \infty$. Since $X$ is locally compact, we have that
$$X = \bigcup \{ U ~|~ U \text{ open with compact closure} \} $$
is a union of opens with finite measure.
\end{proof}

Regular contents are particularly natural in the context of locally compact Hausdorff spaces, in which case they correspond uniquely to locally finite measures.

\begin{theorem} \label{regularcontentslocallycompact}
Let $X$ be a locally compact Hausdorff space. The assignment $$\begin{array}{rcl}
 (-)_* : \mathrm{RegCont}(X) & \rightarrow & \mathrm{Meas}_{\mathrm{loc.~fin.}}(X) \\
 \lambda & \mapsto & \lambda_*
\end{array}$$
is a bijection between the set of regular contents on $X$ and the set of locally finite measures on $X$.
\end{theorem}

\begin{proof}
Let us define an inverse to $(-)_*$. Given a locally finite measure $\mu$ on $X$, define
$$\begin{array}{rcl}
\mu^\mathcal{K} : \mathcal{K}(X) & \rightarrow & [0,\infty) \\
K & \mapsto & \mu^\mathcal{K}(K) = \inf \{ \mu(U) ~|~ K \subset U, U \text{ open} \}
\end{array}$$ 
First of all, this is well-defined. Since $\mu$ is locally finite, there exists a cover $U_i, i \in I,$ of $X$ such that $\mu(U_i) < \infty$. If $K$ is compact, it must be contained in finitely many of the $U_i$, therefore $K$ is contained in an open of finite measure, and hence $\mu^\mathcal{K}(K) < \infty$. Now let us show that $\mu^\mathcal{K}$ is a regular content.
\begin{itemize}
\item \emph{Monotone:} Let $K_1 \subset K_2$ be two compact sets. Then $ K_2 \subset U$ for $U$ open implies $K_1 \subset U$, therefore $\mu^\mathcal{K}(K_1) \leq \mu^\mathcal{K}(K_2)$.
\item \emph{Subadditive:} Let $K_1, K_2$ be compact. Then
$$\begin{array}{rcl}
\mu^\mathcal{K}( K_1 \cup K_2 ) &=& \inf \{ \mu(U) ~|~ K_1 \cup K_2 \subset U, U \text{ open} \} \\
 & \leq & \inf \{ \mu(U_1 \cup U_2) ~|~ K_1 \subset U_1,  K_2 \subset U_2, U_1, U_2 \text{ open} \} \\
 &\leq &  \inf \{ \mu(U_1) + \mu(U_2) ~|~ K_1 \subset U_1,  K_2 \subset U_2, U_1, U_2 \text{ open} \} \\
 &=& \mu^\mathcal{K}(K_1)  + \mu^\mathcal{K}(K_2) 
\end{array}$$
\item \emph{Additive on disjoint pairs:} Now let $K_1, K_2$ be compact such that $K_1 \cap K_2 = \emptyset$. Since $X$ is Hausdorff, by Lemma \ref{compactseparationlemma} there exists disjoint open sets $V_1, V_2$ such that $K_1 \subset U_1$ and $K_2 \subset U_2$. Therefore, if $K_1 \cup K_2 \subset U$ for $U$ open, we obtain $U_1 = U \cap V_1, U_2 = U \cap V_2$ such that $U_1$ and $U_2$ are disjoint and
$$  K_1 \subset U_1, K_2 \subset U_2 \text{ and } U_1 \cup U_2 \leq U. $$
Therefore
$$\begin{array}{rcl}
\mu^\mathcal{K}( K_1 \cup K_2 ) &=& \inf \{ \mu(U) ~|~ K_1 \cup K_2 \subset U, U \text{ open} \} \\
 & = & \inf \{ \mu(U_1 \cup U_2) ~|~ K_1 \subset U_1,  K_2 \subset U_2, U_1, U_2 \text{ open and disjoint} \} \\
 & = &  \inf \{ \mu(U_1) ~|~ K_1 \subset U_1,  U_1 \text{ open} \} + \inf \{ \mu(U_2) ~|~ K_2 \subset U_2,  U_2 \text{ open} \} \\
 &=& \mu^\mathcal{K}(K_1)  + \mu^\mathcal{K}(K_2).
\end{array}$$
\item \emph{Regular:} Let $K$ be compact. By definition if $K \ll K'$, there exists open $U$ such that $K \subset U \subset K'$, hence $$\mu^\mathcal{K}(K) \leq \mu(U) \leq \mu^\mathcal{K}(K').$$
Conversely, if $K \subset U$ for $U$ open, since $X$ is locally compact Hausdorff, by Lemma \ref{compactopenseparation} there exist $K \subset V \subset K' \subset U$ with $V$ open and $K'$ compact. Therefore
$$ \mu^\mathcal{K}(K) =  \inf \{ \mu(U) ~|~ K \subset U, U \text{ open} \} =  \inf \{ \mu^\mathcal{K}(K') ~|~ K \ll K' \}.$$
\end{itemize}
We now claim that $\mu$ will be recovered from $\mu^\mathcal{K}$. If $U$ is open, we have
$$U = \bigcup_{ V \ll U } V $$
where we write $ V \ll U$ whenever there exists a compact $K$ such that $V \subset K \subset U$. Using continuity of $\mu$ we see that
$$\mu(U) = \sup_{  V \ll U } \mu(V) = \sup_{  K \subset U, \text{ compact} } \mu^\mathcal{K}( K ) = (\mu^\mathcal{K})_*(U).$$
Conversely, if we start with a regular content $\lambda$, we know that 
$$\lambda(K) = \inf_{ K \ll K' } \lambda(K') = \inf_{ K \subset U, \text{ open} } \lambda_*(U) = (\lambda_*)^\mathcal{K}(K).$$
This shows that $(-)_*$ and $(-)^\mathcal{K}$ are inverse to each other.
\end{proof}

\begin{remark}
Since pushforward of a locally finite measure against a partially defined proper map $f : X \rightarrow Y$ with open support again produces a locally finite measure (Proposition \ref{pushforwardproperties}), this means that we obtain a functor
$$ \mathrm{RegCont} : \mathrm{LocCHaus}_{\mathrm{part}~\mathrm{prop}} \rightarrow \mathrm{Set}$$
with $\mathrm{LocCHaus}_{\mathrm{part}~\mathrm{prop}}$ being the category of locally compact Hausdorff spaces and partially defined proper maps with open support between them, and $\mathrm{RegCont}(X) \cong \mathrm{Meas}_{\mathrm{loc. fin}}(X)$ the set of regular contents, and can be identified classically with the set of positive Radon measures on $X$. This functor is a sheaf when restricted to open inclusions, as shown in Theorem \ref{measuresaresheaves}.
\end{remark}

\begin{example} As a slightly exotic, but interesting example of a regular content, consider Example \ref{scissorscongruence1}. Recall that we fix $X$ to be a geometry, such as Euclidean space $E^n$, the sphere $S^n$ or hyperbolic space $H^n$, of dimension $n$, and consider the associated lower bounded distributive lattice $D(X)$ given by $n$-dimensional polytopes. The locally coherent space $L(D(X), fin) = {{X_{\mathrm{poly}}}}$ is in fact Hausdorff, since $D(X)$ is a Boolean ring (see Definition \ref{Booleanring} and Lemma \ref{Booleanringhausdorff} below), and comes with a continuous, proper map ${X_{\mathrm{poly}}} \rightarrow X$, which is $G$-equivariant, where $G$ is the group of isometries of $X$, and for which $X_{\mathrm{poly}}$ is a polytopal, highly disconnected version of the classical cover $C \rightarrow [0,1]$ by the Cantor set.

Extending the volume of a polytope produces the locally finite measure $\mathrm{Vol}_n$ on $ {X_{\mathrm{poly}}}$, and therefore a regular content on the set of compact subsets of ${X_{\mathrm{poly}}}$ (generated via intersections of compact opens, which correspond formally to polytopes). This means there is an associated measurable locale $(X_{\mathrm{poly}})_{Leb} = (X_{\mathrm{poly}})^{\mathrm{Vol}_n}$ together with an induced faithful and locally finite measure $\mathrm{Vol}_n$ and a measure-preserving map
$$(X_{\mathrm{poly}})_{Leb} \rightarrow X_{\mathrm{poly}},$$
and with it an integration theory built purely formally from the combinatorial (!) notion of $n$-dimensional polytopes, together with their volume, which works formally the same as traditional Lebesgue measure. If $G$ is a group of isometries of $X$, it acts naturally on $(X_{\mathrm{poly}})_{Leb}$.

Furthermore, pushforward of the measure along $X_{\mathrm{poly}} \rightarrow X$ equips $X$ with the classical Lebesgue measure, as for any open set $U \subset X$, it is a standard fact that
$$\mu(U) = \sup \{ \mu(P) ~|~ P \subset U,~ P \text{ is } n\text{-dim. polytope} \}.$$
In this sense, Lebesgue measure is obtained from the combinatorial notion of measure of a polytope, a fact that should not be too surprising, but nonetheless pleasant to observe from purely formal constructions.
\end{example}

\begin{remark}
If $X$ is locally compact Hausdorff, then any valuation $\nu : \mathcal{K}(X) \rightarrow [0, \infty)$ can be used to construct a regular content. The measure $\nu_*$ on $X^\nu = L( \mathcal{K}(X), \nu)$ induces via pushforward along $p_\nu : X^\nu \rightarrow X$ a measure on $X$. This measure is locally finite, and hence induces the regular content
$$ \tilde{\nu}(K) = \inf_{ K \ll K' } \nu(K').$$ 
\end{remark}

We want to state a similar statement with regard to contents. Any content on a locally compact Hausdorff space can be used to construct a regular content.

\begin{lemma} \label{inducedregularcontent} Let $X$ be a locally compact Hausdorff space and let $\lambda : \mathcal{K}(X) \rightarrow [0,+\infty)$ be a content. Then the function
$$\begin{array}{rcl}
\tilde{\lambda} : \mathcal{K}(X) & \rightarrow & [0,+\infty) \\
K  & \mapsto & \inf \{ \lambda(K') ~|~ K  \ll K' \}
\end{array}$$
is a regular content on $X$. The associated inner measure $\tilde{\lambda}_*$ on $\mathcal{O}(X)$ is given by
$$ \tilde{\lambda}_*(U) = \sup_{ K \subset U \text{ compact}} \lambda(K). $$
\end{lemma}

\begin{remark}
Traditionally, Lemma \ref{inducedregularcontent} appears quite often implicitly when constructing a Radon measure. Typically, a content $\lambda$ is used to construct an induced Radon measure, see e.g.\ Halmos \cite[\S 53 and \S 54]{halmos2013measure}. Restricting this induced measure to the set of compact subsets gives $\tilde{\lambda}$. The proof of Lemma \ref{inducedregularcontent} follows closely with the proof of \S 54, Theorem C in Halmos.
\end{remark}

\begin{proof}
We verify the four conditions.
\begin{itemize}
\item \emph{Monotone:} Let $K_1 \subset K_2$ be two compact sets. Suppose $K_2 \ll K'$. Then clearly $K_1 \ll K'$ as well. Therefore
$$\tilde{\lambda}(K_1) = \inf \{ \lambda(K') ~|~ K_1  \ll K' \} \leq \inf \{ \lambda(K') ~|~ K_2  \ll K' \}  = \tilde{\lambda}(K_2).$$
\item \emph{Subadditive:} Let $K_1$ and $K_2$ be two compact sets. Suppose $K_1 \ll K_1'$ and $K_2 \ll K_2'$. Then clearly $K_1 \cup K_2 \ll K_1' \cup K_2'$. Therefore
$$\begin{array}{rl}
  & \tilde{\lambda}(K_1 \cup K_2) \leq \inf \{ \lambda(K_1' \cup K_2') ~|~ K_1  \ll K_1' \text{ and } K_2 \ll K_2' \} \\ \leq  & \inf \{ \lambda(K_1') + \lambda(K_2') ~|~ K_1  \ll K_1' \text{ and } K_2 \ll K_2' \} = \tilde{\lambda}(K_1) + \tilde{\lambda}(K_2). 
\end{array} $$
\item \emph{Additive on disjoint pairs:} Let $K_1$ and $K_2$ be two disjoint compact sets. Since $X$ is Hausdorff, the two compact sets can be separated, i.e.\ there exists disjoint opens $U_1, U_2$ such that $K_1 \subset U_1$ and $K_2 \subset U_2$ by Lemma \ref{compactseparationlemma}. Using that $X$ is locally compact as well, by Lemma \ref{compactopenseparation} we also find $K_1', K_2'$ compact such that $K_1 \ll K_1' \subset U_1$ and $K_2 \ll K_2' \subset U_2$. In particular, $K_1'$ and $K_2'$ are disjoint as well. Therefore, if $K_1 \cup K_2 \ll K'$ for some compact $K'$, we have $$K_1 \cup K_2 \ll (K' \cap K_1') \cup (K' \cap K_2') \subset K'$$ as well, with $K_1 \ll (K' \cap K_1')$ and $K_2 \ll (K' \cap K_2')$, and the sets $K' \cap K_1'$ and $K' \cap K_2'$ being disjoint. We conclude that
$$\begin{array}{c}
 \hspace{5ex} \tilde{\lambda}(K_1 \cup K_2) = \inf \{ \lambda(K_1' \cup K_2') ~|~ K_1  \ll K_1', K_2 \ll K_2'  \text{ and } K_1' \cap K_2' = \emptyset \} = \tilde{\lambda}(K_1) + \tilde{\lambda}(K_2). 
\end{array} $$
\item \emph{Regular:} Let $K$ be compact. Since $X$ is locally compact, whenever $K \ll K'$, there exists $K''$ compact such that $K \ll K'' \ll K'$. Therefore
$$\tilde{\lambda}(K) = \inf \{ \lambda(K') ~|~ K  \ll K' \} = \inf \{ \tilde{\lambda}(K'') ~|~ K \ll K'' \}.$$
\end{itemize}
For the claim about the induced measure on $\mathcal{O}(X)$, let $U \subset X$ be open. For any $K \subset U$ compact there exists $K'$ compact with $K \ll K'  \subset U$ by Lemma \ref{compactopenseparation} hence
$$\tilde{\lambda}_*(U) = \sup_{ K \subset U \text{ compact}} \tilde{\lambda}(K) = \sup_{ K \subset U \text{ compact}} \inf_{K \ll K'} \lambda(K') = \sup_{ K \subset U \text{ compact}} \lambda(K). $$
\end{proof}

\subsection{A representation theorem for measurable locales} \label{sectionrepresentationtheorem}

Next we want to address a type of representation theorem: Every measurable locale arises from a regular content on a locally compact Hausdorff space.

\begin{theorem} \label{representationtheorem}
Let $L$ be a measurable locale. Then there exists a locally compact Hausdorff space $X$ together with a regular content $\lambda$ such that $L \cong X^\lambda$.
\end{theorem}

\begin{remark}
Theorem \ref{representationtheorem} has a classical sister theorem in the context of operator algebras, that states that every commutative von Neumann algebra is obtained as $L^\infty(X, \mu)$ for a Radon measure $\mu$ on a locally compact Hausdorff space $X$, see \cite[V Theorem 1.18.]{Takesaki1979}. In fact, under the duality provided by Theorem \ref{gelfanddualityvonneumann}, these two theorems are equivalent.
\end{remark}

Before we prove Theorem \ref{representationtheorem}, let us give a slight generalization of classical Stone-duality for profinite spaces.

\begin{definition} \label{Booleanring}
A \emph{Boolean ring} is a lower bounded distributive lattice $(D,\leq)$, such that for each $c, d \in D$ there exists an element $d \setminus c$ such that $c \wedge d \setminus c = 0$ and $ c \vee d \setminus c = c \vee d $.
\end{definition}

\begin{lemma} \label{Booleanringhausdorff}
Let $D$ be a Boolean ring and let $L$ be the locale associated to the frame $\mathrm{Ind}(D) \cong \mathrm{Sh}(D,fin;\mathbf{2})$. Then $L$ is spatial, corresponding to a locally compact Hausdorff space.
\end{lemma}

\begin{remark}
For a Boolean ring $D$, the posets $D_{/d}$ always form Boolean algebras. Under Stone duality, this means that $X$ is in fact locally profinite, i.e.\ has a cover by profinite spaces.
\end{remark}

\begin{proof}
The locale $L$ is locally coherent, and thus spatial and locally compact, as described in \cite[Section 3.1]{lehner2025algebraicktheorycoherentspaces}. To see that it is Hausdorff, note that by Example \ref{point} points of $L$ are given by prime filters of $D$. Let $x \neq y \subset D$ be two such prime filters. Since both are non-empty filters, there must exist $d_x \in x$ and $d_y \in y$. Then $d = d_x \vee d_y$ is a common element in both $x$ and $y$. Furthermore, since $x$ and $y$ are distinct filters, there needs to exist $c \in x, c \not\in y$. We have $c \vee d \setminus c = d \in y$, and also $c \not\in y$, it follows that since $y$ is a prime filter, we have $d \setminus c \in y$. But viewed geometrically, this means that the representables $[c]$ and $[d \setminus c]$ in $\mathrm{Ind}(D)$ are two disjoint opens that separate $x$ and $y$.
\end{proof}

\begin{remark}
In fact, the above proof shows that $X$ is totally disconnected as well: Any two points can be separated via compact open subsets, which in the presence of Hausdorffness are closed and open.
\end{remark}

\begin{lemma} \label{intersectioncompactopen}
Let $D$ be a Boolean ring and let $X$ be the topological space associated to the frame $\mathrm{Ind}(D) \cong \mathrm{Sh}(D,fin;\mathbf{2})$. Assume $K \subset X$ is a compact subset. Then $K = \bigcap_{ i \in I } U_i$ for a collection of compact opens $U_i$.
\end{lemma}

\begin{proof}
Let $K \subset X$ be a compact subset. Since $1 = \bigvee_{d \in D} [d]$ in $\mathcal{O}(X)$ is a directed union of compact opens, there exists a compact open $U$ such that $K \subset U$. Furthermore, every compact set is also closed, since $X$ is Hausdorff, therefore every open is obtained via unions from clopen sets. By passing to complements, we get that every closed set is obtained as an intersection of clopen sets, in particular $K$. By intersecting with $U$, we can assume this intersection consists of compact opens.
\end{proof}

\begin{proof}[Proof of Theorem \ref{representationtheorem}]
Choose a locally finite and faithful measure $\mu$ on $L$ and consider the Boolean ring $D = \mathcal{O}(L)^{\mu\text{-}fin}$ of $\mu$-finite chunks of $L$.  Define $X$ to be the locally compact Hausdorff space associated to the frame $\mathrm{Ind}(D) \cong \mathrm{Sh}(D,fin;\mathbf{2})$. Since $L$ is given by the frame $\mathrm{Sh}(D,\mu\text{-}inn;\mathbf{2})$ by Theorem \ref{valuationframedeterminedbyfinite}, and the $\mu$-inner topology contains all $fin$-coverings, we have an embedding
$$ i : L \hookrightarrow X.$$

Then the pushforward along $i$ equips $X$ with a measure $i_* \mu$. It is immediate that this measure is locally finite, as $X$ is generated by the set of compact opens (identified with) $D$, which have finite measure. By Theorem \ref{regularcontentslocallycompact} the measure $\mu$ induces a regular content $\lambda$ on $X$.

We claim that $L \cong X^\lambda$, or equivalently that 
$$ \mathrm{Sh}(D,\mu\text{-}inn;\mathbf{2}) \cong \mathrm{Sh}(\mathcal{K}(X),\lambda\text{-}inn;\mathbf{2}).$$
To see this, it suffices to show that the image of $D \cong \mathcal{K}^o(X)$ as the set of compact opens of $X$ is a basis for the right-hand side, and that the induced Grothendieck pretopology agrees with the $\mu$-inner topology.

Every compact set $K$ of $X$ is obtained as a directed intersection of compact opens $U_i, i \in I,$ by Lemma \ref{intersectioncompactopen}. But note that the right adjoint
$$ i_* : \mathrm{Sh}(D,\mu\text{-}inn;\mathbf{2}) \rightarrow \mathrm{Sh}(D,fin;\mathbf{2}) = \mathcal{O}(X)$$
preserves infima. Since elements of finite measure are closed under infima, and infima in $\mathcal{O}(X)$ are computed as interiors of intersections, we get that the open interior
$$ K^o = ( \bigcap_{ i \in I } U_i )^o = \bigwedge_{i \in I} U_i$$
is again compact open. We also see that the inclusion
$$ K^o \subset K $$
gives an equality in measure, i.e.\ $\lambda(K^o) = \lambda(K)$, as the right hand side is calculated as $\inf_{i \in I} \mu(U_i)$ by definition, and the left hand side is calculated by the same term by Lemma \ref{infformulaboolean}. Therefore $\{ K^o \subset K  \}$ is in fact a $\lambda$-approximation and we see that all representables $[K] \in 
\mathrm{Sh}(\mathcal{K}(X),\lambda\text{-}inn;\mathbf{2})$ can be covered by compact opens of $X$, in other words, the image of $D$ is a basis for $\mathrm{Sh}(\mathcal{K}(X),\lambda\text{-}inn;\mathbf{2})$. It is also clear that the restriction of the $\lambda$-inner pretopology on $\mathcal{K}(X)$ to $D$ agrees with the $\mu$-inner pretopology, as the functor $D \cong \mathcal{K}^o(X) \subset \mathcal{K}(X)$ is valuation-preserving.
\end{proof}

\section{Examples} \label{examples}

A theory is only as good as the examples it describes. As such, we will describe a few classical examples of measure spaces. The results here are well known---The main point we want to demonstrate is that the locale-theoretic picture is fairly easily constructed from well-known results in the literature.

\subsection{The Lebesgue reals} \label{lebesguereals}

The most fundamental example of a measure is that of the Lebesgue measure $\lambda_d$ on $\mathbb{R}^d$. Lebesgue measure is an example of a Radon measure, which means that restricting $\lambda_d$ to the set $\mathcal{K}(\mathbb{R}^d)$ of compact sets of $\mathbb{R}^d$ gives a regular content. Let us quickly sketch how one can construct this (regular) content without reference to the standard literature on measure theory and show some elementary properties.

Let $n \geq 0$. Define a \emph{dyadic cube of thinness $n$} to be a set of the form
$$\prod_{i = 1}^d \left[\frac{a_i}{2^n}, \frac{a_i+1}{2^n}\right] \subset \mathbb{R}^d$$
for values $a_1, a_2, \hdots, a_d \in \mathbb{Z}$. Call a set $D$ a \emph{standard set of thinness $n$} if it is given as a finite union of dyadic cubes of thinness $n$. Define Lebesgue measure $\lambda_d$ on the set of standard sets of thinness $n$ as
$$\lambda_d(D) = \frac{\#\{\text{distinct dyadic boxes of thinness } n \text{ contained in }D\}}{2^{nd}}.$$

\begin{figure}
    \centering
    \includegraphics[width=0.45\textwidth]{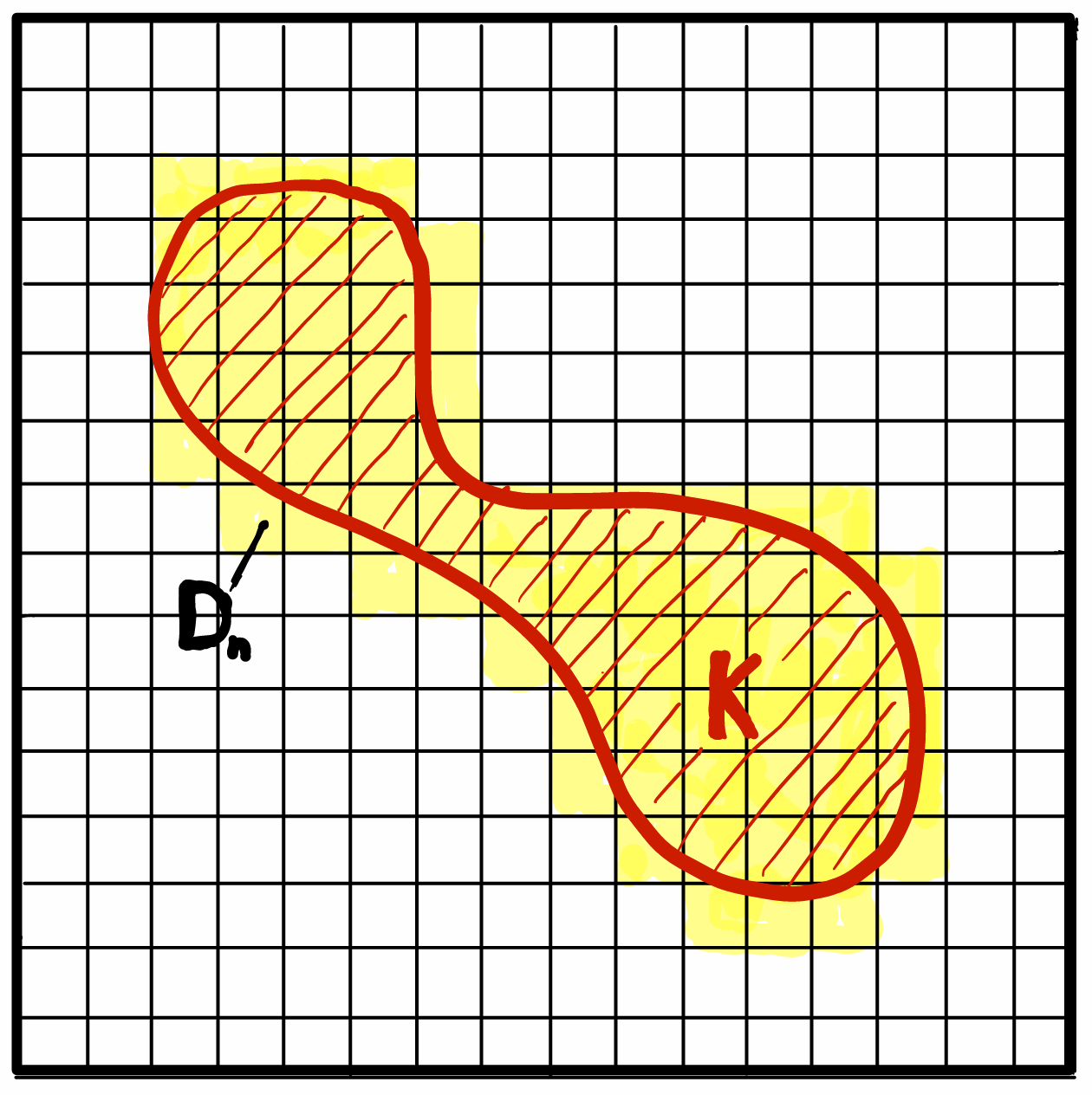}
    \caption*{Approximating a compact set $K$ from above via dyadic boxes}
\end{figure}

Any compact set $K \in \mathcal{K}(\mathbb{R}^d)$ is obtained as
$$K = \bigcap_{n \in \mathbb{N}} D_n$$
where $D_n$ is the smallest finite union of dyadic boxes of thinness $n$ containing $K$. Define the Lebesgue measure on $K$ as
$$\mu(K) = \inf_{n \in \mathbb{N}} \lambda_d(D_n).$$
It is elementary to check that $\lambda_d$ is a regular content on $\mathcal{K}(\mathbb{R}^d)$.

With this defined, we get a corresponding measurable locale $\mathbb{R}^d_{Leb} = \mathbb{R}^{\lambda_d}$ together with an induced measure $\lambda_d$ on the locale of sublocales $\mathfrak{Sl}(\mathbb{R}^d)$, and in particular on all open subsets of $\mathbb{R}^{\lambda_d}$. The locale $\mathbb{R}^d_{Leb}$ has no points for $d > 0$, as one sees by an application of Proposition \ref{pointsregularcontent}. One can also observe that any open $U \subset \mathbb{R}^d$ contains a dyadic cube and therefore satisfies $\lambda_d(U) > 0$, hence the canonical map $\mathbb{R}^d_{Leb} \rightarrow \mathbb{R}^d$ is dense by Proposition \ref{injectivity}. The relation to the classical Lebesgue measure is that the frame corresponding to $\mathbb{R}^d_{Leb}$ is in fact isomorphic to the (complete) Boolean algebra obtained by taking the quotient of the $\sigma$-algebra of Lebesgue measurable subsets of $\mathbb{R}^d$ modulo the ideal of $\lambda_d$-null sets, a fact that we will more generally prove with Theorem \ref{radonmeasures} in the last part of this article.

It is a geometric argument to see that $\lambda_d$ is translation invariant, i.e.\ if $v \in \mathbb{R}^d$, then $\lambda_d(S+v) = \lambda_d(S)$ for any sublocale $S$. Since the measurable locale $\mathbb{R}^d_{Leb}$ is obtained as the result of applying a functor to the valuation site $(\mathcal{K}(\mathbb{R}^d), \lambda_d)$ it suffices to show invariance of $\lambda_d$ when acting on compact subsets, which in turn reduces to the same statement for standard sets. If $v$ is a dyadic vector, i.e.\ of the form $\left(\frac{a_i}{2^n}\right)_{i = 1}^d$ for some $n \geq 0$ and $a_1, a_2, \hdots, a_d \in \mathbb{Z}$, the statement is clear. The case of a general translation follows from the density of dyadic numbers in $\mathbb{R}$.

One can also show that $\lambda_d$ is rotation invariant. Here one observes more generally that $\lambda_d$ is invariant under the action of the special linear group $\mathrm{SL}_d(\mathbb{R})$, which contains the group of rotations. The group $\mathrm{SL}_d(\mathbb{R})$ is generated by \emph{elementary matrices} $E_{ij}(a) = I + a e_{ij}$ with $a \in \mathbb{R}$ and $i \neq j$, where $I$ is the identity matrix, and $e_{ij}$ the matrix with single non-zero entry $1$ at position $(i,j)$, hence it suffices to verify that such elementary matrices act in a measure-preserving way. Again, it suffices to show this for dyadic boxes, for which it is an elementary geometric argument, as elementary matrices act via shearing transformations.

Let $m \geq n$. Denote by $p : \mathbb{R}^m \rightarrow \mathbb{R}^n$ the projection onto the first $n$ factors. Consider the restriction of $p|_{[0,1]^m}$ to the $m$-dimensional cube $[0,1]^m$ as a map $[0,1]^m \rightarrow [0,1]^n$. Since ${p|_{[0,1]^m}}^{-1}(K) = K \times [0,1]^{m-n}$, one observes that this map is measure-preserving. We obtain the classical corollary.

\begin{lemma} \label{projectioninverseimage}
Suppose $C \subset \mathbb{R}^n$ is a compact null set and $K \subset p^{-1}(C) \subset \mathbb{R}^m$ is compact. Then $K$ is again a compact null set.\footnote{This lemma of course also follows from the classical Fubini-Tonelli theorem.}
\end{lemma}

\begin{proof}
By covering $K$ with a finite (!) set of cubes, we may as well assume $K \subset [0,1]^m$ and $C \subset [0,1]^n$. But then the standard projection $p : [0,1]^m \rightarrow [0,1]^n$ is measure-preserving, therefore $\lambda_m(K) \leq \lambda_m( p^{-1}(C) ) = \lambda_n(C) = 0.$
\end{proof}

\subsection{The Haar measure on a locally compact Hausdorff group} \label{haarmeasure}

In the following, let $G$ be a locally compact Hausdorff group. We sketch the classical construction of the Haar measure on $G$.

\begin{definition}
A (left) \emph{Haar measure} on $G$ is a non-trivial, left invariant and locally finite measure $\mu$ on $G$.
\end{definition}

Concretely, we say that a measure $\mu : \mathcal{O}(G) \rightarrow [0, \infty]$ is \emph{left invariant} if the homeomorphisms $g \cdot - : G \rightarrow G$ are measure-preserving for every $g \in G$. 

A Haar measure automatically satisfies that $\mu(U) > 0$ for all non-empty opens $U$. W.l.o.g.\ assume that the neutral element $e \in U$, and let $K \subset G$ be a compact set. Then
$$K \subset \bigcup_{k \in K} kU,$$ hence by compactness of $K$ a finite set of translations of $U$ covers $K$. Therefore, if there exists a single non-empty open set $U$ with $\mu(U) = 0$, the induced regular content on all compact sets must be trivial, and hence $\mu$ itself is trivial as well. Let us state the classical existence-and-uniqueness theorem of the Haar measure.

\begin{theorem}
Let $G$ be a locally compact Hausdorff group. Then there exists a Haar measure $\mu$.
\end{theorem}

\begin{proof}[Sketch of proof]
Since $G$ is assumed locally compact, the construction of a left-invariant locally finite measure on $G$ is equivalent to the the construction of a left-invariant regular content by Theorem \ref{regularcontentslocallycompact}. It furthermore suffices to give a left-invariant (not-necessarily regular) content $\lambda$ by Lemma \ref{inducedregularcontent}. We sketch the standard argument given in e.g.\ Halmos \cite[Chapter XI]{halmos2013measure} for the construction of such $\lambda$. 
Given a compact set $U$ with non-empty interior, and a compact set $K$, define the \emph{ratio}
 $$(K : U) = \min \left\lbrace  \# I \mid I \subset G, \text{ s.t. }  K \subset \bigcup_{g \in I} gU \right\rbrace.$$
This number is finite: By non-emptiness of $U$ there exists a cover of $K$ by sets of the form $gU$ for $g \in G$, and since $K$ is compact, a finite set of such elements $g$ suffices to cover $K$.

Now fix a compact set $A \subset G$ with non-empty interior (such a set exists since $G$ is assumed to be locally compact Hausdorff). We think of $A$ as being a \emph{reference scale}. Then define for each compact $U$ with non-empty interior the function
$$ \begin{array}{rcl}
\lambda_U : \mathcal{K}(G) & \rightarrow & [0,\infty) \\
                   K  & \mapsto & \frac{(K : U)}{(A : U)}. 
\end{array}$$
One verifies that this is monotone, subadditive and left invariant, and satisfies additivity for a restricted class of compacts, see \cite[\nopp \S 58 Theorem A]{halmos2013measure}. If $K_1$, $K_2$ are compact sets such that $K_1 U^{-1} \cap K_2 U^{-1} = \emptyset,$ then
$$  \lambda_U (K_1 \cup K_2) =  \lambda_U (K_1) +  \lambda_U (K_2). $$
Furthermore it is easy to verify that $0 \leq \lambda_U(K) \leq ( K : A )$.
In order to obtain an actual content  $\lambda$, one needs to take a suitable limit over the net indexed by the neighborhood filter $\mathcal{U}^c(e)$ of compact neighborhoods of the identity element $e \in G$, that sends $U$ to the function $\lambda_U$. This may not exist uniquely, but to show the existence of a candidate for $\lambda$ one argues that this net lives in the compact Hausdorff space $ \prod_{ K \in \mathcal{K}(G) } [ 0, (K : A)] $ (by Tychonoff's theorem), and then chooses a convergent sub-net. Finally, one verifies that the corresponding limit $\lambda$ actually fulfills the conditions for being a left-invariant content \cite[\nopp \S 58 Theorem B]{halmos2013measure}.
\end{proof}

\begin{remark}
The use of choice for the construction of $\lambda$, which then has to further be modified to define the actual left invariant regular content
$$ \tilde{\lambda}(K) = \inf_{ K \ll K' } \lambda(K')$$
is somewhat unsatisfactory, as the actual value for a given compact set $K$ is quite inexplicit. Another argument for the construction of $\tilde{\lambda}(K)$ via the construction of a left-invariant mean on the space $C_c(G)$ of compactly supported real-valued functions on $G$ is due to Alfsen \cite{Alfsen1963}. It would be preferable to have an elementary argument that gives the value of $\tilde{\lambda}(K)$ directly.
\end{remark}

\begin{example}
Suppose $G$ is a locally compact Hausdorff group and $K$ a compact subgroup. Then the quotient map $G \rightarrow G / K$ is a proper map, and therefore Proposition \ref{pushforwardproperties} equips $G / K$ with a locally finite measure $\mu$ via pushforward of a choice of Haar measure on $G$. This measure can be shown to be $G$-invariant \cite[Corollary 2.53]{folland2016course}.

We obtain many classical examples together with their measures as a corollary: Hyperbolic $n$-space $H^n = \mathrm{SO}^+(n,1) / \mathrm{SO}(n)$, and the $n$-sphere $S^n = \mathrm{O}(n) / \mathrm{O}(n-1)$ to mention just two.
\end{example}

%
%

\section{Measurable locales generated from Hausdorff spaces}

It can sometimes be useful to have a way to generate measurable locales without the need to specify one particular measure. We will discuss this in the following. Recall the notion of a compactly marked Hausdorff space $(X,N)$ (Definition \ref{compactlymarked}), given by a Hausdorff space $X$ and an ideal $N \subset \mathcal{K}(X)$ of compact subsets. If $U \subset X$ is an open subset, denote by $N|_U \subset \mathcal{K}(U)$ the ideal of compacts $K \in N$ such that $K \subset U$.

\begin{definition}
Let $X$ be a Hausdorff space. An ideal $N \subset \mathcal{K}(X)$ of compact sets is called \emph{measurable} if there exists a cover $U_i, i \in I$ of $X$ such that for each $i \in I$ there exists a Radon valuation $\mu_i$ on $U_i$ such that $N|_{U_i}$ agrees with the set of compact $\mu_i$-null sets. We will refer to compacts $K \in N$ as \emph{null sets}.

The category $\mathrm{MMHS}$ of measurably marked Hausdorff spaces is defined as the corresponding full subcategory of compactly marked Hausdorff spaces $\mathrm{CMHS}$.
\end{definition}

Observe that for any $K \in N$, since $K$ is compact, there exist finitely many $i_1, \hdots, i_n \in I$ such that $K \subset U_{i_1} \cup \cdots \cup U_{i_n}$. Using Lemma \ref{compactdecompositionlemma}, we see that $K$ decomposes as
$$ K = K_{i_1} \cup \cdots \cup K_{i_n} $$
with $K_{i_j} \in N|_{U_{i_j}}.$ In other words, $N$ is generated under finite unions from the ideals $N|_{U_{i}}, i \in I,$ or equivalently $N = \bigvee_{i \in I} N|_{U_{i}}$ in $\mathcal{O}(X^\mathcal{K})$. We think of this as saying that an abstract compact null set is equivalently given as a finite union of local pieces that are actual null sets.

\begin{remark} Let $(X, \mu), (Y, \nu)$ be two Hausdorff spaces equipped with Radon valuations, viewed as measurably marked in the canonical way. Denote by $N$ the set of compact $\mu$-null sets and by $M$ the set of compact $\nu$-null sets. Recall that a continuous map $f : X \rightarrow Y$ is called compatible if:
\begin{enumerate}
\item For all $K \in N$ we have $f(K) \in M$.
\item For all $C \in M$ we have $\{ K \subset f^{-1}(C) \text{ compact} \} \subset N$. 
\end{enumerate}
If $f$ is measure-preserving, then condition (2) is automatically satisfied. Hence a measure-preserving map is compatible if images of compact null sets remain compact null sets.

The requirement (1) is not automatically satisfied for a measure-preserving map. Let $(X,\delta_x)$ be an arbitrary space equipped with the Dirac measure at a point $x \in X$. Then the unique map $f : (X,\delta_x) \rightarrow (\mathrm{pt}, \delta_{\mathrm{pt}})$ is measure-preserving. However, if $X$ contains another point $y$, it does not satisfy the condition that null sets are preserved, since
$$0 = \delta_x(\{y\}) \neq \delta_{\mathrm{pt}}( f(\{y\}) ) = \delta_{\mathrm{pt}}( \mathrm{pt} ) = 1. $$
\end{remark}

Recall the functor
$$\begin{array}{rcl}
b : \mathrm{CMHS} & \rightarrow & \mathrm{BoolLoc} \\
     (X,N) &\mapsto & b(X,N)
\end{array}$$
provided by Corollary \ref{booleanoutofcompact}.

\begin{proposition} \label{measurableoutofmeasurablymarked}
Let $(X,N)$ be a measurably marked Hausdorff space. Then $b(X,N) = (N^c)_{\neg \neg} \hookrightarrow X^\mathcal{K}$ is a measurable locale.
\end{proposition}

\begin{proof}
Choose a covering $U_i, i\in I,$ of $X$ together with Radon valuations $\mu_i$ on $U_i$ such that $N|_{U_i}$ agrees with the ideal of compact $\mu_i$-null sets. Pulling back this covering along the map $p : b(N) \hookrightarrow X^\mathcal{K} \rightarrow X$ gives a covering $p^*( U_i ), i \in I,$ of $b(N)$. We can identify the individual opens as $p^*(U_i) = b(N|_{U_i})$: For each $i \in I$, we have the composition of pullback squares
$$\begin{tikzcd}
b(N|_{U_i}) \ar[r, hook] \ar[d, hook] \arrow[dr, phantom, "\scalebox{1.0}{$\lrcorner$}", very near start, color=black] & b(N) \ar[d, hook] \\
{U_i}^\mathcal{K} \ar[r, hook] \arrow[dr, phantom, "\scalebox{1.0}{$\lrcorner$}", very near start, color=black] \ar[d] & X^\mathcal{K} \ar[d] \\
U_i \ar[r, hook] & X
\end{tikzcd}$$
with the top square given by Lemma \ref{openpullbackdoublenegation} and the bottom square given by Remark \ref{pullbackofopens}. The locale $b(N|_{U_i})$ agrees with the inner locale ${U_i}^{\mu_i}$ by Theorem \ref{almostbooleangivesboolean}, and is thus equipped with the faithful and locally finite measure ${\mu_i}_*$, hence $b(N)$ can be covered by opens equipped with faithful and finite measures.
\end{proof}

The following theorem is now immediate from Corollary \ref{booleanoutofcompact} and Proposition \ref{measurableoutofmeasurablymarked}.

\begin{theorem} \label{functormeasurablymarked}
The functor $b$ restricts to a well-defined functor
$$ b : \mathrm{MMHS} \rightarrow \mathrm{MblLoc}. $$
\end{theorem}

\subsection{The measurable locale associated to a smooth manifold}

Theorem \ref{functormeasurablymarked} has an immediate application in the construction of a functor on smooth manifolds, that associates to a smooth manifold $M$ the measurable \emph{Lebesgue} locale  $M_{Leb}$ together with a natural map
$$ M_{Leb} \rightarrow M.$$
Recall that an $n$-dimensional \emph{chart} $(U,\varphi)$ of a topological space $M$ consists of an open subset $U \subset M$, together with a homeomorphism $\varphi : U \rightarrow \varphi(U) \subset \mathbb{R}^n$, where $\varphi(U)$ is an open subset of $\mathbb{R}^n$. A Hausdorff space $M$ is called a \emph{large smooth n-manifold} if there exists a collection of charts $(U_i, \varphi_i), i \in I,$ such that $\{U_i \subset M ~|~ i \in I\}$ is a covering of $M$, and all transition functions $\varphi_j {\varphi_i}^{-1} : \varphi_i( U_i \cap U_j ) \rightarrow \varphi_j( U_i \cap U_j )$ are smooth maps. (Compare, e.g.,\ \cite{Lee2003} Chapter 1)

Note that any large smooth manifold is a locally compact Hausdorff space. We call a large smooth manifold $M$ of \emph{standard size} or simply a \emph{smooth manifold} if it is second-countable.

\begin{definition}
Let $M$ be a (large) smooth manifold of dimension $n$. A compact subset $K \subset M$ is called \emph{null} if there exists a finite collection of charts $(U_i, \varphi_i), i \in I,$ such that $K$ can be written as $K = \bigcup K_i$ with $K_i \subset U_i$ a compact subset, and $\varphi_i(K_i) \subset \varphi_i(U_i) \subset \mathbb{R}^n$ being a null set for the Lebesgue measure on $\mathbb{R}^n$.
\end{definition}

Note that any compact set $K \subset M$ is always contained in a finite collection of charts. For this reason, for many proofs we may assume w.l.o.g.\ that $K$ is actually a compact subset of a smooth manifold of standard size. The main reason to require smoothness for our manifolds is the following crucial lemma.\footnote{Strictly speaking, the use of $C^1$-manifolds would suffice.}

\begin{lemma}[\cite{Lee2003} Theorem 6.9] \label{imageofnullundersmooth}
Let $f : M \rightarrow N$ be a smooth map between two smooth manifolds $M,N$. Then $f$ maps compact null sets to null sets.
\end{lemma}

\begin{lemma}
Let $K \subset M$ be a compact subset of a large smooth $n$-manifold $M$. The following are equivalent.
\begin{enumerate}
\item $K$ is a compact null set.
\item For any covering of $M$ by charts $(U_i, \varphi_i), i \in I,$ there exist a \emph{finite} subset $F \subset I$ and compact subsets $K_i \subset U_i$ for each $i \in F$, such that $K = \bigcup_{i \in F} K_i$ and $\varphi_i(K_i) \subset \varphi_i(U_i) \subset \mathbb{R}^n$ is a null set for the Lebesgue measure on $\mathbb{R}^n$ for each $i \in F$.
\end{enumerate}
\end{lemma}

\begin{proof}
The implication from (2) to (1) is immediate. Now assume that $K$ is a compact null set, i.e.\ there exists a finite collection of charts $(V_j, \psi_j), j \in J,$ and compact sets $K_j \subset V_j$, such that $K = \bigcup_{j \in J} K_j$ and $\psi_j(K_j) \subset \mathbb{R}^n$ is a Lebesgue null set for each $j \in J$. Let $(U_i, \varphi_i), i \in I,$ be a covering of $M$ by charts. Since $K$ is compact and $M$ is Hausdorff, there exists a finite subset $F \subset I$ and  $K_i \subset U_i$ compact subsets for $i \in F$ such that $K = \bigcup_{i \in F} K_i$. 

We need to argue that $\varphi_i(K_i) \subset \varphi_i(U_i)$ is a Lebesgue null set. We can write $K_i = \bigcup_{j \in J} K_i \cap K_j$ with $K_i \cap K_j \subset U_i \cap V_j$ for each $j \in J$. Now since $\psi_j(K_j) \subset \mathbb{R}^n$ is a Lebesgue null set, the same holds for $\psi_j(K_i \cap K_j)$. Since the transition function $\varphi_i \psi_j^{-1}$ is a smooth function, by \ref{imageofnullundersmooth} the same holds for the compact set $\varphi_i(K_i \cap K_j)$. But this means that $\varphi_i(K_i) = \bigcup_{j \in J} \varphi_i(K_i \cap K_j)$ is a finite union of compact Lebesgue null sets, and hence itself a Lebesgue null set.
\end{proof}

Note that in particular, if a compact set $K$ is contained entirely in a chart $(U,\varphi)$, it is null iff its image in $\varphi(U) \subset \mathbb{R}^n$ is a Lebesgue null set. It is also straightforward to see that finite unions of compact null sets remain null sets. Hence the set of compact null sets forms an ideal $N$.

\begin{definition}
Let $M$ be a large smooth manifold. Let $N \in \mathcal{O}(M^\mathcal{K})$ be the ideal of compact null sets. Then
the \emph{Lebesgue locale} associated to $M$ is defined as
$$M_{Leb} = b(N) \hookrightarrow M^\mathcal{K}.$$
It comes with a canonical map $\mathrm{can} : M_{Leb} \hookrightarrow M^\mathcal{K} \xrightarrow{\theta} M.$
\end{definition}

\begin{proposition}
Let $M$ be a large smooth manifold of dimension $n$ and let $N$ be the ideal of compact null sets of $M$. Then $N$ is a measurable ideal. In particular, $M_{Leb}$ is a measurable locale.
\end{proposition}

\begin{proof}
Choose a covering of $M$ by charts $(U_i, \varphi_i), i \in I$. Then $\varphi_i : U_i \rightarrow \varphi_i(U_i) \subset \mathbb{R}^n$ is a homeomorphism such that $N |_{U_i}$ is mapped to the set of compact Lebesgue null sets contained in $\varphi_i(U_i)$, which, as an open subset of $\mathbb{R}^n$, is a Hausdorff space equipped with the Radon valuation given by Lebesgue measure. The statement that $M_{Leb}$ is a measurable locale is given by Proposition \ref{measurableoutofmeasurablymarked}.
\end{proof}

\begin{lemma} \label{manifolddensity}
The canonical map $\mathrm{can} : M_{Leb} \rightarrow M$ is dense.
\end{lemma}

\begin{proof}
We need to show that $\mathrm{can}_*(0) = \emptyset$ for the composite 
$$\mathrm{can} : b(N) \xrightarrow{i} M^\mathcal{K} \xrightarrow{\theta} M.$$
By definition $i_*(0) = N$. Analogous to the proof given for Proposition \ref{injectivity} we have
$$\theta_*(N) = \bigvee_{ U ~:~ \theta^*(U) \leq N } U = \bigvee \{ U ~|~ \text{ for all } K \subset U \text{ compact } K \text{ is null} \}.$$
Now let $U$ be a non-empty open subset of $M$. By choosing a non-empty chart $(V, \varphi)$ contained in $U$, we can find a compact set $K \subset V \subset U$ such that the Lebesgue measure of $\varphi(K) \subset \varphi(V)$ is not zero. Therefore $U$ is not contained in $\theta_*(N)$, in other words, we must have $\mathrm{can}_*(0) = \theta_*(N) = \emptyset$.
\end{proof}

We now address the question of functoriality of the assignment $M \mapsto M_{Leb}$. Recall that a \emph{submersion} $f : M \rightarrow N$ is a smooth map between (potentially large) smooth manifolds such that the differential $df_p$ has full rank at each point $p \in M$. (Compare \cite[Chapter 4]{Lee2003}.) We require the following standard result about submersions, namely that they are locally projections.

\begin{theorem}[\cite{Lee2003} Theorem 4.12] Suppose $M$ and $N$ are two smooth manifolds of dimension $m$ and $n$, respectively, with $m \geq n$, and $f : M \rightarrow N$ a smooth submersion. Then for each $p \in M$ there exists smooth charts $(U,\varphi)$ for $M$ containing $p$ and $(V,\psi)$ for $N$ containing $f(p)$ such that $f(U) \subset V$ and $\psi \circ f \circ \varphi^{-1} : \varphi(U) \rightarrow \psi(V)$ agrees with the projection $\pi : \mathbb{R}^m \rightarrow \mathbb{R}^n$ onto the first $n$ factors.
\end{theorem}

The following theorem was sketched by Dmitri Pavlov in a MathOverflow answer \cite{20820}, but, to the best of the author’s knowledge, it has not yet appeared in the literature.

\begin{theorem} \label{lebesguelocale}
There exists a functor
$$\begin{array}{rcl}
(-)_{Leb} : \mathrm{Man}_{\mathrm{subm}} & \rightarrow &  \mathrm{MblLoc} \\
 M & \mapsto & M_{Leb}
\end{array}$$
where $\mathrm{Man}_{\mathrm{subm}}$ is the category of large smooth manifolds and smooth submersions as maps between them.
\end{theorem}

\begin{proof}
The proof of Theorem \ref{lebesguelocale} reduces entirely to Theorem \ref{functormeasurablymarked}. The only crucial point is to verify that a given smooth submersion $f : M \rightarrow N$ is compatible with the markings given by compact null sets. To do so, we need to verify two statements.
\begin{itemize}
\item The image of a compact null set under $f$ remains a compact null set. This only needs smoothness and is provided by Lemma \ref{imageofnullundersmooth}.
\item Let $C \subset N$ be a compact null set, and $K \subset f^{-1}(C) \subset N$ a compact set. We need to argue that $K$ is a compact null set. Using compactness of $K$, cover $K$ by finitely many charts $(U_i, \varphi_i), i \in I,$ together with charts $(V_i, \psi_i), i \in I,$ such that $f$ maps $U_i$ into $V_i$ and such that $\psi_i \circ f \circ \varphi_i^{-1}$ agrees with the standard projection $\pi : \mathbb{R}^m \rightarrow \mathbb{R}^n$ onto the first $n$ factors. Choose compact sets $K_i \subset U_i$ such that $K = \bigcup_{i \in I} K_i$. Then $f(K_i) \subset V_i$ is a compact null set, since it is contained in the null set $C$, and hence $\psi_i(f(K_i)) \subset \psi_i(V_i)$ is a Lebesgue null set. We can now use Lemma \ref{projectioninverseimage} to conclude that $\varphi_i(K_i) \subset \pi^{-1}(\psi_i(f(K_i)))$ is a Lebesgue null set, in other words $K_i$ are null sets for all $i \in I$ and so is $K$.
\end{itemize}
\end{proof}

\begin{remark}
The usual condition of second countability that is typically required in standard textbooks on differential geometry is not necessary for Theorem \ref{lebesguelocale}. This means any large smooth manifold $M$ has an associated measurable locale $M_{Leb}$ and under Gelfand duality \ref{gelfanddualityvonneumann} a corresponding commutative von Neumann algebra $L^\infty(M)$, and with it all the standard theorems one usually requires for integration. The set $C_b(M)$ of bounded continuous functions naturally embeds into $L^\infty(M)$ by Lemma \ref{manifolddensity}. This opens the potential for a working integration theory even on non-second countable manifolds. Of course, second countability will appear naturally, e.g.\ any choice of a finite measure on $M$ has second countable support.
\end{remark}

\begin{remark}
In the interpretation of Boolean locales as generalized discrete spaces, we can think of the natural map $M_{Leb} \rightarrow M$ as equipping $M$ with a \emph{different underlying space}. This underlying space necessarily has no points for $n$-dimensional manifolds with $n > 0$, but may be more fundamental for physical considerations than the ordinary set of points of $M$. This idea was already present in the writing of von Neumann \cite[254]{von1998continuous}, who suggests that $M_{Leb}$ ``may be considered as the basic system for the logical treatment of classical mechanics''.
\end{remark}

\begin{remark}
The statements made in this section could be made more generally for $C^1$-manifolds and continuously differentiable maps.
\end{remark}

\section{Comparison with classical measure spaces} \label{classicalcomparison}

In this section, we aim to connect our notion of valuation sites with the classical notion of a measure space. The central example is the following: If 
$(X, \mathcal{L}, \mu)$ is a measure space, define
$$ \mathcal{L}^{fin} = \{ E \in \mathcal{L} ~|~ \mu(E) < \infty \}.$$
It is clear that $(\mathcal{L}^{fin}, \mu)$ defines a valuation site. We claim that under suitable conditions, this site determines the quotient algebra $\mathcal{L} / \mathcal{N}$ where $\mathcal{N}$ is the ideal of null sets. This condition is \emph{localizability}, as introduced by Segal \cite{segal1951}.

\subsection{Localizable measure spaces} \label{localizablemeasurespace}

\begin{definition}
Let $(X, \mathcal{L}, \mu)$ be a measure space. Let $\mathcal{E} \subset \mathcal{L}$ be a collection of measurable sets. A measurable set $E$ is called \emph{essential supremum} of $\mathcal{E}$ if
\begin{itemize}
\item For all $E' \in \mathcal{E}$, we have $\mu( E' \setminus E ) = 0$.
\item Whenever $F$ is a measurable set such that $\mu( E' \setminus F ) = 0$ for all $E' \in \mathcal{E}$, then also $\mu( E \setminus F) = 0$. 
\end{itemize}
\end{definition}

\begin{remark}
We see more or less by definition that a set $E$ is an essential supremum of $\mathcal{E}$ iff the class $[E]$ is the supremum of the set of classes $[E'], E' \in \mathcal{E}$ in the quotient algebra $\mathcal{L} / \mathcal{N}$, where $\mathcal{N}$ is the ideal of null sets. In particular, essential suprema are unique up to null sets.
\end{remark}

\begin{lemma}[See also \cite{PAVLOV2022106884}, Proposition 4.41.] \label{essentialsupmeasure}
Let $(X, \mathcal{L}, \mu)$ be a measure space and suppose $E$ is the essential supremum of a directed set $\mathcal{E}$. Then
$$\mu(E) = \sup_{E' \in \mathcal{E}} \mu(E').$$
\end{lemma}

\begin{proof} We can assume w.l.o.g.\ that $\mathcal{E}$ is closed under finite unions. First note that we have for $E' \in \mathcal{E}$ by definition $\mu( E' \setminus E ) = 0$, therefore $\mu(E) \geq \mu(E')$, hence the lemma is trivially true if $\sup_{E' \in \mathcal{E}} \mu(E') = \infty.$ Therefore assume $\sup_{E' \in \mathcal{E}} \mu(E') < \infty$ for the rest of the proof.

Inductively find a sequence of elements $E_n \in \mathcal{E}$ such that $E_n \subset E_{n+1}$ and
$$\mu(E_n) + \frac{1}{n} \geq \sup_{E' \in \mathcal{E}} \mu(E').$$
Take the countable (!) union $F = \bigcup_{n \in \mathbb{N}} E_n$. It is clear that $\mu(F) = \lim_{n \rightarrow \infty} \mu(E_n) =  \sup_{E' \in \mathcal{E}} \mu(E') < \infty$.

Now take $G \in \mathcal{E}$.  We know that $\mu( G \cup E_n ) \leq \sup_{E' \in \mathcal{E}} \mu(E')$ for all $n \in \mathbb{N}$, since $G$ and $E_n$ are in $\mathcal{E}$ and we assumed $\mathcal{E}$ to be closed under finite unions. Therefore
$$\mu( G \cup E_n ) = \mu( G \setminus E_n) + \mu( E_n ) \xrightarrow{ n \rightarrow \infty } \mu(G \setminus F) + \sup_{E' \in \mathcal{E}} \mu(E') \leq \sup_{E' \in \mathcal{E}} \mu(E').$$
Hence we obtain that $\mu(G \setminus F) = 0$. Since $E$ is an essential supremum of $\mathcal{E}$, we conclude that $\mu( E \setminus F ) = 0$. Furthermore, we have
$$\mu( F \setminus E ) = \lim_{n \rightarrow \infty} \mu( E_n \setminus E ) = 0.$$ But this means
$$\mu(E) = \mu(E \cap F ) + \mu( E \setminus F ) =  \mu(E \cap F ) = \mu(E \cap F ) + \mu( F \setminus E  ) = \mu(F).$$ 
This concludes the proof.
\end{proof}

\begin{definition}[\cite{fremlin2000measure2}, Definition 211G] A measure space $(X, \mathcal{L}, \mu)$ is called  \emph{semi-finite} if for every $E \in \mathcal{L}$ such that $\mu(E) = \infty$, there exists $F \subset E$ such that $F \in \mathcal{L}$ and $0 < \mu(F) < \infty$. A semi-finite measure space $(X, \mathcal{L}, \mu)$ is called \emph{localizable} if arbitrary essential suprema exist.
\end{definition}

\begin{remark}
In particular, for a localizable measure space $(X, \mathcal{L}, \mu)$, the Boolean algebra $\mathcal{L}/\mathcal{N}$ is complete; in other words, a Boolean frame.
\end{remark}

\begin{lemma} \label{finiteapproximation}
Let $(X, \mathcal{L}, \mu)$ be a localizable measure space. Then every $E$ in $\mathcal{L}$ is the essential supremum of the set $\{ E' \in \mathcal{L} ~|~ E' \subset E, ~ \mu(E') < \infty \}$.
\end{lemma}

\begin{proof} Fix $E \in \mathcal{L}$ and suppose $F \in \mathcal{L}$ is such that $\mu(E' \setminus F ) = 0$ for all measurable $E' \subset E$ such that $\mu(E') < \infty$. We consider two cases:
\begin{itemize}
\item If $\mu(E \setminus F ) < \infty$, then $E' = E \setminus F \subset E$ satisfies $\mu(E') < \infty$, but by assumption $F$ is such that
$$ 0 = \mu( E' \setminus F ) = \mu( E' ) = \mu(E \setminus F ).$$
\item If $\mu(E \setminus F ) = \infty$, by semi-finiteness, we can choose $E' \subset E \setminus F \subset E$ such that $0 < \mu(E') < \infty$. But then by assumption on $F$ we have that
$$ 0 = \mu( E' \setminus F ) = \mu( E' ),
 $$
a contradiction.
\end{itemize}
Therefore we conclude $\mu(E \setminus F ) = 0$ as the only possibility.
\end{proof}

\begin{theorem}
Let $(X, \mathcal{L}, \mu)$ be a localizable measure space. Then we have a measure-preserving frame isomorphism
$$(\mathrm{Sh}(\mathcal{L}^{fin}, \mu;\mathbf{2}),\mu_*) \cong (\mathcal{L}/\mathcal{N}, \mu )$$
induced by the natural homomorphism $\mathcal{L}^{fin} \rightarrow \mathcal{L}/\mathcal{N}$.
\end{theorem}

\begin{proof}
It is clear from Lemma \ref{finiteapproximation} that the image of $\mathcal{L}^{fin}$ in $\mathcal{L}/\mathcal{N}$ gives a basis, hence by the Basis Theorem \ref{basistheorem}, we know that $\mathcal{L}/\mathcal{N}$ is obtained from the induced Grothendieck topology on $\mathcal{L}^{fin}/\mathcal{N}$, which furthermore induces a Grothendieck topology on $\mathcal{L}^{fin}$. We only need to check that this Grothendieck topology agrees with the $\mu$-inner topology on $\mathcal{L}^{fin}$. This is equivalent to saying that for $E \in \mathcal{L}$ with $\mu(E) < \infty$, we have that a set $\{ E_i \subset E ~|~ i \in I\}$ is a $\mu$-cover iff $E$ is the essential supremum of $\{ E_i ~|~ i \in I\}$.

First assume that $\{ E_i \subset E ~|~ i \in I\}$ is a $\mu$-cover. This means, for any $\epsilon > 0$ there exists $E_{i_1},\hdots E_{i_n}$ such that
$$\mu(E) - \mu( E_{i_1} \cup \hdots \cup E_{i_n} ) < \epsilon.$$
By definition $\mu( E_i \setminus E ) = \mu( \emptyset ) = 0$ for all $i \in I$. Now assume $F$ is any measurable set such that $\mu(E_i \setminus F) = 0$ for all $i \in I$. Take $\epsilon > 0$ and choose $E_{i_1},\hdots E_{i_n}$ as above. Then
$$ \mu( E \setminus F ) = \mu( E \setminus F ) - \mu( (E_{i_1} \cup \hdots \cup E_{i_n} ) \setminus F ) \leq \mu( E ) - \mu( E_{i_1} \cup \hdots \cup E_{i_n}  ) < \epsilon.$$
Since this holds for any $\epsilon > 0$, we conclude that $\mu( E \setminus F ) = 0$.

Conversely, assume that $E$ is the essential supremum of $\{ E_i ~|~ i \in I\}$. Since essential suprema of a finite collection of measurable sets is always realized by finite union, we are left to argue the case where the set $\{ E_i ~|~ i \in I\}$ is directed. In this case, Lemma \ref{essentialsupmeasure} shows that $\{ E_i \subset E ~|~ i \in I\}$ is in fact a $\mu$-approximation. That same lemma also shows that the inner valuation and $\mu$ agree.
\end{proof}

\subsection{Radon measures} \label{radonmeasurespaces}

While the notion of a localizable measure space connects classical measure theory with the site-theoretic approach, the previous section is still mostly interesting from a theoretical point of view. In this section we deal with a class of measures that covers most cases of applications of classical measure theory in analysis and other fields, namely so-called \emph{Radon measures}. In the case of a Radon measure $(X, \mathcal{B}, \mu)$, the underlying set $X$ is equipped with a Hausdorff topology, and the measure is asked to interact favourably with the topology. The standard reference for Radon measures is Schwartz \cite{schwartz1973radon}. The basic point of this section is that in the case of a Radon measure, for the construction of the measure algebra the information of the entire $\sigma$-algebra is somewhat redundant. We will show that the measure algebra can be constructed purely from the datum of the lower bounded lattice of compact sets $\mathcal{K}(X)$, together with its valuation $\mu$.

\begin{definition}[Compare \cite{fremlin2000measure4} 411H Definition, and \cite{BLECHER2022758} Definition 6.4.] Let $X$ be a Hausdorff topological space. A \emph{Radon measure} on $X$ is a $\sigma$-algebra $\mathcal{B}$ and a measure $\mu$ on $\mathcal{B}$ such that
\begin{itemize}
\item $(X, \mathcal{B}, \mu)$ is a complete measure space.
\item $\mathcal{B}$ contains all open sets.
\item If $E \subset X$ and $E \cap F \in \mathcal{B}$ for all $F \in \mathcal{B}$ with $\mu(F) < \infty$, then $E \in \mathcal{B}$.
\item For every $E \in \mathcal{B}$ we have
$$\mu(E) = \sup \{ \mu(C) ~|~ C \subset E, C \text{ compact} \}.$$
\item For every $x \in X$ there is an open neighborhood $U$ of $x$ with $\mu(U) < \infty$.
\end{itemize}
\end{definition}

Note that the last condition implies that $\mu(C) < \infty$ for any compact subset $C$, as $C$ can be covered by finitely many opens with finite measure. Observe that the restriction of $\mu$ to the set of compact subsets $\mathcal{K}(X)$ gives a regular content, and the measure on every element $E \in \mathcal{B}$ is determined by this regular content. In this sense, for the purposes of measure theory, the information contained in a regular content and that of a Radon measure can be considered equivalent.

Radon measures are abundant throughout measure and probability theory. The basic examples are the ones already given in Example \ref{examplesradon}: Euclidean space $\mathbb{R}^n$ equipped with its Lebesgue measure, more generally Haar measures on locally compact Hausdorff groups, and Riemannian manifolds together with their volume measures, to name a few.

\begin{theorem}[\cite{fremlin2000measure4} Theorem 415A] \label{radonlocalizable}
A measure space $(X, \mathcal{B}, \mu)$ with $\mu$ a Radon measure is localizable.
\end{theorem}

\begin{lemma} \label{radonessentialsup}
Let $(X, \mathcal{B}, \mu)$ be a Radon measure on a Hausdorff space $X$. Let $E \in \mathcal{B}$. Then $E$ is the essential supremum of the set
$$\{ C ~|~ C \subset E, C \text{ compact} \}.$$
\end{lemma}

\begin{proof}
Let $E \in \mathcal{B}$. It is clear that $\mu( C \setminus E ) = \mu( \emptyset ) = 0$ for all compact $C \subset E$. Now assume that $F \in \mathcal{B}$ is such that $\mu(C \setminus F) = 0$ for all compact $C \subset E$. Assume $\mu(E \setminus F) > 0$. Since $\mu$ is a Radon measure, there must exist a compact $K \subset E \setminus F \subset E$ with $\mu(K) > 0$. But then
$$0 < \mu(K) = \mu( K \setminus F) = 0,$$
a contradiction. Therefore, we must have $\mu(E \setminus F) = 0$ and thus $E$ satisfies the properties of the essential supremum of $\{ C ~|~ C \subset E, C \text{ compact} \}$.
\end{proof}

\begin{theorem} \label{radonmeasures}
Let $(X, \mathcal{B}, \mu)$ be a Radon measure on a Hausdorff space $X$. Then we have a measure-preserving frame isomorphism
$$(\mathrm{Sh}(\mathcal{K}(X), \mu; \mathbf{2}), \mu_*) \cong (\mathcal{B}/\mathcal{N}, \mu )$$
induced by the natural homomorphism $\mathcal{K}(X) \rightarrow \mathcal{B}/\mathcal{N}$, where $\mathcal{K}(X)$ is the lower bounded distributive lattice of compact subsets of $X$, and $\mathcal{N} \subset \mathcal{B}$ is the ideal of $\mu$-null sets.
\end{theorem}

\begin{proof}
Theorem \ref{radonlocalizable} guarantees that $\mathcal{B}/\mathcal{N}$ is a frame. Lemma \ref{radonessentialsup} furthermore shows that the image of $\mathcal{K}(X)$ is a basis. We are left to identify the induced Grothendieck topology on $\mathcal{K}(X)$ with the $\mu$-inner topology.

It is clear that coverings by finite unions agree with finite essential suprema. Now assume $\{K_i \subset K ~|~ i \in I \}$ is a directed $\mu$-approximation of compact sets. We want to show that $K$ is the essential supremum of the sets $K_i$. It is clear that $\mu( K_i \setminus K ) = \mu( \emptyset ) = 0.$  Now assume that $F \in \mathcal{B}$ is such that $\mu( K_i \setminus F) = 0$ for all $i \in I$. Let $\epsilon > 0$ and pick $i_0 \in I$ such that
$$\mu( K \setminus K_{i_0} ) = \mu(K) - \mu(K_{i_0}) < \epsilon.$$
Then
$$\mu( K \setminus F ) = \mu( K_i \setminus F ) + \mu( (K \setminus K_i ) \setminus F ) \leq \mu( K \setminus K_i ) < \epsilon.$$
Since this is true for any $\epsilon > 0$, we conclude that $\mu( K \setminus F ) = 0$. In other words, we checked that $K$ satisfies the condition to be an essential supremum.

Conversely, if $K$ is an essential supremum of a set of compact sets $\{K_i ~|~ i \in I \}$, then $\{K_i \subset K ~|~ i \in I \}$ is a $\mu$-approximation by Lemma \ref{essentialsupmeasure}. 

Finally, we also see that the inner measure on $\mathrm{Sh}(\mathcal{K}(X), \mu; \mathbf{2})$ agrees with $\mu$ by definition of a Radon measure.
\end{proof}


\begingroup
\setlength{\emergencystretch}{8em}
\printbibliography
\endgroup

\end{document}